\documentclass[11 pt]{amsart}
\usepackage{hyperref}
\usepackage{amsmath,amsthm,amsfonts,amssymb, amscd}
\usepackage[dvips]{graphicx}
\usepackage{verbatim}
\usepackage{scalerel}
\usepackage{mathrsfs}
\usepackage{setspace}
\usepackage{tikz-cd}
\usepackage[english]{babel}
\usepackage{thm-restate}
\usepackage{cleveref}
\usepackage{thmtools}

\hypersetup{
    linktocpage,
    colorlinks,
    linkcolor={magenta!100!black},
    citecolor={cyan!100!black},
    urlcolor={yellow!100!black}
}

\newcommand{\R}{\mathbb{R}}

\newcommand{\N}{\mathbb{N}}

\newcommand{\M}{\mathcal{M}}

\newcommand{\diag}{\text{diag}}

\newcommand{\norm}[1]{\left\lVert#1\right\rVert}

\DeclareMathOperator{\id}{id}

\DeclareMathOperator{\grad}{grad}
\DeclareMathOperator{\Lip}{Lip}
\DeclareMathOperator{\supp}{supp}

\newtheorem{MainThm}{Theorem}

\newtheorem{definition}{Definition}[section]
\newtheorem{lemma}[definition]{Lemma}
\newtheorem{proposition}[definition]{Proposition}
\newtheorem{theorem}[definition]{Theorem}
\newtheorem{corollary}[definition]{Corollary}

\newtheorem{question}[definition]{Question}

\newtheorem{example}[definition]{Example}

\def\vol{{\rm vol }}

\newtheoremstyle{boldremark}
  {3pt}   
  {3pt}   
  {\normalfont} 
  {}      
  {\bfseries} 
  {.}     
  {.5em}  
  {}      

\theoremstyle{boldremark}
\newtheorem{remark}[definition]{Remark}

\begin{document}
\title[Geodesic stretch]{Maximal stretch and Lipschitz maps on  Riemannian manifolds of negative curvature
}
\author{Xian Dai}
\address{Laboratoire Jean-Alexandre Dieudonné, Université Côte d'Azur, Nice 06200, France}
 \email{Xian.DAI@univ-cotedazur.fr}

\author{Gerhard Knieper}
\address{Faculty of Mathematics,
Ruhr University Bochum, Bochum 44780, Germany}
 \email{gerhard.knieper@rub.de}
\date{}
\maketitle

\begin{abstract}
In his seminal work on Teichm\"uller spaces (\cite{Th98}), Thur\-ston introduced the maximal stretch for a pair of hyperbolic metrics on a closed surface of genus $\mathcal{G}\geq 2$ and showed that the logarithm of this quantity induces an asymmetric metric in the Teichmüller space. 
He also showed that the subset of the surface on which the maximal stretch is attained is a geodesic lamination. In this paper, we define the maximal stretch analogously for closed manifolds equipped with Riemannian metrics of variable negative curvature and investigate the structure of the related Mather set on the unit tangent bundle. In contrast to the Teichm\"uller space, the Mather set may not be the lifts of geodesic laminations in this broader setting. However, in our paper, we will discuss similar features shared by the Mather set with geodesic laminations. We also connect the study of the Mather set with the theory of best Lipschitz maps. 
\end{abstract}

\setcounter{tocdepth}{1}
\tableofcontents
\section{Introduction}
Given a pair of hyperbolic metrics $g_1$ and $g_2$, i.e. metrics of constant curvature $-1$, on a closed surface $S$ of genus $\mathcal{G}\geq 2$, we associate to each free homotopy class $[\gamma]$ of closed curves on $S$ the quantity
$$r_{g_1,g_2}([\gamma]):=\frac{\ell_{g_2}([\gamma])}{\ell_{g_1}([\gamma])},$$
where $\ell_{g_i}([\gamma])$ denotes the length of the unique closed geodesic in the free homotopy class $[\gamma]$ with respect to the hyperbolic metric $g_i$ for $i=1,2$. 

Let $[\pi_1 S]$ denote the set of free homotopy classes of closed curves on $S$. In his seminal work \cite{Th98}, Thurston considered the supremum of $r_{g_1,g_2}([\gamma])$ over all $[\gamma]\in [\pi_1 S]$  as a quantity for comparing $g_1$ and $g_2$,
 $$S(g_1,g_2) := \sup_{[\gamma]\in[\pi_1 (S)] } \frac{\ell_{g_2}([\gamma])}{\ell_{g_1}([\gamma])}.$$

In this note, we refer to the quantity $S(g_1,g_2)$ as the \emph{maximal stretch} between $g_1$ and $g_2$. Thurston shows that $d_T(g_1,g_2) := \log S(g_1,g_2)$, defines an asymmetric metric on the Teichm\"uller space $\mathcal{T}(S)$ of the closed surface $S$. The metric $d_T$ is called \emph{Thurston's asymmetric metric}. Moreover, he showed that the Teichm\"uller space is a geodesic metric space with respect to $d_T$ induced by a non-reversible  Finsler metric. He further demonstrated that the maximal stretch
 $S(g_1,g_2)$ can always be approximated by a sequence $\{r_{g_1,g_2}([\gamma_n])\}_{n\in \mathbb{N}}$, where each $[\gamma_n]$ is a free homotopy class that can be represented by simple closed curves. In particular, the supremum $S(g_1,g_2)$ is attained by a \emph{geodesic lamination} --- a closed subset of the hyperbolic surface that consists of a disjoint union of simple complete geodesics, which may be closed or bi-infinite. 

 The study presented in this note is motivated by the following questions: Suppose we replace the surface $S$ by a closed manifold $M$ of dimension $n\geq 2$ which admits a smooth negatively curved Riemannian metric, and suppose that the space of hyperbolic metrics on $S$ is replaced by the infinite dimensional space of smooth Riemannian metrics of negative curvature on $M$, denoted by $R^-(M)$. Then in this broader setting,  what aspects of Thurston's theory 
 remain true?

 In \cite{GKL22} Section 5, the second author, together with Guillarmou and Lefeuvre have shown the following theorem which partially 
 generalizes the work of Thurston in the framework of general negatively curved manifolds.

 \begin{theorem}\label{thm,dT=dF}
    Let $ R^-_1(M)$ be the set of metrics in $R^-(M)$ with topological entropy $1$. Then 
    $d_T(g_1, g_2) = \log S(g_1,g_2)$ descends to an asymmetric  pseudo-metric on the isometry classes of  $R^-_1(M)$. It is an
    asymmetric metric in a sufficiently small $C^k$-neighborhood of the diagonal, where $k$ depends only on the dimension of $M$. Furthermore, $d_T$ induces an asymmetric Finsler norm whose associated Finsler metric $d_F$ dominates $d_T$.
 \end{theorem}
 
 \begin{remark}
 We recall that hyperbolic surfaces always have topological entropy one. We make the following observations related to Theorem \ref{thm,dT=dF}.
 \begin{itemize}
     \item 
The function $d_T$ is a
 metric on the isometry classes of  $R^-_1(M)$ if and only if marked length spectrum rigidity holds, that is, if $g_1,g_2 \in R^-_1(M)$ satisfy $\ell_{g_1}([\gamma]) =\ell_{g_2}([\gamma])$ for all $ [\gamma]\in [\pi_1 M]$, then  $g_1$ and $g_2$ are isometric.
 In particular, when $M$ is a surface, from the marked length spectrum rigidity theorem proved by Croke \cite{Cr90} and independently Otal \cite{Ot90} for negatively curved surfaces, $d_T$ defines a metric on $ R^-_1(M)$. 
 When $\dim M>2$, the fact that $d_T$ is  also a metric in a small neighborhood of the diagonal in  $ R^-_1(M)$ follows from the local marked length spectrum rigidity theorem proved by Guillarmou and Lefeuvre (see \cite{GL19}). 
 \item 
The asymmetric Finsler metric $d_F$ in Theorem \ref{thm,dT=dF} is defined without assuming marked length spectrum rigidity.  
While $d_F \ge d_T$, we do not know, unlike the case of Teichm\"uller space, 
    whether the equality $d_T= d_F$ holds on $R_1^-(M)$. As observed in \cite{GL19}, the equality would imply the (global) marked length spectrum rigidity (see also Question \ref{question: Finsler}
    in Appendix \ref{question: open question}).
     \end{itemize}
 \end{remark}
The purpose of this paper is to study the other parts of Thurston's theory, namely the structure of \emph{maximally stretched measures} and its relation to geodesic laminations. We explain the notion of maximally stretched measures. Fix a family of closed $g_1$-geodesics $\gamma^{g_1}_n$ so that the length ratios $r_{g_1,g_2}([\gamma_n])$ approximates the maximal stretch $S(g_1,g_2)$ and consider a sequence of closed orbit probability measures $\delta^{g_1}_{v_n}$ with initial unit tangent vector $v_n$ tangential to $\gamma^{g_1}_n$. Then up to subsequence, these Dirac measures $\delta^{g_1}_{v_n}$ converge weakly to some $\phi^{g_1}$-invariant probability measure $m$, where $\phi^{g_1}$ is the geodesic flow with respect to the metric $g_1$. We call such a measure $m$ a \emph{maximally stretched measure}. All maximally stretched measures form a subset $MS(g_1,g_2)$ in the space of $\phi^{g_1}$-invariant probability measures. The closure of the union of their supports 
 $$\mathcal{M}(g_1,g_2) :=\overline{ \bigcup_{m\in MS(g_1,g_2)} \supp m}$$
is called the \emph{Mather set}. The Mather set will be one of the central object of study in this note.

 \subsection{Topological structure of the Mather set}
 \hfill\\
When $g_1$ and $g_2$ are hyperbolic metrics representing distinct points on the Teichm\"uller space $\mathcal{T}(S)$, the Mather set $\mathcal{M}(g_1,g_2)$ always projects to a \emph{geodesic lamination} on $S$ (\cite[Section 3]{Th98}).  A natural question arises: in general, for $g_1,g_2\in R^-(M)$, does the Mather set $\mathcal{M}(g_1,g_2)$ exhibit structural similarities with geodesic laminations?

One well-known property of geodesic laminations is that they are nowhere dense. Furthermore, a surprising result  (\cite[Theorem I]{BS85}) from Birman and Series shows that the union of all simple geodesics forms a nowhere dense set of the hyperbolic surface and has Hausdorff dimension one. The same result holds for the lift of this set to the unit tangent bundle of the hyperbolic surface (\cite[Theorem III]{BS85}). 

Along these lines, we prove the following similar property for the Mather set.

\begin{MainThm} [Theorem \ref{thm, EmtpyInterior}]
Given $g_1,g_2\in R^-(M)$, the Mather set $\mathcal{M}(g_1,g_2)$ is nowhere dense if and only if the marked length spectra of $g_1$ and $g_2$ are not proportional,  i.e. there does not exist a constant $C>0$ such that $\ell_{g_2}([\gamma])=C \ell_{g_1}([\gamma])$ for all $[\gamma]\in[\Gamma]$. 
\end{MainThm}

In fact, if the marked length spectra of $g_1$ and $g_2$ are proportional, then the associated Mather set $\mathcal{M}(g_1,g_2)$ coincides with the entire unit tangent bundle $S^{g_1}M$. The theorem therefore exhibits a dichotomy in the structure of $\mathcal{M}(g_1,g_2)$, according to whether the marked length spectra of $g_1$ and $g_2$ are proportional. In general, the condition that two metrics have proportional marked length spectra is quite restrictive. For example, if $g_1$ and $g_2$ are conformally equivalent, then they have proportional marked length spectra if and only if $g_2$ is a constant multiple of $g_1$ (see \cite[Theorem 2]{Ka88}).

\subsection{Measures of maximal entropy on the Mather set}
 \hfill\\
A different facet of the study of the Mather set which connects the marked length spectra involves its thermodynamic properties. Section \ref{section,thermodynamics} is devoted to this topic. Using the thermodynamic formalism, we investigate measures of maximal entropy for the geodesic flow $\phi^{g_1}$ restricted to the Mather set $\mathcal{M}(g_1,g_2)$. In particular, we prove the existence of such measures and construct examples of them as weak limits of some equilibrium states (see Section \ref{subsection,MMEMatherSet}). These limiting measures are often referred to as \emph{zero-temperature limits}.

\begin{MainThm}[Theorem \ref{prop,limitIsTopEntropy}] \label{theorem,MeasureMaxEntropy}
     There exists a $\phi^{g_1}$-invariant probability measure $m_+$ arising as weak limits of some equilibrium states which is a measure of maximal entropy of the geodesic flow $\phi^{g_1}$ restricted to the Mather set $\M(g_1,g_2)$, that is,
$$h_{m_{+}} (\phi^{g_1}) = h_{top}(\phi^{g_1}, \mathcal{M}(g_1,g_2)),$$
    where $h_{top}(\phi^{g_1}, \mathcal{M}(g_1,g_2))$ is the topological entropy of $\phi^{g_1}$ on the closed invariant subset $\mathcal{M}(g_1,g_2)$ and $h_{m_{+}} (\phi^{g_1})$ denotes the metric entropy of $m_+$.
\end{MainThm}

Moreover, we show that the value $h_{top}(\phi^{g_1}, \mathcal{M}(g_1,g_2))$ reflects whether the marked length spectra of $g_1$ and $g_2$ are proportional.

\begin{MainThm}[Theorem \ref{cor,RigidityTopEntropy}]
 Given $g_1,g_2\in R^-(M)$, the topological entropy on the Mather set $\M(g_1,g_2)$ satisfies
 $$h_{top}(\phi^{g_1}, \mathcal{M}(g_1,g_2))=h_{top}(\phi^{g_1})$$
if and only if the marked length spectra of $g_1$ and $g_2$ are proportional, i.e. there exists a constant $C>0$ such that $\ell_{g_2}([\gamma])= C \ell_{g_1}([\gamma])$ for all $[\gamma]\in[\Gamma]$. 
\end{MainThm}  

 When the marked length spectra of $g_1$ and $g_2$ are not proportional, the metric entropies of the equilibrium states used for constructions in Theorem \ref{theorem,MeasureMaxEntropy} strictly monotonically decrease to the metric entropy of $m_+$, which is the topological entropy $h_{top}(\phi^{g_1}, \mathcal{M}(g_1,g_2))$  (see Corollary \ref{cor,dichotomyEntropy}). 
 
 When $g_1,g_2$ are hyperbolic metrics on a closed surface representing distinct points in the Teichmüller space $\mathcal{T}(S)$, a maximally stretched measure can be given by a \emph{measured lamination}. A \emph{measured lamination}\footnote{This definition differs from the more standard presentation commonly used in the Teich\-müller theory community (see, for example, \cite{Bon88}). But it serves the topics of this note better.} on $(S, g_1)$ is a geodesic lamination that possesses a $\phi^{g_1}$-invariant reflexive \footnote{A measure $m$ on $S^{g_1}M$ is reflexive if $\iota^*m=m$, where $\iota:S^{g_1}M \to S^{g_1}M$ is the involution given by $\iota(v)=-v$. } probability measure: the projection of the full support of this invariant measure onto $S$ is the underlying geodesic lamination. We will often drop the distinction in notation between a lamination with measure and the measure itself. 
 
 Measured geodesic laminations have zero metric entropies (see \cite[Lemma 2.3.5]{Na03}). As a consequence, the topological entropy $h_{top}(\phi^{g_1}, \mathcal{M}(g_1,g_2))$ is zero in this case. One might then naively expect, for all metrics $g_1,g_2$ in $R^-(M)$ with non-proportional marked length spectra, the topological entropy $h_{top}(\phi^{g_1}, \mathcal{M}(g_1,g_2))$ also equals zero. This is however not always the case. Later, in Section \ref{section,BLM}, we give an example of $g_1,g_2$ in $R^-(M)$ with non-proportional marked length spectra and maximally stretched measures supported on the Mather set $\mathcal{M}(g_1,g_2)$ with positive metric entropy (see Example \ref{MatherNotLamination}).

\subsection{ The Mather set is maximally stretched}
 \hfill\\
In the previous section, we mentioned that when $g_1,g_2$ are hyperbolic metrics representing distinct points in the Teichmüller space $\mathcal{T}(S)$, a maximally stretched measure in $MS(g_1,g_2)$ can be taken as a measured lamination. One feature of measured laminations (and also geodesic laminations) is that they are topological objects by nature, regardless of the auxiliary hyperbolic metric chosen for their definitions. We specifically write $\lambda_g$ to represent the realization of a measured lamination $\lambda$ with respect to a hyperbolic metric $g$.

In the Teichm\"uller space $\mathcal{T}(S)$, the maximally stretched measures which are measured geodesic laminations have a deep connection with Lipschitz maps. Thurston proved, for hyperbolic metrics $g_1$ and $g_2$ representing distinct points in $\mathcal{T}(S)$ and for any measured lamination $\lambda$ in $MS(g_1,g_2)$, there exists a Lipschitz homeomorphism  $ f \colon (S,g_1) \to (S,g_2)$ homotopic to the identity with the property that each leaf of the measured lamination $\lambda_{g_1}$ is mapped by $f$ to a corresponding leaf of $\lambda_{g_2}$ and each leaf is linearly stretched by a maximal factor --- the number equals both the Lipschitz constant of $f$ and the maximal stretch $S(g_1,g_2)$ (\cite[Theorem 8.1, Theorem 8.5]{Th98}, see also \cite{GK17}, \cite{PW22}). This property helps reveal a profound equality in the Teichm\"uller theory, which will be discussed in Section \ref{subsection weightedBestLip}.

 In the setting of general negatively curved metrics on the $n$-dimensional closed manifold $M$, we would like to understand whether orbits in the support of a maximally stretched measure are also in some sense maximally stretched. Since the theory of Lipschitz maps in this broader setting is still quite obscure, we defer the demonstration of some partial results in this direction until Subsection \ref{subsection weightedBestLip} and Section  \ref{subsection StrechLocus}. Here we first exhibit an analogous and better understood phenomenon on the unit tangent bundles and for H\"older orbit equivalences between geodesic flows. We show that orbits of the geodesic flow $\phi^{g_1}$ are ``maximally stretched" by some ``good" H\"older orbit equivalences on the Mather set. 

 Denote by $S^g M$ the unit tangent bundle of $M$ with respect to a metric $g$.
  
\begin{MainThm}[Theorem \ref{corollary, LinearConjugacy}, Proposition \ref{prop, AubryMather}]\label{theorem, maximalConjugacy}
Suppose $g_1,g_2\in R^-(M)$. Then there exists a H\"older orbit equivalence $G \colon S^{g_1}M \to S^{g_2}M $
 between the geodesic flows  $\phi^{g_1}$ and $\phi^{g_2}$ given by
$$
\phi^{g_2}_{\tau(v,t)} G(v) = G(\phi_t^{g_1}(v)),
$$
where $\tau: S^{g_1}M \times \mathbb{R} \to \R$ is a time change function that satisfies, for all $v\in S^{g_1}M$,
$$\tau(v,t)\leq  S(g_1,g_2)t.$$
Moreover, if $v\in \mathcal{M}(g_1,g_2)$, then for all $t\in \R$,
$$\tau(v,t) =S(g_1,g_2)t.$$
\end{MainThm}

In fact, the above equality $\tau(v,t) =S(g_1,g_2)t$ in Theorem \ref{theorem, maximalConjugacy} holds for a potentially larger set than the Mather set $\mathcal{M}(g_1,g_2)$. This set is called the \emph{Aubry set} $\mathcal{A}(g_1,g_2)$, motivated by Fathi's monograph on weak KAM theory (see \cite{Fat08}), and is closely related to the study of the Mather set. The Aubry set is defined using \emph{(weak) supersolutions} --- functions on the unit tangent bundle $S^{g_1}M$ that satisfy certain inequalities involving the maximal stretch $S(g_1,g_2)$. These inequalities become equalities along ``extremal orbits", which are orbits of the Aubry set (of some supersolution).  We refer the reader to Section \ref{subsection,SupersolutionAubry} for precise definitions and details of the Aubry set. We highlight in the end that the Mather set, which characterizes extremal orbits purely in a measure-theoretic sense, is always contained in the Aubry set.

 \subsection{Weighted least Lipschitz constants} \label{subsection weightedBestLip}
 \hfill\\
As mentioned before, the study of Mather sets and maximal stretches involves fundamentally  the theory of Lipschitz maps in the classical Teichm\"uller space. A natural question, motivated from Theorem \ref{theorem, maximalConjugacy}, is the following: in the setting of general negatively curved metrics on the $n$-dimensional closed manifold $M$, when restricting to the Mather set $\mathcal{M}(g_1,g_2)$, can the H\"older orbit equivalences described in Theorem \ref{theorem, maximalConjugacy} on unit tangent bundles be induced from some Lipschitz maps defined on the base manifold $M$? 

In the Teichm\"uller space $\mathcal{T}(S)$, the Lipschitz maps that maximally stretch measured geodesic laminations in the Mather set are \emph{best Lipschitz maps} between hyperbolic surfaces. Given $g_1,g_2\in R^-(M)$, a \emph{best Lipschitz map} is a Lipschitz map from $(M,g_1)$ to $(M,g_2)$ whose Lipschitz constant equals the \emph{least Lipschitz constant} $L(g_1,g_2)$, which is defined as the infimum of Lip\-schitz constants $\Lip (f, g_1,g_2)$ among all Lipschitz maps $f \colon(M,g_1)\to (M,g_2)$ homotopic to the identity. Best Lipschitz maps between $(M,g_1)$ and $(M,g_2)$ always exist for compactness reasons.

  When $g_1,g_2$ are hyperbolic metrics on a closed surface $S$, the maximal stretch $S(g_1,g_2)$ exhibits surprising relations with the least Lipschitz constant $L(g_1,g_2)$. A fundamental result of Thurston \cite{Th98} shows that the equality $S(g_1,g_2)= L(g_1,g_2)$ always holds.  Moreover, as already partially mentioned in the last subsection, any best Lipschitz map $f \colon  (S,g_1)\to (S,g_2)$, homotopic to the identity, always maximally stretches every measured geodesic lamination $\lambda$ in the Mather set, in the sense that each leaf of $\lambda_{g_1}$ is mapped by $f$ to a corresponding leaf of $\lambda_{g_2}$ with each leaf being linearly stretched by $L(g_1,g_2)$.

In the setting of general negatively curved metrics on $M$, it is straightforward to see that $S(g_1,g_2) \leq L(g_1,g_2)$ for $g_1,g_2\in R^-(M)$ (see Proposition \ref{prop,LgeqS}). However, equality may fail even when $M$ is a surface: There exist examples of variable negatively curved metrics $g_1$ and $g_2$ where the strict inequality $S(g_1,g_2)< L(g_1,g_2)$ holds (see Appendix \ref{ExampleS<L}).

As a more flexible alternative to the best Lipschitz constant of Lipschitz maps, we introduce a new constant that takes into account both Lipschitz maps and invariant measures. Specifically, for a Lipschitz map $f$, we consider its average with respect to a  $\phi^{g_1}$-invariant probability measure $m$ as follows,
  \begin{align*}
L_m(f) := \int_{S^{g_1}M} \norm{Df(v)}_{g_2} dm(v).
 \end{align*}

Conceptually, this integral is the ``local Lipschitz constant" of $f$ weighted by the $\phi^{g_1}$-invariant probability measure $m$.  We  define the \emph{ $m$-weighted least Lipschitz constant } $L_m(g_1,g_2)$ as $$L_m(g_1,g_2)=\inf\limits_{f\in \Lip_{\id}(M, g_1, g_2)}\int_{S^{g_1}M} \norm{Df(v)}_{g_2} dm(v),$$
 where $\Lip_{\id}(M, g_1, g_2)$ denotes the space of all Lipschitz maps from $(M,g_1)$ to $(M,g_2)$ that are homotopic to the identity. By definition, it is clear that $L_m(g_1,g_2)\leq L(g_1,g_2)$ regardless of the choice of invariant measures.

 We have the following characterization of $L_m(g_1,g_2)$.

\begin{MainThm}[Corollary \ref{cor,MainThm3}] \label{mainThm3}
 Suppose $g_1,g_2 \in R^-(M)$. For any maximally stretched measure, we have
\begin{align*} 
S(g_1, g_2) \leq L_m(g_1,g_2).
\end{align*}
Furthermore, if for some maximally stretched measure $m_0$, there exists  $f_0 \in \Lip_{\id}(M, g_1, g_2)$  such that $$S(g_1,g_2) =\int_{S^{g_1}M} \norm{Df_0(v)}_{g_2} dm_0(v)=L_{m_0}(g_1,g_2) ,$$
then  $f_0$ maps any $g_1$-geodesic in the support of $m_0$ to a corresponding $g_2$-geodesic (up to parametrization).
\end{MainThm}

The Lipschitz map $f_0$ in Theorem \ref{mainThm3} is called a \emph{$m_0$-weighted best Lip\-schitz map} with respect to the maximally stretched measure $m_0$. It shares the following common feature with the usual best Lipschitz maps in the Teichm\"uller space $\mathcal{T}(S)$: it maps geodesics in the support of a maximally stretched measure to geodesics in the target space.

 For any pair of metrics $g_1,g_2\in R^-(M)$ the functional  $f \mapsto L_m(f)$ is lower-semicontinuous and convex (see 
 Proposition \ref{prop,PropertiesG3} and Proposition \ref{prop,PropertiesG}). Furthermore, if we can restrict  $L_m$  to Lipschitz maps with a uniformly bounded Lipschitz constant (see Proposition \ref{prop,infAchieved}), then the infimum $L_m(g_1,g_2)$ is achieved by some Lipschitz maps.  However, in general, we do not know whether the infimum can be achieved by a Lipschitz map. Moreover, it remains unclear to us whether there always exists a maximally stretched measure $m_0$ for which the equality
 $S(g_1, g_2) = L_{m_0}(g_1,g_2)$ holds (see also Question \ref{question: WeightedLipMap} in Appendix \ref{question: open question}).

\subsection{The Mather set and the stretch locus}
\hfill\\
Weighted best Lipschitz maps also provide new tools for understanding the Mather set and the \emph{stretch locus}. In \cite{GK17}, Gu\'{e}ritaud and Kassel introduce the notion of \emph{stretch locus}, originally in the context of hyperbolic $n$-spaces $\mathbb{H}^n$. For our purpose, we explain this concept in the context of general negatively curved metrics: Given $g_1,g_2\in R^-(M)$, the stretch locus is a subset of $M$ that is maximally stretched by every best Lipschitz map from $(M,g_1)$ to $(M,g_2)$. More precisely, given a best Lipschitz map $f \colon  (M, g_1) \to (M,g_2)$, its stretch locus $E_f(g_1,g_2)$ is the set of points $x$ such that the restriction of $f$ to any neighborhood of $x$ achieves the Lipschitz constant $L(g_1,g_2)$. The \emph{stretch locus} $E(g_1,g_2)$ is then defined as the intersection of the stretch loci of all best Lipschitz maps (see Definition \ref{defn, stretchLocus}). 

In the setting of constant negatively curved (not necessarily compact) manifolds of dimension $\geq 2$, Gu\'{e}ritaud and Kassel proved in \cite{GK17}, when $L(g_1,g_2)>1$, the equality $L(g_1,g_2)=S(g_1,g_2)$ still holds by showing that the stretch locus $E(g_1,g_2)$ forms a maximally stretched geodesic lamination (see Theorem \ref{GK, geodesicLamination} for details). This generalizes Thurston's results in \cite{Th98} to higher dimensions in the constant negative curvature setting. Motivated by their work and as a corollary of Theorem \ref{mainThm3}, we characterize a relation between the Mather set $\mathcal{M}(g_1,g_2)$ and the stretch locus $E(g_1,g_2)$ in the setting of $g_1,g_2\in R^-(M)$ and when the maximal stretch equals the least Lipschitz constant, i.e. $S(g_1,g_2)=L(g_1,g_2)$, where $M$ is assumed to be closed and of dimension $n\geq 2$ as usual.

\begin{MainThm}[see Proposition \ref{proposition,MatherStretchLocus}] \label{thm,ProjMatherStretchLocus}
   Suppose for $g_1,g_2\in R^{-} (M)$, we have
   $$S(g_1,g_2)= L(g_1,g_2).$$
Then the projection of the Mather set $\mathcal{M}(g_1,g_2)$ onto $M$ is contained in the stretch locus from $g_1$ to $g_2$, that is,
     $$\pi(\mathcal{M}(g_1,g_2)) \subset E(g_1,g_2),$$
     where $\pi \colon S^{g_1}M \to M$ is the projection map.
\end{MainThm}

Generally speaking, the projection of the Mather set $\pi(\mathcal{M}(g_1,g_2))$ represents the largest measured part of the stretch locus $E(g_1,g_2)$ when $S(g_1,g_2)= L(g_1,g_2)$. In the Teichm\"uller space $\mathcal{T}(S)$, the equality $S(g_1,g_2)= L(g_1,g_2)$ always holds and the phenomenon described in Theorem \ref{thm,ProjMatherStretchLocus} is well understood --- thanks to the fact that, in this setting, stretch loci are always maximally stretched geodesic laminations (\cite[Section 1]{Th98}, \cite[Section 5]{GK17}). However, for general $g_1,g_2\in R^-(M)$, the situation is more mysterious: on the one hand, we do not yet know how large  the subset of pairs of metrics in $R^-(M)$ is for which the equality $S(g_1,g_2)= L(g_1,g_2)$ holds, except some partially understood examples provided in the next subsection. On the other hand, the condition $S(g_1,g_2)= L(g_1,g_2)$ does not imply (see Example \ref{MatherNotLamination}) that the stretch locus $E(g_1,g_2)$ is a lamination, even though it contains the projection of the Mather set. This suggests new complexity in the structure of stretch loci beyond the classical setting.

 \subsection{Curvature bounds which imply that the stretch locus is a geodesic lamination.} 
 \label{subsection StrechLocus}
\hfill\\
From the picture of the Teichm\"uller Theory, it may be appealing at the beginning to guess that the Mather set $\mathcal{M}(g_1,g_2)$  always projects to some geodesic laminations for general negatively curved metrics $g_1,g_2$. However as already mentioned in previous subsections, this is indeed not the case. Even when $M$ is a closed surface, and two metrics $g_1,g_2$ are chosen to be conformal and ``near"  hyperbolic metrics, the Mather set $\mathcal{M}(g_1,g_2)$ can contain countably many closed orbits with self-intersections (see Example \ref{MatherNotLamination}).

Nevertheless, the methods of \cite{GK17} generalize partially to $R^-(M)$ and give many examples in the variable negative curvature metric setting such that the stretch loci $E(g_1,g_2)$ and projection of the Mather sets $\pi(\mathcal{M}(g_1,g_2))$ are geodesic laminations. Their methods yield the following.

For a smooth Riemannian metric $g$ on $M$, denote the minimal upper bound and the maximal lower bound of its sectional curvature as $ K^+_g $ and $K^-_g$, respectively.

\begin{MainThm}[Theorem \ref{GK, geodesicLamination}, and  \cite{GK17}] \label{thm,GKIntro} 
Suppose $g_1,g_2 \in R^-(M)$ satisfies 
   \begin{equation}\label{eqtn,LaminationConditions}
        0 < \frac{K^-_{g_1}}{K^+_{g_2}}<L(g_1,g_2)^2, 
        \end{equation}  
    Then 
    $$S(g_1,g_2) = L(g_1,g_2).$$
    Moreover, $E(g_1,g_2)$ is a geodesic lamination.
\end{MainThm}

In particular, their conditions for $E(g_1,g_2)$ (and hence $\pi(\mathcal{M}(g_1,g_2))$) to be a geodesic lamination is an open condition in $R^-(M)\times R^-(M)$ (see Proposition \ref{prop,LaminationConditions}). Therefore, when $M=S$ is a closed surface, one can find many examples in $R^-(S)$ for which the stretch locus $E(g_1,g_2)$ is a geodesic lamination. They are obtained by perturbing hyperbolic metrics away from the Teichm\"uller space $\mathcal{T}(S)$ (see Remark \ref{remark,LaminationSurvive}). On the other hand, we also know that the conditions provided by \cite{GK17} are not necessary conditions for $E(g_1,g_2)$ to be a geodesic lamination (see examples in Appendix \ref{appendix, C}). It remains an interesting question to understand in this general setting, what are necessary and sufficient conditions for the stretch locus $E(g_1,g_2)$ (and $\pi(\mathcal{M}(g_1,g_2))$) to be geodesic laminations.

\subsection{Other related works}
\hfill\\
 Recently, Georgios Daskalopoulos and Karen Uhlenbeck study in a sequence of papers  \cite{DU22}  \cite{DU24II} \cite{DU24} infinite harmonic maps as best Lipschitz maps in variable curvature metrics setting. They construct Lie algebra valued transverse measures as dual to best Lipschitz maps. It would be interesting to understand the relation between maximally stretched measures of this note and their work. 
 
 In addition to Gu\'{e}ritaud and Kassel's paper \cite{GK17} on best Lipschitz maps and maximal stretch in the setting of equivariant representations in $\mathbb{H}^n$, there has recently been many interests in higher rank symmetric spaces and higher Teichm\"uller theory. Thurston's asymmetric metric has been introduced in higher Teichm\"uller spaces \cite{CDPW24}. Our note might provide useful ideas in further study of these directions.

\subsection{Structure of the paper}
\begin{itemize}
    \item  In Section \ref{section, GeodesicStretch} we recall basic facts about the geodesic stretch which is a major tool of this note. We start by introducing the geodesic stretch with respect to invariant probability measures. Then, after introducing geodesic currents, we connect the theory of geodesic currents and intersections with the geodesic stretch.  We also recall in Section \ref{section, GeodesicStretch} orbit equivalence of geodesic flows on manifolds of negative curvature. 
    \item   In Section \ref{sec:MC}, we introduce a special geodesic stretch, called the maximal stretch. This leads to the study of Mather sets and Aubry sets in Section \ref{subsection, MCMather}.  In Section \ref{subsection, propMCMather} we collect properties of the Mather set that are important in this paper. Finally, in Subsection \ref{subsection AubrySetPeierlBarrier}, we introduce the Peierls barrier and its zero level set. We then explain their relations with the Aubry set and the Mather set.
    \item  Section \ref{section,thermodynamics} is devoted to the study of thermodynamic properties of the Mather set. We use equilibrium states theory to investigate measures of maximal entropy for the geodesic flow on the Mather set. 
     \item  Starting from Section \ref{section, StretchLip}, we turn our attention to Lipschitz maps. Section \ref{section, StretchLip} investigates the relation between geodesic stretch and Lipschitz maps with their Lipschitz constants weighted by some invariant probability measure. 
      \item  In Section \ref{section,BLM}, we review the theory of best Lipschitz maps and stretch loci of best Lipschitz maps from \cite{GK17}. We relate the stretch locus of best Lipschitz maps with the Mather set. We give examples for which both the stretch locus and the projection of the Mather set are not geodesic laminations.
       \item Finally in the last Section \ref{section,StretchLipVol}, we collect many known inequalities about maximal stretch, least Lipschitz constants and volumes. The equalities in these results are related to interesting rigidity statements. 
\end{itemize}

\subsection{Some notations}
\hfill\\
We collect here some commonly used notations in the note for the readers' convenience. 

\begin{table}[h]
\centering
\begin{tabular}{ll}
\hline
Notation & Meaning \\
\hline
$M$ & smooth closed manifold admitting Riemannian metrics\\
    & of negative sectional curvature \\
$S$ & orientable smooth closed topological surface of genus $\geq 2$ \\    
$\widetilde{M}$ & universal cover of $M$\\
$\Gamma$ &   fundamental group $\pi_1(M)$ \\
$[\Gamma]$ & space formed by conjugacy classes of elements in the group $\Gamma$,\\ 
& equivalently, the space of free homotopy classes of closed curves on $M$\\  
$[\gamma]$ & element in $[\Gamma]$\\
$TM$ & tangent bundle of $M$\\
$S^gM$ &  unit tangent bundle of $M$ with respect to a metric $g$\\
$\pi$ &  canonical projection from $TM$ to $M$ \\
$\phi^{g}$ & geodesic flow of $g$ \\
$R^-(M)$ & set of smooth negatively curved Riemannian metrics on $M$\\
$R^-_1(M)$ & subset of metrics with topological entropy $1$\\
$\mathcal{M} (\phi^{g})$ & space of $\phi^{g}_t$-invariant Borel measures \\
$\mathcal{M}^1 (\phi^{g})$ & subset of $\phi^{g}_t$-invariant probability measures\\
$X_g$ & infinitesimal generator of the geodesic flow $\phi_t^g$,\\ &given by $X_g(v)=\left.\frac{d}{dt} \right |_{t=0} \phi^g_t(v)$\\
$\partial^{(2)} \widetilde M$ & space of unparametrized oriented geodesics of $\widetilde{M}$,\\
& identified with the set 
$
\partial^{(2)} \widetilde M = \partial \widetilde M \times \partial \widetilde M \setminus \diag
$\\
$\iota$ & flip map on $\partial^{(2)} \widetilde{M}$ exchanging two factors of $\partial^{(2)} \widetilde{M}$,\\
& or involution on $S^{g_1}M$ given by $\iota(v)=-v$\\
$\mathcal{C}(\Gamma)$ & space of (oriented) geodesic currents\\
\hline
\end{tabular}
\end{table}

\begin{table}[h]
\centering
\begin{tabular}{ll}
\hline
Notation & Meaning \\
\hline
$\mathcal{PC}(\Gamma)$ &  projectivized space of (oriented) geodesic currents\\
$\mathcal{C}^{sym}(\Gamma)$ & space of unoriented geodesic currents\\
$L_g(\alpha)$ & $g$-length of a general Lipschitz curve $\alpha$ \\
$\ell_{g}([\gamma])$ &  $g$-length of the uniquely determined closed $g$-geodesic $\gamma^{g}$ \\
&up to parametrization in the free homotopy class $[\gamma]$\\
\hline
\end{tabular}
\vspace{0.1cm}
\caption{Notation for some common objects.}
\end{table}

\subsection*{Acknowledgements.}
We are grateful to much help of Stefan Nemirovski during the preparation of this note. The first author thanks Huiping Pan for useful explanations concerning the Teichm\"uller theory. We also thank Fanny Kassel for pointing out nice examples in Section 10.3 and Section 10.4 of their paper (\cite{GK17}). X.D. acknowledges funding by the European Research Council under ERC-Advanced grant 101095722.\\
 Both authors were partially supported by the German Research Foundation (DFG),
CRC TRR 191, \textit{Symplectic structures in geometry, algebra and dynamics.}\\

\section{Geodesic stretch} \label{section, GeodesicStretch}
\subsection{Basics on geodesic stretch} \label{subsection, BasicGeodesicStretch}
\hfill\\
Throughout this note, let $M$ be a closed oriented $n$-dimensional manifold admitting a Riemannian metric $g$ of negative sectional curvature. We will denote $\widetilde{M}$ as the universal covering of $M$ and $\Gamma$ the group of covering transformations on $\widetilde{M}$, identified with $\pi_1(M)$.
To simplify notation, we will always denote the lift of a metric $g$ on $\widetilde{M}$ as $g$ as well. 

Fix two Riemannian metrics $g_1$ and $g_2$ on $M$.  For each vector $v$
in the unit tangent bundle $S^{g_1}M$ of $g_1$, let us
 consider the arclength parametrized geodesic $c^{g_1}_v(s)$ on $M$ with respect to the Riemannian metric $g_1$, with initial condition
$\dot{c}^{g_1}_v(0) = v$.  We denote $\widetilde{c}^{g_1}_v(s)$ to be one of the lifts of $c^{g_1}_v(s)$ on the universal cover $(\widetilde{M},g_1)$. Then for each $t\in \R$, we let
 $$a(v,t) := d_{g_2}(\widetilde{c}^{g_1}_v(0),\widetilde{c}^{g_1}_v(t)) $$ be the $g_2$-distance of the endpoints of the
segment $\{\widetilde{c}^{g_1}_v(s),
0 \le s \le t \}$ with respect to $g_1$.

 It is an easy consequence of the triangle 
inequality that
the map $a(v,t)$ is a subadditive cocycle for
the geodesic flow $\phi^{g_1}$ induced by the metric $g_1$, i.e.
for all $v \in S^{g_1}M$ and 
$t_1, t_2 \in \R$,
$$a(v, t_1 + t_2) \le a(v, t_1 ) + a(\phi^{g_1}_{t_1} v, t_2).$$

Let us denote the space of $\phi^{g_1}_t$-invariant measures as $\mathcal{M} (\phi^{g_1})$ and the subset of $\phi^{g_1}_t$-invariant probability measures as $\mathcal{M}^1 (\phi^{g_1})$.   Suppose $m$ is an element in $\mathcal{M}^1 (\phi^{g_1})$. Then the subadditive ergodic 
theorem \cite[Theorem 10.1]{Wa82} implies that the limit
$$
I_m(g_1, g_2, v) := \lim_{t \to \infty} \frac{a(v,t)}{t}
$$
exists for $m$ almost every $v \in S^{g_1}M$ and defines a 
$m$-integrable function on $S^{g_1}M$ invariant under the geodesic
flow $\phi^{g_1}$. Furthermore, we have:
$$
\int_{S^{g_1}M} I_m (g_1, g_2, v) d m  =  \lim_{t \to \infty} \frac{1}{t} \int_{S^{g_1}M}
a(v,t) dm = \inf_{t > 0} \frac{1}{t} \int_{S^{g_1}M} a(v,t) dm\;.
$$
We refer to \cite{Kn95} for details.
\begin{definition}
Let $g_1, g_2$ be two Riemannian metrics on $M$ and let $m$ be an element in $\mathcal{M}^1 (\phi^{g_1})$. Then 
$$
I_m (g_1, g_2):= \int_{S^{g_1}M} I_{m}(g_1, g_2, v) d m(v)
$$
is called the \emph{geodesic stretch} of $g_1, g_2$ with respect to $m$. 
\end{definition}
Now we want to calculate the geodesic stretch of closed Riemannian manifolds  of negative curvature. We denote $R^-(M)$ to be the set of smooth negatively curved Riemannian metrics on $M$. Since for two different metrics $g_1,g_2$ in $R^-(M)$, the visual boundary $\partial_{g_1} \widetilde M$ is identified with the visual boundary of $\partial_{g_2} \widetilde M$, we denote the boundary simply by $\partial\widetilde M$. Moreover, this identification is H\"older continuous with respect to the visual metrics induced by both $g_1$ and $g_2$. The boundary $\partial\widetilde M$ has a natural H\"older structure (\cite[Chapter III.H.3]{BH99}).

For $\xi \in  \partial \widetilde M$ and $x_0 \in \widetilde M$, the Busemann function $x \mapsto b_\xi^g(x_0, x)$ is defined by
$$
 b_\xi^g(x_0, x) = \lim_{z \to \xi} d_g(x_0, z) - d_g(x,z).
$$
The Busemann functions are $C^2$ convex functions. They are also 1-Lipschitz. Furthermore, the following relations hold:
\begin{enumerate} \label{busemann cocycle}
    \item  \label{prop1bus} $b_\xi^g(x_0, x) =  b_\xi^g(x_0, x_1)+  b_\xi^g(x_1, x) \quad \quad $ (\text{cocycle property})
\item \label{prop2bus} $b_{\gamma \xi}^g(\gamma x_0, \gamma x )=  b_\xi^g(x_0, x) $ \; \text{for all}  \; $\gamma \in \Gamma$ \quad ($\Gamma $ \text{invariance})
\end{enumerate}
We call the map $\partial\widetilde M \times \widetilde M \times  \widetilde M  \to \R $ with $(\xi, x_0,x_1) \mapsto 
b_\xi^g(x_0, x_1)$ the Busemann cocycle. Besides being $C^2$ in the variables $x_0$ and $x_1$, the map is H\"older 
in $\xi$ with respect to the visual metric (\cite[Section 2.6]{Bou95}).

We define $B^g(x, \xi) : = \grad_{x} b_\xi^g(x_0,x)$.  This is well defined by property (\ref{prop1bus}). Regardless of choices of $x$ and $\xi$, we have $||B^g(x, \xi)||_g=1$.
Furthermore, for $v \in S^g \widetilde M $, we set
$v^g_\pm = \widetilde{c}^{g}_v(\pm \infty)  \in \partial \widetilde M$, where $\widetilde{c}^{g}_v(s)$ is the $g$-arclength parameterized geodesic on $\widetilde{M}$, with initial condition
$\dot{\widetilde{c}}^g_v(0) = v$. We write $\pi \colon T \widetilde M \to \widetilde M$ and $\pi \colon T  M \to M$ as the canonical projections from the tangent bundle to its base.

The set 
$$
\partial^{(2)} \widetilde M = \partial \widetilde M \times \partial \widetilde M \setminus \diag
$$
corresponds to the set of unparametrized oriented geodesics on $\widetilde M$ for any metric of negative curvature on $M$.  
Given $g \in R^-(M) $ and $p_0$ a fixed reference point in $\widetilde M $, the map $H^g_{p_0}: S^g \widetilde M \to  \partial^{(2)} \widetilde M \times \R$ given by
\begin{equation} \label{Hopfmap}
H^g_{p_0}(v) = (v^g_-, v^g_+,  b_{ v^g_+}^{\scriptscriptstyle g}(p_0, \pi (v)))
\end{equation}
is a H\"older homeomorphism which is called the \emph{Hopf parametrization} of $S^g \widetilde M$ (see also \cite{ST21}).
\begin{lemma}
We have  for all $v \in S^g \widetilde M $ and $t\in \R$,
$$
H^g_{p_0}( \phi_t^gv) =  (v^g_-, v^g_+,  b_{v^g_+}^{\scriptscriptstyle g}(p_0, \pi (v)) +t),
$$
and for $ \gamma \in \Gamma$,
$$
H^g_{p_0}( \gamma_* v) = \gamma_H H^g_{p_0}(v),
$$
where the action of $\Gamma$ on $\partial^{(2)} \widetilde M \times \R$ is defined by
$$
\gamma_H (\xi_-, \xi_+, t) = ( \gamma \xi_-, \ \gamma \xi_+, \ t+ b_{\xi_+}^{g}(\gamma^{-1}p_0,  p_0) ).
$$
\end{lemma}

\begin{lemma}\label{lem,Busem-distance}
Let $g_1, g_2$ be two Riemannian metrics in $R^-(M)$.
Then there is a constant $k =k(g_1, g_2)$ such that for all $v \in S^{g_1} \widetilde M$ and $\xi = v^{g_1}_+$,
$$
| d_{g_2}(\pi (v) , \pi (\phi_t^{g_1} v)) -b_\xi^{g_2}(\pi(v),  \pi(\phi_t^{g_1} v))| \le k.
$$
\end{lemma}
\begin{proof}
For  $v \in S^{g_1} \widetilde M$, consider $w \in S^{g_2} \widetilde M$ given by  $w = B^{g_2}(\pi(v), \xi)$. Since $(\widetilde M,g_1)$ and $(\widetilde M,g_2)$ are quasi-isometric, by the Morse Lemma (\cite{Mo24}), there exists a constant $R_1=R_1(g_1,g_2)>0$ such that for each $t \ge 0$, we can find some $s(t) \in \R$ and
$$
d_{g_2}( \pi (\phi_t^{g_1} v ),\pi (\phi_{s(t)}^{g_2} w )) \le R_1(g_1, g_2). 
$$
This implies that
$$
|d_{g_2}( \pi (v), \pi (\phi_t^{g_1} v )) -d_{g_2}( \pi (v), \pi (\phi_{s(t)}^{g_2} w ))| 
\le R_1(g_1,g_2),
$$
and
$$
|b^{g_2}_\xi( \pi (v), \pi (\phi_t^{g_1} v )) -b^{g_2}_\xi( \pi (v), \pi (\phi_{s(t)}^{g_2} w ))|
 \le R_1(g_1,g_2),
$$
where the first inequality follows from the triangle inequality and the second from the fact that Busemann functions are $1$-Lipschitz.
Hence
\begin{equation*}
\begin{split}
 & | d_{g_2}(\pi (v) , \pi (\phi_t^{g_1} v))  -b_\xi^{g_2}(\pi(v),  \pi (\phi_t^{g_1} v))| \\ 
 = &
   | d_{g_2}(\pi (v) , \pi (\phi_t^{g_1} v)) - d_{g_2}( \pi (v), \pi (\phi_{s(t)}^{g_2} w ))
  +  d_{g_2} ( \pi (v), \pi (\phi_{s(t)}^{g_2} w))\\
 &  - b_\xi^{g_2}(\pi(v),  \pi(\phi_{s(t)}^{g_2} w))
  +b_\xi^{g_2}(\pi(v),  \pi(\phi_{s(t)}^{g_2} w)) -b_\xi^{g_2}(\pi(v),  \pi(\phi_t^{g_1} v))| \\
  \le  &
 | d_{g_2}(\pi (v) , \pi (\phi_t^{g_1} v)) - d_{g_2}( \pi (v), \pi (\phi_{s(t)}^{g_2} w ))| +|b_\xi^{g_2}(\pi(v),  \pi(\phi_{s(t)}^{g_2} w)) -b_\xi^{g_2}(\pi(v),  \pi(\phi_t^{g_1} v))|\\
  \le & 2 R_1(g_1, g_2) =: k(g_1,g_2).
\end{split}
\end{equation*}
\end{proof}

\begin{corollary} \label{corollary, GeodesicStretchBusemannFunction} 
Let $g_1, g_2$ be two Riemannian metrics in $R^-(M)$
and  let $m$ be a measure in $\mathcal{M}^1 (\phi^{g_1})$. Then for $m$-almost  all $v \in  S^{g_1}  M$, any of its lifts $\widetilde{v} \in  S^{g_1} \widetilde M$ satisfies 
$$
\lim_{t \to \infty} \frac{1}{t}d_{g_2} (\pi (\widetilde{v}), \pi(\phi_t^{g_1} \widetilde{v})) = \lim_{t \to \infty} \frac{1}{t}\int_0^t g_2(B^{g_2}(\pi (\phi_s^{g_1}\widetilde{v}), \widetilde{v}^{g_1}_+), \phi_s^{g_1}\widetilde{v})ds.
$$
Furthermore, 
$$
I_m(g_1,g_2) = \int_{S^{g_1}M} g_2(B^{g_2}(\pi (\widetilde{v}), \widetilde{v}^{g_1}_+), \widetilde{v}) dm.
$$
\end{corollary}
\begin{proof}
For $m$-almost  all $v \in  S^{g_1}  M$ and any lift $\widetilde{v} \in  S^{g_1} \widetilde M$ of $v$ the limit $\lim\limits_{t \to \infty} \frac{1}{t}d_{g_2} (\pi (\widetilde{v}), \pi(\phi_t^{g_1} \widetilde{v}))$ exists and is equal to $I_m(g_1, g_2, v)$. If $\xi = \widetilde{v}^{g_1}_+$, we obtain by Lemma \ref{lem,Busem-distance},
\begin{equation*}
\begin{split}
 I_m(g_1, g_2, v) =&\lim_{t \to \infty}  \frac{1}{t}d_{g_2} (\pi (\widetilde{v}), \pi(\phi_t^{g_1} \widetilde{v})) \\
=&\lim_{t \to \infty}  \frac{1}{t} b_\xi^{g_2}(\pi (\widetilde{v}),  \pi(\phi_t^{g_1} \widetilde{v}))\\
=&\lim_{t \to \infty}  \frac{1}{t} \int_0^t g_2(\grad_{\pi (\phi_s^{
 g_1}\widetilde{v})} b_\xi^{g_2}(\pi (\widetilde{v}), \pi (\phi_s^{g_1}\widetilde{v})), \phi_s^{g_1}\widetilde{v})ds\\
=&\lim_{t \to \infty}  \frac{1}{t} \int_0^t g_2(B^{g_2}(\pi (\phi_s^{g_1}\widetilde{v}),   \widetilde{v}^{g_1}_+), \phi_s^{g_1}\widetilde{v})ds.
\end{split}
\end{equation*}
By the Birkhoff ergodic theorem for invariant measures (\cite[Theorem 1.14]{Wa82}), we obtain
\begin{equation*}
\begin{split}
I_m(g_1,g_2) &=    \int_{S^{g_1}M} \lim_{t \to \infty} \frac{1}{t}d_{g_2} (\pi (\widetilde{v}), \pi(\phi_t^{g_1} \widetilde{v})) dm(v)\\
&= \int_{S^{g_1}M}  \lim_{t \to \infty} \frac{1}{t} \int_0^t g_2(B^{g_2}(\pi (\phi_s^{g_1}\widetilde{v}),   \widetilde{v}^{g_1}_+ ), \phi_s^{g_1}\widetilde{v})ds dm(v) \\
&= \int_{S^{g_1}M} g_2(B^{g_2}(\pi (\widetilde{v}), \widetilde{v}^{g_1}_+ ), \widetilde{v}) dm(v).
\end{split}
\end{equation*}
\end{proof}

\subsection{Orbit equivalence of geodesic flows on manifolds of negative curvature}
\hfill\\
In this subsection, we construct an orbit equivalence between geodesic flows with respect to different metrics in $R^-(M)$ which will be useful in the sequel. This construction is also in \cite{ST21}.

\begin{lemma} \label{lem:conj}

Let $g_1, g_2$ be two Riemannian metrics in $R^-(M)$.
Consider the map $\Psi_{g_1, g_2}: S^{g_1}\widetilde M  \to  S^{g_2}\widetilde M $ defined by $\Psi_{g_1, g_2}(v)  =w$ where 
$w \in  S^{g_2}\widetilde M $ is the unique vector with $w^{g_2}_+ = v^{g_1}_+$ and  $w^{g_2}_-= v^{g_1}_-$ and 
$b^{g_2}_{v^{g_1}_+}( \pi(v), \pi(w)) =0$. Then $\Psi_{g_1, g_2}$ is a H\"older continuous surjective map with
$$
\phi^{g_2}_{\tau(v,t)}(\Psi_{g_1, g_2}(v)) = \Psi_{g_1, g_2}( \phi_t^{g_1}(v)),
$$
where 
$$
\tau(v,t) = b^{g_2}_{v^{g_1}_+} ( \pi(v), \pi(\phi_t^{g_1}(v)) ) =\int_{0}^t g_2( B^{g_2}(\pi \phi_s^{g_1}v, v^{g_1}_+) , \phi_s^{g_1}v) ds,
$$
 for all $v \in S^{g_1}\widetilde M $. Furthermore, for all $\gamma \in \Gamma$
we have 
$$
\gamma_{\ast}\Psi_{g_1, g_2}(v) =\Psi_{g_1, g_2} (\gamma_{\ast}v),
$$ and 
${\tau(\gamma_{\ast}v,t)} = \tau(v,t)$. Therefore, the map $\Psi_{g_1, g_2}$ descends to an orbit equivalence between the geodesic flows on the quotients $S^{g_1}M$ and $S^{g_2}M$.\\ 
Moreover
$$
\Psi_{g_2, g_1} \circ \Psi_{g_1, g_2}(v) =  \phi_{t(v)}^{g_1}(v)
$$
for all $v  \in S^{g_1}\widetilde M$ where $t(v) = b^{g_1}_{v^{g_1}_+}( \pi(v), \pi( \Psi_{g_1, g_2}(v) ))$.
\end{lemma} 
\begin{proof}
 By the cocycle property (\ref{prop1bus}) of the Busemann function and the assumption that $b^{g_2}_{v^{g_1}_+}( \pi(v), \pi(\Psi_{g_1, g_2}(v) )) =0$ , for each $(v,t) \in  S^{g_1}\widetilde M \times \R$, we have 
 \begin{align*}
&b^{g_2}_{v^{g_1}_+}( \pi \phi_t^{g_1}(v), \pi \phi^{g_2}_{\tau(v,t)}\Psi_{g_1, g_2}(v)) = b^{g_2}_{v^{g_1}_+}(  \pi  \phi_t^{g_1}(v),\pi \Psi_{g_1, g_2}(v)) +\tau(v,t)\\
=& b^{g_2}_{v^{g_1}_+}( \pi(  \phi_t^{g_1}(v)), \pi(v)) +  b^{g_2}_{v^{g_1}_+}( \pi(v),   \pi(\Psi_{g_1, g_2}(v))  ) +\tau(v,t) =0.
 \end{align*}
 This implies by the definition of $\Psi_{g_1, g_2}$,
$$
\Psi_{g_1, g_2}( \phi_t^{g_1}(v))  = \phi^{g_2}_{\tau(v,t)}(\Psi_{g_1, g_2}(v)),
$$
for $\tau(v,t) = b^{g_2}_{v^{g_1}_+} ( \pi(v),\pi(\phi_t^{g_1}(v)))$.

The   $\Gamma$-equivariance of the orbit equivalence and the time change follow from the  $\Gamma$-invariance (\ref{prop2bus}) of the Busemann function. Furthermore, for
each $v  \in S^{g_1}\widetilde M$, there exists a number $t(v) \in \R$ such that
$$
\Psi_{g_2, g_1} \circ \Psi_{g_1, g_2}(v) =  \phi_{t(v)}^{g_1}(v).
$$
The definition of the orbit equivalence $\Psi_{g_2, g_1}$ gives
$$
0 = b^{g_1}_{v^{g_1}_+}( \pi(\Psi_{g_1, g_2}(v) ), \pi(\phi_{t(v)}^{g_1}(v))) =   b^{g_1}_{v^{g_1}_+}( \pi(\Psi_{g_1, g_2}(v) ), \pi(v)))  - t(v),
$$
which yields the last assertion.
\end{proof}

With abuse of notation, we also denote the quotient orbit equivalence as $\Psi_{g_1,g_2}: S^{g_1}M \to S^{g_2}M$. The next corollary follows immediately from the construction of the orbit equivalence.

\begin{corollary}
Let $g_1, g_2$ be two metrics in $R^-(M)$ and let $m$ be an ergodic measure in $\mathcal{M}^1 (\phi^{g_1})$. Then for $m$-almost every $v \in  S^{g_1} \widetilde M$, 
we have
$$
\lim_{t \to \infty} \frac{\tau(v,t)}{t} = I_m(g_1,g_2) = \int_{S^{g_1}M} g_2(B^{g_2}(\pi (v), v^{g_1}_+) , v )dm.
$$
\end{corollary}

Since the Busemann function is $\Gamma$ invariant, the definition for $\tau(v,t)$ descends to $S^{g_1}M$.
\begin{definition}\label{def:timechange}  
We call the function $\tau=\tau_{g_1,g_2}: S^{g_1}M \times \R  \to \R$ given by
$$
\tau_{g_1,g_2}(v,t) =  b^{g_2}_{\widetilde{v}^{g_1}_+} ( \pi(\widetilde{v}), \pi(\phi_t^{g_1}(\widetilde{v})) )
$$
the \emph{time change} and  $a_{g_1,g_2} : S^{g_1}M  \to \R$  given by
$$
a_{g_1,g_2}(v) = \left.\frac{d}{dt}  \right|_{t = 0} 
  b^{g_2}_{\widetilde{v}^{g_1}_+} ( \pi(\widetilde{v}), \pi(\phi_t^{g_1}(\widetilde{v})) )  = g_2(B^{g_2}(\pi (\widetilde{v}), \widetilde{v}^{g_1}_+) , \widetilde{v} )
$$
the \emph{infinitesimal time change} of $\phi^{g_1}$ to  $\phi^{g_2}$, where $\widetilde{v}$ is a lift of $v$ in $S^{g_1}\widetilde{M}$.
\end{definition}

It is clear that the Busemann function and its derivatives are defined on the universal cover $S^{g_1}\widetilde{M}$. We will often drop the distinction between $v\in S^{g_1}M$ and $\widetilde{v}\in S^{g_1}\widetilde{M}$ and write
$
a_{g_1,g_2}(v) = \left.\frac{d}{dt}  \right|_{t = 0} 
  b^{g_2}_{v^{g_1}_+} ( \pi(v), \pi(\phi_t^{g_1}(v)) ) 
  = g_2(B^{g_2}(\pi (v), v^{g_1}_+) , v )
$.
 
\begin{remark} \label{remark,regularityTimeChange}
We make the following remarks about time changes and infinitesimal time changes.
\begin{enumerate}
\item
The H\"older structure of the boundary $\partial M$ implies the H\"older continuity of $a_{g_1g_2}$. 

\item 
The time change $\tau_{g_1,g_2}$ is an additive (H\"older) cocycle. It satisfies for $t,s\in \R$
$$
\tau_{g_1,g_2}(v,t+s) = \tau_{g_1,g_2}(v,t) + \tau_{g_1,g_2}(\phi_t^{g_1}v,s)
$$
and 
$$
\tau_{g_1,g_2}(v,t) = \int_0^t a_{g_1,g_2}(\phi_s^{g_1}v) ds.
$$
\end{enumerate}
\end{remark}

We discuss in the next remark the existence of a bijective orbit equivalence.

\begin{remark} 
Note that $\Psi_{g_1, g_2}: S^{g_1}M  \to S^{g_2}M$ maps orbits of the geodesic flow of $g_1$ surjectively onto orbits of the geodesic flow of $g_2$ but is not necessarily
injective.
An injective map can be obtained by modifying the time change and the infinitesimal time change in its Liv\v{s}ic cohomology class (see Gromov \cite{Gr00} and also \cite{Kn02} for details). Since these infinitesimal time changes are Liv\v{s}ic cohomologous, we are free to choose which one to use for our theory.
Throughout this note, we prefer to work with the non-modified 
orbit equivalence and the non-modified infinitesimal time change $a_{g_1,g_2}$.
\end{remark}

Recall $X_g(v) =\left.\frac{d}{dt} \right |_{t=0} \phi^g_t(v)$ the infinitesimal generator of the geodesic flow
$\phi^g_t : S^gM \to S^gM$.

\begin{corollary}\label{cor:equiv}

Let $g_1, g_2$ be two Riemannian metrics in $R^-(M)$. Then the following are equivalent.
\begin{enumerate}
\item
The metrics $g_1$ and $g_2$ have the same marked length spectrum.
\item
The infinitesimal time change is cohomologous to $1$, i.e. there exists a H\"older continuous function $u: S^{g_1}M \to \R$ 
such that for any $t\in\R$
$$
a_{g_1,g_2}(v) -1 = X_{g_1}(u)(v).
$$
\item
The geodesic flows of $g_1$ and $g_2$ are time-preserving conjugate. More precisely, 
 there exists a H\"older continuous map $G: S^{g_1}M \to S^{g_2}M$  such that
 $$
 \phi_t^{g_2}(G(v)) =   G (\phi_t^{g_1}(v)),
 $$
and  $G(v)^{g_2}_{\pm} = v^{g_1}_{\pm}$.
 \end{enumerate}
\end{corollary}
\begin{proof}
We prove $(1)\Rightarrow (2) \Rightarrow (3) \Rightarrow (1)$. Assume that $g_1, g_2$ have the same marked length spectrum. Let $\phi^{g_1}_t(v) $ be a periodic orbit with period $l>0$. Then
$$
\phi^{g_2}_{\tau(v,l)}\circ \Psi_{g_1, g_2}(v) = \Psi_{g_1, g_2}\circ\phi^{g_1}_{l}(v),
$$
with $ l= \tau(v,l) = \int_0^l a_{g_1,g_2}( \phi_s^{g_1}(v)) ds$. By Liv\v{s}ic theorem \cite[Chapter 19]{KH95}, there exists a H\"older continuous function $u: S^{g_1}M \to \R$ differentiable along the orbits of the geodesic flow $\phi_t^{g_1}$
such that
$$
a_{g_1,g_2}(v) -1 = X_{g_1}(u)(v).
$$
We obtain (2). Therefore
$$
\tau(v,t) = t + u(\phi_t^{g_1}(v)) -u(v).
$$
 Hence 
 $$
 \phi^{g_2}_t \phi^{g_2}_{u(\phi_t^{g_1}(v)) -u(v)}\Psi_{g_1, g_2}(v) =\Psi_{g_1, g_2}(\phi_t^{g_1}(v)),
$$
and $G(v) = \phi^{g_2}_{-u(v)}\Psi_{g_1, g_2}(v) $ defines the required orbit equivalence, which preserves the points at infinity. This yields (3).
Under this assumption the marked length spectrum of $g_1$ and $g_2$ are the same. This finishes the proof. 
\end{proof}

\subsection{Geodesic currents, intersections, and geodesic stretches}
\hfill\\
In this section, we reinterpret the geodesic stretch using the framework of \emph{geodesic currents}, introduced by Bonahon (see \cite{Bon91} and \cite{ST21}). In certain contexts, geodesic currents offer a more 
convenient and flexible tool than invariant measures 
for the purposes of our theory. This is, for instance, evident in Proposition \ref{prop,triangleInequality} and later Proposition \ref{prop,whyIntegralWellDefined}. Throughout the note, we will alternate between  the perspectives of invariant probability measures and geodesic currents, depending on which is more suitable to the context.

\subsubsection{Geodesic currents} We introduce \emph{geodesic currents} in this subsection and discuss many of its properties. We omit many proofs in this subsection since they are classical in the field. Standard references for this topic include \cite{Bon88} and \cite{Bon91}.

\begin{definition}
A \emph{geodesic current}\footnote{We consider in this note oriented geodesic currents in contrast to unoriented geodesic currents introduced in \cite{Bon88}.} is a locally finite $\Gamma$-invariant positive Borel measure $\mu$ on $\partial^{(2)} \widetilde M$.
\end{definition}

Let us denote the space of $\Gamma$-invariant geodesic currents by $\mathcal{C} (\Gamma) $. It is equipped with the weak-* topology of convergence of Borel measures, that is,  a sequence $\mu_n \in \mathcal{C} (\Gamma) $
converges to  $\mu \in \mathcal{C} (\Gamma) $ if and only if for all continuous functions $\varphi: \partial^{(2)} \widetilde M  \to \R$
with compact support, we have 
$$ \lim_{n \to \infty} \int \varphi d\mu_n =  \int \varphi d\mu. $$ 

The simplest examples of geodesic currents are Dirac currents. For any $\gamma\in \Gamma$, let $\delta_{\gamma}$ denote the Dirac measure supported at $(\gamma^-,\gamma^+)\in \partial^{(2)} \widetilde M$, where $\gamma^-$ (resp. $\gamma^+$) is the repelling (resp. attractive) fixed point of $\gamma$ on $\partial \Gamma$. Then associated to the conjugacy class $[\gamma]\in[\Gamma]$ is the \emph{Dirac current} in $\mathcal{C} (\Gamma) $ given by
$$\delta_{[\gamma]}=\sum_{\alpha\in[\gamma] }\delta_{\alpha}.$$

Let $ \R_{>0} $ denote the set of all positive real numbers. The following lemma is classical (see, for example, \cite{Bon91}).
\begin{lemma}\label{prop:proj-current}
 Denote by $\mathcal{PC}(\Gamma) =  \mathcal{C} (\Gamma)/ \R_{>0} $ the projectivized space of currents. Then $\mathcal{PC}(\Gamma)$ is compact.
\end{lemma}

We then explain the relation between geodesic currents and invariant probability measures. Recall for $g \in R^-(M)$ and $p_0 \in \widetilde M$, we have introduced the Hopf map 
$H^g_{p_0} : S^g \widetilde M \to \partial^{(2)} \widetilde M \times \R$ in Section \ref{subsection, BasicGeodesicStretch}.
\begin{proposition}\label{prop:m-current}  
 For $g \in R^-(M)$, each geodesic current $\mu\in \mathcal{C}(\Gamma)$ defines a  measure $m_{\mu}^g$ which is both $\Gamma$-invariant and $\phi_t^g$-invariant, with local product structure $dm_{\mu}^g= d\mu \times dt$ in the Hopf coordinate $H^g_{p_0}$. This map $\mu\mapsto m^g_{\mu}$ yields a homeomorphism between $\mathcal{C} (\Gamma)$ and $ \mathcal{M}(\phi^g)$ which is linear in the sense that $k_1\mu_1+k_2\mu_2$ is mapped to $k_1 m_{\mu_1}^g+k_2 m_{\mu_2}^g$ for $k_1,k_2\geq 0$.

Moreover, for a measurable and non-negative function $f \colon  S^g \widetilde M \to \R$, its Lebesgue integral with respect to $m^g_{\mu}$ can be expressed as 
$$
\int_{S^g\widetilde M} \ f dm_{\mu}^g = \int_{\partial^{(2)} \widetilde M}
\int_{\R} f(\phi_t^g (v_{(\xi_-, \xi_+)}))dt \ d\mu (\xi_-, \xi_+).
$$
where $v_{(\xi_-, \xi_+)} = (H^g_{p_0})^{-1} (\xi_-, \xi_+, 0)$,
\end{proposition}

The \emph{flip invariant currents}, which are oriented currents invariant under the flip map $\iota$ exchanging two factors of $\partial^{(2)} \widetilde{M}$, correspond to \emph{reflexive invariant measure} which are invariant under the involution on $S^{g_1}M$ given by $\iota(v)=-v$. In particular, the Liouville current is flip invariant, since the Liouville measure is reflexive. 

To further consider integration over the quotient unit tangent bundle $S^g M$, we need to take a \emph{fundamental domain} $\mathcal{F}$ for the $\Gamma$ action on $\widetilde M$, that is,  a connected and contractible set of $\widetilde M$ such that each $\Gamma$-orbit $\Gamma\cdot p$ meets the interior of $\mathcal{F}$
at exactly one point. We let
$$ S^g \mathcal{F} := \bigcup_{p \in \mathcal{F} }S^g_p \widetilde {M}
$$
be a fundamental domain for the $\Gamma$ action on $S^g \widetilde {M}$.

\begin{remark}\label{remark:CurrentMeasure}
For each metric $g\in R^-(M)$, the $\Gamma$-invariant measure $m_{\mu}^g$ on $S^g \widetilde M$ induces a  $\phi^g$-invariant measure on $S^g M$ as follows. 
If $f \colon S^g M \to \R $ is a bounded measurable function and $ \widetilde f \colon S^g \widetilde {M} \to \R $ is its lift, then with abuse of notation, we let
$$
\int_{S^g M} f \ dm_{\mu}^g := \int_{S^g \widetilde M} \widetilde f  \cdot \chi_{_{S^g \mathcal{F}}}     \ dm_{\mu}^g
$$
where $\chi_{_{S^g \mathcal{F}}}$ is the characteristic function of the set $S^g \mathcal{F}$. The definition is independent of the choice of the fundamental domain $\mathcal{F}$.
\end{remark}

\subsubsection{Geodesic stretches and intersections via geodesic currents} \label{subsection GeodesicStretchIntersection}
We reinterpret the geodesic stretch from the viewpoint of geodesic currents. 

\begin{definition}\label{def:defGeodesicStretch}
The  geodesic stretch is
 \begin{align*}
 &I:  \mathcal{C}(\Gamma) \times R^-(M) \times R^-(M) \to \mathbb{R}_{\geq 0}\\
& I_\mu(g_1,g_2) := I_{\hat m_{\mu}^{g_1}}(g_1, g_2),
\end{align*}
where
$$\hat m_{\mu}^{g_1} :=\frac{m_{\mu}^{g_1}}{m_{\mu}^{g_1}(S^{g_1} M)}\in \mathcal{M}^1(\phi^{g_1})$$ and the geodesic stretch $I_{\hat m_{\mu}^{g_1}}(g_1, g_2)$ defined using geodesic flow invariant probability measure has been introduced in Section \ref{subsection, BasicGeodesicStretch}.
\end{definition}

There is a pairing called the \emph{intersection} that is closely related to the geodesic stretch. 

\begin{definition}\label{def:intersection}
The \emph{intersection} $i: R^-(M) \times \mathcal{C}(\Gamma) \to \mathbb{R}_{\geq 0}$  is given by
$$
i(g, \mu) := m_{\mu}^g(S^gM).
$$
\end{definition}

\begin{remark} \label{remark,intersection}
The intersection $i(g, \cdot)$ is the translation distance function (length function) with respect to $g$ (see \cite{Bon91}). In particular, when $\mu=\delta_{[\gamma]}$ is the Dirac current associated to $[\gamma]\in [\Gamma]$, the intersection $i(g, \delta_{[\gamma]})$ equals the $g$-length of the closed $g$-geodesic in the free homotopy $[\gamma]\in [\Gamma]$. 
\end{remark}

To relate the intersection and the geodesic stretch, we need the following proposition. For its proof, we refer to \cite[Proposition 2.13, Remark 2.14]{ST21}.
\begin{proposition} \label{prop: transformation-formula}
Let $g_1, g_2\in R^-(M)$  and  $\mu \in \mathcal{C} (\Gamma)$.
If $G:S^{g_2}M \to \R$ is a $ m_{\mu}^{g_2}$ integrable function, then $G\circ \Psi_{g_1, g_2} : S^{g_1}M \to \R$ is  $ m_{\mu}^{g_1}$ integrable 
and 
$$
\int_{S^{g_2}M} G \ dm_{\mu}^{g_2} = \int_{S^{g_1}M} G \circ \Psi_{g_1, g_2}  a_{g_1,g_2} dm_{\mu}^{g_1},
$$
where $\Psi_{g_1, g_2}$ is the map introduced in Lemma \ref{lem:conj}.
In particular,
$$
m_{\mu}^{g_2}(S^{g_2}M) = \int_{S^{g_1}M}   a_{g_1,g_2} dm_{\mu}^{g_1}.
$$

\end{proposition}

The following relates the intersection and the geodesic stretch for geodesic currents. 

\begin{corollary}\label{cor:s-current} 
Given  a geodesic current  $\mu\in\mathcal{C}(\Gamma) $ and given $g_1,g_2 \in R^-(M)$ so that $m_{\mu}^{g_1} \in \mathcal{M}(\phi^{g_1})$ has finite mass, then 
$$
I_\mu(g_1,g_2) = \frac{i(g_2, \mu)}{i(g_1, \mu)}.
$$
In particular, for all $\lambda >0$,
$$I_{\lambda\mu}(g_1,g_2) = I_{\mu}(g_1,g_2) $$  and therefore  $I_{\mu}(g_1,g_2)$ descends to be defined on the projective space $\mathcal{PC}(\Gamma)$. We denote the descended map as $I([\mu], g_1,g_2)$.

\end{corollary}

\begin{proof}

From the definition of the geodesic stretch (Definition \ref{def:defGeodesicStretch}) and Proposition \ref{prop: transformation-formula}, for $\mu\in\mathcal{C}(\Gamma) $, we have
    $$
I_\mu(g_1,g_2) = \frac{1}{m_{\mu}^{g_1}(S^{g_1}M)}\int_{S^{g_1}M}    a_{g_1,g_2} dm_{\mu}^{g_1} = \frac{m_{\mu}^{g_2}(S^{g_2}M) }{m_{\mu}^{g_1}(S^{g_1}M)} = \frac{i(g_2, \mu)}{i(g_1, \mu)}.
$$
\end{proof}

The space of \emph{unoriented currents}, denoted as $\mathcal{C}^{sym}(\Gamma)$, is a quotient of the space $\mathcal{C}(\Gamma)$ by the flip map $\iota$ exchanging two factors of $\partial^{(2)} \widetilde{M}$.. The flip invariant currents, which are oriented geodesic currents invariant under $\iota$, are naturally identifed as currents in $\mathcal{C}^{sym}(\Gamma)$. 

When $M=S$ is a closed surface, the intersection generalizes to the geometric intersection number for curves and extends to be a continuous symmetric function $i:\mathcal{C}^{sym} (\Gamma) \times  \mathcal{C}^{sym} (\Gamma) \to \mathbb{R}$, called the \emph{intersection number} (\cite[Proposition 3]{Bon88}, see also \cite{Ot90}). We will present the magic of this concept in a proof in the last section (Proposition \ref{prop, bonahonIntersection}). Unfortunately, in higher dimensions, a similar pairing of geodesic currents like the intersection number is not known. When $M$ is a surface, the intersection is identified with the intersection number as follows. 
  
 \begin{proposition}
   When $\dim M=2$, there is an embedding of  $R^-(M) $ into $\mathcal{C} (\Gamma)$ via Liouville currents: for each $g \in R^-(M)$, we associate  the Liouville current $\lambda_g$ in $\mathcal{C}(\Gamma) $  defined by demanding that
    $m_{\lambda_g}^g $ is the non-normalized Liouville measure with
   $m_{\lambda_g}^g (S^gM) = 2 \pi \vol(M,g)$. Furthermore,  this embedding identifies the intersection and the intersection number for flip invariant currents. For each flip invariant current $\mu \in \mathcal{C}(\Gamma)$,
     $$
    i(\lambda_g, \mu) = i(g, \mu) =  m_{\mu}^g(S^gM).
    $$  
 \end{proposition}
   \begin{proof}
   As has been observed by Otal (see \cite{Ot90}) for negatively curved surfaces the Liouville current is determined by the marked length spectrum. Using this he obtained
   the rigidity of the marked length spectrum which yields 
     the first statement. Note that his proof cannot be extended to higher dimensions.
     
     We want to show $ i(\lambda_g, \mu) = i(g, \mu)$.  For a fixed $g_0  \in R^-(M) $ and all free homotopy classes $[\gamma]\in [\Gamma]$, we can associate a flip invariant current $\frac{1}{2}(\delta_{[\gamma]}+\delta_{[\gamma^{-1}]})$. By the definition of infinitesimal time change $a_{g_0,g}$, we have
    $$
    i(\lambda_g, \frac{1}{2}(\delta_{[\gamma]}+\delta_{[\gamma^{-1}]}))= \ell_{g}([\gamma])= \int _{S^{g_0} M} a_{g_0,g} dm^{g_0}_{\delta_{[\gamma]}}, 
$$ 
Hence for any flip invariant current $\mu$, we can find by Sigmund's theorem (\cite[Theorem 1]{Sig72}) a sequence $\delta_{[\gamma_n]}$ converging weakly to $ \mu$, after scaling by some real multiplies. By the flip invariant property, it holds indeed $\frac{1}{2}(\delta_{[\gamma_n]}+\iota^*\delta_{[\gamma_n]})=\frac{1}{2}(\delta_{[\gamma_n]}+\delta_{[\gamma_n^{-1}]})$ also converging weakly to $\mu$. Using the continuity of the intersection number and  Proposition 
\ref{prop: transformation-formula} again, we obtain
$$
i(\lambda_g, \mu)= \int _{S^{g_0} M} a_{g_0,g} dm^{g_0}_{\mu} =  m_{\mu}^g(S^gM) = i(g, \mu). 
$$
   \end{proof}

We continue to study properties of the geodesic stretch. The following lemma is an immediate consequence of Corollary \ref{cor:s-current}.
\begin{lemma} \label{lemma,cocycleRelationStretch}

 The geodesic stretch $I: \mathcal{PC}(\Gamma) \times R^-(M) \times R^-(M) \to \mathbb{R}_{\geq 0}$ given by
$$
I([\mu], g_1,g_2) =I_{\mu}(g_1,g_2)=  \frac{i(g_2, \mu)}{i(g_1, \mu)}
$$
satisfies for any $\mu\in\mathcal{C}(\Gamma)$ and any $g_1, g_2, g_3\in R^-(M)$,
 $$
I_\mu(g_1, g_3) =I_\mu(g_1, g_2) I_\mu(g_2, g_3).
$$

In particular,
$$
I_\mu(g_1, g_2) = \frac{1}{I_\mu(g_2, g_1)}.
$$   
\end{lemma}

Moreover,

\begin{lemma}[Continuity of geodesic stretches]\label{lemma, GeodesicStretchContinuous}
The geodesic stretch 
$$I: \mathcal{PC}(\Gamma) \times R^-(M) \times R^-(M) \to \mathbb{R}_{\geq 0}$$
$$
I([\mu], g_1,g_2) =I_{\mu}(g_1,g_2)=  \frac{i(g_2, \mu)}{i(g_1, \mu)}
$$
 is continuous.   
\end{lemma}  
\begin{proof}
According to Corollary \ref {cor:s-current},
we have
$$
I_\mu(g_1,g_2) = \frac{1}{m_{\mu}^{g_1}(S^{g_1}M)}\int_{S^{g_1}M}    a_{g_1,g_2} dm_{\mu}^{g_1} =\int_{S^{g_1}M}    a_{g_1,g_2} d\widehat{m}_{\mu}^{g_1},
$$
where  $a_{g_1,g_2}$ is the infinitesimal time change given by
$$
a_{g_1,g_2}(v) = \left.\frac{d}{dt}  \right|_{t = 0} 
  b^{g_2}_{v^{g_1}_+} ( \pi(v), \pi(\phi_t^{g_1}(v)) )  = g_2(B^{g_2}(\pi (v), v^{g_1}_+) , v )
$$
We first show the continuity for $(g,\mu) \mapsto I_\mu(g_1,g)$.
For this we claim that for any $g_2 \in R^-(M)$ all $\varepsilon >0$ there exists a $C^2$ open neighborhood $U$ of $g_2$ such that
$$
 \|B^{g}(\pi (v), v^{g_1}_+) -B^{g_2}(\pi (v), v^{g_1}_+) \|_{g_2} < \varepsilon
$$
for all $v \in S^{g_1}\widetilde M$. 
If this is not the case, there exists $\varepsilon_0 >0$ and sequences $\{h_n\}_{n\geq0} \subset R^-(M) $ and $\{v_n\}_{n\geq0} \subset S^{g_1} \widetilde M$ where  $h_n$ converges to $g_2$ such that
$$
\|B^{h_n}(\pi (v_n), v^{g_1}_{n_+}) -B^{g_2}(\pi (v_n), v^{g_1}_{n_+}) \|_{g_2} \ge \varepsilon_0.
$$
By compactness of $M = \widetilde M / \Gamma$, we can assume that $v_n$ converges to some $v \in S^{g_1} \widetilde M $. 
Note that $B^{h_n}(\pi (v_n), v^{g_1}_{n_+} ) \in S^{h_n}M$ is the initial vector
of the $h_n$-geodesic $c_n: [0, \infty) \to \widetilde M$ with  $c_n(0) = \pi (v_n)$ and $c_n(\infty) = v^{g_1}_{n_+}$. By choosing a subsequence of $c_n$, we can assume that $c_n$ converges to a $g_2$ geodesic $c$ with $c(0) = \pi (v)$ and $c(\infty) =  v^{g_1}_+$
we have $\dot c(0) = B^{g_2}(\pi (v), v^{g_1}_+)$ and therefore $B^{h_n}(\pi (v_n), v^{g_1}_{n_+})$ converges to $B^{g_2}(\pi (v), v^{g_1}_+) $. Since by a similar argument, $B^{g_2}(\pi (v_n), v^{g_1}_{n_+})$ converges to $B^{g_2}(\pi (v), v^{g_1}_+)$, we obtain a contradiction.
This yields that for all  $\varepsilon >0$ there exists a $C^2$ open neighborhood $U$ of $g_2$ such that
$$
|a_{g_1,g}(v) - a_{g_1,g_2}(v)| < \varepsilon,
$$
for all $g \in U$ and $v \in S^{g_1}M $. Now let $\mu_n \in \mathcal{C}(\Gamma)$
converge weakly to $\mu$ and $h_n \in R^-(M)$ converge in the $C^2$-topology to $g$. We obtain
\begin{align*}
\left|I_{\mu_n}(g_1,h_n) -I_{\mu}(g_1,g_2)\right|& = \left|\int_{S^{g_1}M}   a_{g_1,h_n} d\widehat{m}_{\mu_n}^{g_1} -\int_{S^{g_1}M} a_{g_1,g_2} d\widehat{m}_{\mu}^{g_1}\right|\\
& \le \left|\int_{S^{g_1}M}a_{g_1,h_n} d\widehat{m}_{\mu_n}^{g_1} -\int_{S^{g_1}M}a_{g_1,g_2} d\widehat{m}_{\mu_n}^{g_1} \right|\\
& \quad
 + \left|\int_{S^{g_1}M}a_{g_1,g_2} d\widehat{m}_{\mu_n}^{g_1} - \int_{S^{g_1}M} a_{g_1,g_2} d\widehat{m}_{\mu}^{g_1} \right|.
\end{align*}
Since $a_{g_1,h_n}$ converges uniformly to $a_{g_1,g_2}$ and $\widehat{m}_{\mu_n}^{g_1}$ converges weakly to $\widehat{m}_{\mu}^{g_1}$,
we obtain that $I_{\mu_n}(g_1,h_n)$ converges to  $I_{\mu}(g_1,g_2)$. This implies the continuity of  $(\mu,g) \mapsto I_\mu(g_1,g)$ for all $g_1 \in R^-(M)$.

Since $I_\mu(g_1, g_2) = \frac{1}{I_\mu(g_2, g_1)} $, 
we also obtain the continuity of $(g , \mu) \mapsto I_\mu(g, g_2)$.
Finally for jointly continuity, consider for any $g \in R^-(M)$ the following equality
$$
I_\mu(g_1, g_2) = I_\mu(g_1, g)I_\mu(g, g_2)
$$
This yields the continuity of $ (\mu, g_1, g_2) \mapsto I_\mu(g_1, g_2)$  
and hence proves the proposition. 
\end{proof}

\section{Maximal stretch and Aubry-Mather theory}\label{sec:MC}

\subsection{The maximal stretch} \label{subsection, The maximal stretch}
\hfill\\
In this section, we focus on the \emph{maximal stretch}, which is a major object of study in this note. After introducing the maximal stretch in Subsection \ref{subsection, The maximal stretch} and verifying some simple properties of this quantity, we turn in Subsection \ref{subsection, MCMather} to measure-theoretic objects associated with the study of the maximal stretch: these are \emph{maximal currents}, \emph{maximally stretched measures}, and the \emph{Mather set}; we also discuss the non-measure-theoretic counterpart of the Mather set, namely the \emph{Aubry set}, which is defined using weak supersolutions. Then, in Subsection \ref{subsection, propMCMather}, we investigate further properties of the Mather set. Finally, in Subsection \ref{subsection AubrySetPeierlBarrier}, we discuss a closely related concept, the \emph{Peierls barrier} and its properties.

\begin{definition}\label{def:maximal stretch}
For $g_1,g_2 \in R^-(M)$, the \emph{maximal stretch} $S:  R^-(M) \times R^-(M) \to \mathbb{R}_{\geq 0}$ is defined as
$$
S(g_1,g_2) = \sup_{\mu \in \mathcal{C}(\Gamma)} I_\mu(g_1,g_2)= \sup_{\mu \in \mathcal{C}(\Gamma)} \int a_{g_1,g_2}(v) d\hat{m}^{g_1}_{\mu}.
$$

\end{definition}

\begin{remark} \label{remark, supremumGeodesicStretch}
   We make the following remarks about the maximal stretch.
    \begin{enumerate} 
 
      \item Since $I_{\mu}(g_1,g_2) =I([\mu], g_1,g_2) $ is defined on the compact space of projective currents $\mathcal{PC}(\Gamma)$ (Corollary \ref{cor:s-current}) and $I_{\mu}(g_1,g_2)$ is continuous (Lemma \ref{lemma, GeodesicStretchContinuous}). The supremum is realized by some currents, 
    $$ S(g_1,g_2) = \sup_{\mu \in \mathcal{C}(\Gamma)} I_\mu(g_1,g_2)=\max_{\mu \in \mathcal{C}(\Gamma)} I_\mu(g_1,g_2). $$

\item Since periodic orbit measures are dense in the space of invariant probability measures \cite[Theorem 1]{Sig72} and real multiple of Dirac currents are dense in $\mathcal{C}(\Gamma)$, 

 $$ S(g_1,g_2) = \sup_{[\gamma] \in [\Gamma]} I_{\delta_{[\gamma]}}(g_1,g_2)=\sup_{[\gamma] \in [\Gamma]} \frac{\ell_{g_2}([\gamma])}{\ell_{g_1}([\gamma])} , $$
 where  $\ell_{g}([\gamma])$ is the length of the closed $g$-geodesic $\gamma^{g}$ in the free homotopy $[\gamma]\in [\Gamma]$.
      \end{enumerate}

\end{remark}

We explain our notation of lengths of curves.
\begin{remark}
    The following distinction for lengths of general Lipschitz curves and lengths of geodesics is adopted throughout the note. 
    \begin{itemize}
        \item 
    Given $g\in R^-(M)$, we denote $L_g(\alpha)$ as the $g$-length of a general Lipschitz curve $\alpha$ on $M$. 
    \item
    We denote $\ell_{g}([\gamma])$ as the $g$-length of the uniquely determined closed $g$-geodesic $\gamma^{g}$ up to parametrization in the free homotopy $[\gamma]\in [\Gamma]$, that is, $\ell_{g}([\gamma])=L_g(\gamma^{g})$.
    \end{itemize}
\end{remark}

We have the following continuity result for maximal stretch.
\begin{proposition} \label{propositionGK, GeodesicStretchContinuous}
    The maximal stretch $S:  R^-(M) \times R^-(M) \to \mathbb{R}_{\geq 0}$ given by
$$
S(g_1,g_2) = \sup_{[\mu] \in \mathcal{PC}(\Gamma)} I([\mu], g_1,g_2)= \max_{[\mu] \in \mathcal{PC}(\Gamma)} I([\mu], g_1,g_2)
$$
is continuous.
\end{proposition}

\begin{proof}
The proposition follows from Lemma \ref{lemma, GeodesicStretchContinuous} and Lemma
 \ref{simpleLemma} in Appendix \ref{Appendix simpleLemma}.
 \end{proof}

\subsection{Maximally stretched measures and Aubry-Mather theory}\label{subsection, MCMather}
\hfill\\
In this subsection, we introduce maximal currents, the Mather set, and the Aubry set arising from the study of maximal stretches. The terminology and many ideas come from Aubry--Mather theory for Lagrangian systems (see, for example, \cite{Fat08}). We characterize some dynamical properties of the Mather set and the Aubry set in this and the next subsections.

\subsubsection{Maximal currents}
 We denote the set of \emph{maximal currents} of $g_1$ to $g_2$ as
$$
MC(g_1,g_2) = \{ \mu \in \mathcal{C}(\Gamma) \mid I_\mu(g_1,g_2)  =S(g_1,g_2)\}.
$$

We verify here a simple but interesting feature of maximal currents. It mainly exploits Lemma \ref{lemma,cocycleRelationStretch}.
\begin{proposition} \label{prop,triangleInequality}
For $g_1,g_2 ,g \in R^-(M)$,
we have
$$
S(g_1, g_2) \le S(g_1, g)  S(g, g_2) 
$$
where equality holds if and only if there is a current $\mu_0$ in the set $MC(g_1,g) \cap MC(g,g_2)$. 
In particular, in this case,
$ \mu_0 \in MC(g_1,g_2)$. 

Moreover, if the equality holds, then any current $\mu\in MC(g_1,g_2)$ is also in $MC(g_1,g) \cap MC(g,g_2)$.
\end{proposition}
\begin{proof}
The inequality follows since for all $\mu \in \mathcal{C}(\Gamma)$, we have
$$
I_\mu(g_1, g_2) =I_\mu(g_1, g) I_\mu(g, g_2) \le S(g_1,g) S(g,g_2).
$$
The equality 
$$
S(g_1, g_2) = S(g_1, g)  S(g, g_2) 
$$
is equivalent to
$$
\sup_{ \mu \in  \mathcal{C}(\Gamma) } I_\mu(g_1,g) \sup_{ \mu \in  \mathcal{C}(\Gamma) } I_\mu(g,g_2) =\sup_{ \mu \in  \mathcal{C}(\Gamma) } I_\mu(g_1,g_2).
$$
Suppose $ I_{\mu_0}(g_1,g) = \sup_{ \mu \in  \mathcal{C}(\Gamma) } I_\mu(g_1,g)$ and $ I_{\mu_0}(g,g_2) = \sup_{ \mu \in  \mathcal{C}(\Gamma) } I_\mu(g,g_2)$.
Then
\begin{eqnarray*}
\sup_{ \mu \in  \mathcal{C}(\Gamma) } I_\mu(g_1,g_2) &=& \sup_{ \mu \in  \mathcal{C}(\Gamma) } I_\mu(g_1,g) I_\mu(g,g_2) \\
&\le&\sup_{ \mu \in  \mathcal{C}(\Gamma) } I_\mu(g_1,g) \sup_{ \mu \in  \mathcal{C}(\Gamma) }I_\mu(g,g_2)\\
&=& I_{\mu_0}(g_1,g) I_{\mu_0}(g,g_2) \\
&=&I_{\mu_0}(g_1,g_2)\le \sup_{ \mu \in  \mathcal{C}(\Gamma) }I_\mu(g_1,g_2).
\end{eqnarray*}
In particular, $\sup_{ \mu \in  \mathcal{C}(\Gamma) }I_\mu(g_1,g_2) =I_{\mu_0}(g_1,g_2)$.

On the other hand, let us suppose $S(g_1, g_2) = S(g_1, g) S(g, g_2) $ and assume $\sup_{ \mu \in  \mathcal{C}(\Gamma) }I_\mu(g_1,g_2) =I_{\mu_0}(g_1,g_2)$
for some $\mu_0 \in   \mathcal{C}(\Gamma)$. Then
$$
S(g_1,g_2) = I_{\mu_0}(g_1,g_2) =  I_{\mu_0}(g_1,g)   I_{\mu_0}(g,g_2) =S(g_1,g) S(g,g_2) 
$$
which implies the other claim.
\end{proof}

Recall $\iota: \partial^{(2)} \widetilde{M}  \to \partial^{(2)} \widetilde{M}$ is the
flip map exchanging two factors of $\partial^{(2)} \widetilde{M}$.

\begin{lemma}
    Suppose $\mu_0$ is a maximal current.  Then $\iota^*\mu_0$ is also a maximal current. Therefore we can symmetrize $\mu_0$ to obtain a flip invariant maximal current as $\frac{\mu_0+ \iota^*\mu_0}{2}$.
\end{lemma}  
\begin{proof}
    Since $$S(g_1,g_2)= I_{\mu_0}(g_1,g_2)= \sup_{[\gamma]\in[\Gamma]} I_{\delta_{[\gamma]}} (g_1,g_2).$$
    If $\mu_0$ is a maximal current and if we write $\mu_0=\lim_{n\to \infty} \delta_{[\gamma_n]}$ in weak-* topology, then $\iota^*\mu_0$ which is a limit of $\delta_{[\gamma_n^{-1}]}=\iota^* \delta_{[\gamma_n]}$, is also a maximal current. This follows from the observation that
    $\ell_{g_i}([\gamma])=\ell_{g_i}([\gamma^{-1}])$ and hence $I_{\delta_{[\gamma]}} (g_1,g_2)=I_{\delta_{[\gamma^{-1}]}} (g_1,g_2)$.
\end{proof}
For maximally stretched measures introduced next, we can also symmetrize them and obtain reflexive measures (see footnotes in the introduction). However, we do not emphasize this feature in this note, and we work with unsymmetrized objects.

\subsubsection{Maximally stretched measures and the Mather set}

We describe the measure-theoretic counterpart of maximal currents. For $g_1,g_2 \in R^-(M)$, the maximal stretch of $g_1$ to $g_2$ is also
$$
S(g_1,g_2) = \max_{m \in \mathcal{M}^{1}(\phi^{g_1})} I_m(g_1,g_2).
$$

\begin{definition}
Let $g_1, g_2$ be two Riemannian metrics in $R^-(M)$. The set of \emph{maximally stretched measures} of $g_1$ to $g_2$ is defined by
$$
MS(g_1,g_2) = \{ m \in \mathcal{M}^{1}(\phi^{g_1}) \mid I_m(g_1,g_2)  =S(g_1,g_2)\}.
$$
The \emph{Mather set} of $g_1$ to $g_2$ is defined by
$$
\mathcal{M}(g_1,g_2) :=\overline{ \bigcup_{m\in MS(g_1,g_2)} \supp m}.$$
\end{definition}

Using maximal currents, the Mather set of $g_1$ to $g_2$ can be equivalently written as

$$
\mathcal M(g_1, g_2) = \overline{\bigcup_{\mu \in MC(g_1,g_2)}\supp \hat{m}^{g_1}_{\mu} }.
$$

Mather sets in the Teichm\"uller space always project to geodesic laminations; In contrast, for metrics in $R^-(M)$, the Mather set is in general not related to geodesic laminations (see Example \ref{MatherNotLamination}). Nevertheless, some structural features of the Teichm\"uller space are still preserved in this general setting (Theorem \ref{corollary, LinearConjugacy} and Proposition \ref{thm, EmtpyInterior}). Before turning to an analysis of further properties of the Mather set, we complete the picture of the Aubry-Mather theory in our setting by introducing another closely related object, the \emph{Aubry set}.  
 
\subsubsection{Supersolutions and the Aubry set} \label{subsection,SupersolutionAubry}

We first consider \emph{(strong) supersolutions}  which in our context were defined in \cite{LT05}. They are analogous to weak KAM subsolutions studied in Fathi's monograph \cite{Fat08}. 

\begin{definition}\label{def,StrongSuper}
    Let $g_1, g_2  \in R^-(M)$. We call a H\"older continuous function $u:  S^{g_1}M \to \R$ a \emph{strong supersolution} for $a_{g_1,g_2}$, if it is smooth along the flow direction of the geodesic flow $\phi^{g_1}$ and
 for all $v\in S^{g_1}M$,
    $$
S(g_1, g_2) -a_{g_1, g_2}(v) + X_{g_1} u(v) \ge 0,
$$
where $X_{g_1}$ is the infinitesimal generator of the geodesic flow $\phi^{g_1}$ and $X_{g_1}u$ is H\"older continuous.

\end{definition}

We also introduce a weaker version of supersolutions as follows. They will be used in later subsections.
\begin{definition}\label{def,weakSuper}
    We say a continuous function $u:S^{g_1}M \to \mathbb{R}$ is a \emph{weak supersolution} for $a_{g_1,g_2}$, if for all $t_1 \le t_2$, we have
    $$
u(\phi_{t_2}^{g_1}v)- u(\phi_{t_1}^{g_1}v) \ge \int_{t_1}^{t_2} a_{g_1, g_2} (\phi_s^{g_1}v )ds - (t_2-t_1) S(g_1, g_2) .
$$
\end{definition}

\begin{proposition}\label{prop:subsolution}
Let $g_1, g_2  \in R^-(M)$. Then there exists a strong supersolution for $a_{g_1,g_2}$.
\end{proposition}
\begin{proof}
For all $\mu \in \mathcal{C}(\Gamma)$, we know $$\int_{S^{g_1}M} a_{g_1,g_2}(v) d\hat{m}^{g_1}_{\mu} \le S(g_1,g_2) .$$ 

As a consequence, the integral of  $a_{g_1,g_2} -S(g_1,g_2) $ is nonpositive along
all periodic orbits of the geodesic flow $\phi^{g_1}$. Then from \cite[Theorem 1]{LT05}, we obtain that there exists a H\"older continuous function $u:  S^{g_1}M \to \R$ smooth along the flow direction of $\phi^{g_1}$ so that for any $v\in S^{g_1}M$,
 $$X_{g_1} u(v) \ge  a_{g_1,g_2}(v)- S(g_1,g_2). $$
 Therefore $u$ is a supersolution for $a_{g_1,g_2}$.
\end{proof}

Motivated by the weak KAM-theory, we define the following.
\begin{definition}
Let $g_1, g_2  \in R^-(M)$ and let $u:  S^{g_1}M \to \R$ be a weak supersolution for $a_{g_1,g_2}$. We define the set
\begin{align*}
\mathcal A_u(g_1, g_2) = \{ v \in S^{g_1}M \mid u(\phi_{t_2}^{g_1}v)  &= \int_{t_1}^{t_2} a_{g_1, g_2} (\phi_s^{g_1}v )ds - (t_2-t_1) S(g_1, g_2) \\
&+u(\phi_{t_1}^{g_1}v),
 \;\text{for all $t_1,t_2\in\R$ with $t_1 \leq t_2$} \}
\end{align*}
as the  \emph{Aubry set} of $u$.\\ 

The \emph{Aubry set} is then  $$\mathcal{A}(g_1,g_2)= \bigcap_u \mathcal{A}_u(g_1,g_2),$$
where $u$ is any weak supersolution.
\end{definition}

\begin{remark}  \label{remark, AubrySet}
The following statements are clear from the definitions of $\mathcal A_u(g_1, g_2)$ and $\mathcal{A}(g_1,g_2)$.
\begin{enumerate}
\item  $\mathcal A_u(g_1, g_2)$ and $\mathcal{A}(g_1,g_2)$ are compact $\phi^{g_1}_t$-invariant subsets of $S^{g_1}M$. 
\item If we have for $t_1 \leq t_2$
\begin{equation*}
u(\phi_{t_2}^{g_1}(v))= \int_{t_1}^{t_2}  a_{g_1, g_2} (\phi_s^{g_1}(v) )ds - (t_2-t_1) S(g_1, g_2) 
+u(\phi_{t_1}^{g_1}(v)).
\end{equation*}
Then for any subinterval $[s_1,s_2] \subset [t_1,t_2]$, the above equality still holds  by replacing $t_j$ by $s_j$ for $j=1,2$.
\end{enumerate} 
\end{remark}

We would like to relate the Aubry set of a supersolution to the Mather set. For this, it is more convenient to work with strong supersolutions. 
\begin{proposition} \label{prop, AubryMather}
Let $g_1, g_2  \in R^-(M)$ and let $u:  S^{g_1}M \to \R$ be a strong supersolution. Then the Mather set $\mathcal M(g_1, g_2) $ is contained in the Aubry set $\mathcal A_u(g_1, g_2)$ of $u$.
\end{proposition}

\begin{proof}
Consider $v \in \mathcal M(g_1, g_2)$. Then there exists $m_n\in MS(g_1,g_2)$ and a sequence $v_n \in \supp m_n$ so that $v_n$ converges to $v$. By assumption, the function
$f \colon  S^{g_1} M \to \mathbb{R}$ given by
$$f(v):= S(g_1, g_2) -a_{g_1, g_2}(v) + X_{g_1} u(v) $$
is non-negative. Therefore,
$$
\int_{S^{g_1}M} f  d m_n = S(g_1, g_2) -\int_{S^{g_1}M} a_{g_1, g_2} d m_n =0.
$$

Since the integrand is continuous and non-negative, we have that $f$ vanishes identically on the invariant set $\supp m_n$.
In particular, for $t_1\leq t_2$,
$$
u(\phi_{t_2}^{g_1}(v_n)) = \int_{t_1}^{t_2} a_{g_1, g_2} (\phi_s^{g_1}(v_n) )ds - (t_2-t_1) S(g_1, g_2) +u(\phi_{t_1}^{g_1}(v_n)) 
$$
and by continuity, $v$ is also contained in the Aubry set $A_u(g_1, g_2)$ of $u$.
 \end{proof}

The next theorem is the main result of this subsection: we show that one can always find an orbit equivalence between $S^{g_1}M$ and $S^{g_2}M$ so that the time change is maximally linear on the Aubry set of a weak supersolution.

Recall we introduced the time change function $\tau: S^{g_1}M \times \mathbb{R} \to \mathbb{R}$ (Definition \ref{def:timechange}) as
\begin{align*}
    \tau(v,t) =  b^{g_2}_{v^{g_1}_+} ( \pi(v), \pi(\phi_t^{g_1}(v))).
\end{align*}

\begin{theorem}
     \label{corollary, LinearConjugacy}
Let $g_1, g_2  \in R^-(M)$ and let $u:  S^{g_1}M \to \R$ be a weak supersolution.
Let $G: S^{g_1}M \to S^{g_2}M$ be a H\"older orbit equivalence given by 
$$G(v) := \phi^{g_2}_{- u(v)}\Psi_{g_1, g_2}(v).$$ Then 
$$
\phi^{g_2}_{\tau_u(v,t)} G(v) = G(\phi_t^{g_1}(v)),
$$
where $\tau_u(v,t):=\tau(v,t) -u(\phi_t^{g_1}(v))+u(v) \le tS(g_1, g_2)$.
In particular, the geodesic flows $ \phi^{g_1}$ and  $ \phi^{g_2}$ are homothetic on the Aubry set $\mathcal A_u(g_1, g_2) $, in the sense that
$$
 \phi^{g_2}_{ S(g_1,g_2)t} (G(v))  = G( \phi_t^{g_1}(v)),
 $$
for all $v \in \mathcal A_u(g_1, g_2)$.

Furthermore, for all $g_1$-periodic orbits  $\phi_t^{g_1}(v)$ with $v \in \mathcal A_u(g_1, g_2)$ and with period $\ell_{g_1}(v)$,  the orbit $  \phi^{g_2}_{t} (G(v)) $ is periodic with period $ \ell_{g_2}(G(v))$ that satisfies
 $$ \ell_{g_2}(G(v)) = S(g_1, g_2) \ell_{g_1}(v).$$
\end{theorem}
\begin{proof}
In Lemma \ref{lem:conj}, we constructed a H\"older continuous homeomorphism $\Psi_{g_1,g_2}: S^{g_1}M \to S^{g_2}M $ such that
$$
\phi^{g_2}_{\tau(v,t)}\Psi_{g_1, g_2}(v) = \Psi_{g_1, g_2}( \phi_t^{g_1}(v)),
$$
where 
$$
\tau(v,t) =  b^{g_2}_{v^{g_1}_+} ( \pi(v), \pi(\phi_t^{g_1}(v)) )  =\int_0^t a_{g_1, g_2} (\phi_s^{g_1}(v) )ds,
$$
and  
$$
 a_{g_1, g_2} (v)= g_2( B^{g_2}(\pi (v), v^{g_1}_+) , v).
$$

  This yields
  \begin{align*}
  \phi^{g_2}_{- u(v)}   \phi^{g_2}_{\tau(v,t)} \Psi_{g_1, g_2}(v)
  &= \phi^{g_2}_{-u(v)}\Psi_{g_1, g_2}( \phi_t^{g_1}(v))\\
   &= \phi^{g_2}_{u(\phi_t^{g_1}(v))- u(v)}  \phi^{g_2}_{-u(\phi_t^{g_1}(v))} \Psi_{g_1, g_2}( \phi_t^{g_1}(v))\\
    &=\phi^{g_2}_{u(\phi_t^{g_1}(v))- u(v)} G(( \phi_t^{g_1}(v)).
\end{align*}
And therefore,
$$
\phi^{g_2}_{\tau(v,t)- u( \phi_t^{g_1}(v)) + u(v) }\phi^{g_2}_{- u(v)}\Psi_{g_1, g_2}(v) =      G(( \phi_t^{g_1}(v)).  
$$ 
Recall the definition of the orbit equivalence $G$,
$$
\phi^{g_2}_{\tau(v,t)- u( \phi_t^{g_1}(v)) + u(v) }G(v) =      G(( \phi_t^{g_1}(v)).  
$$  
  In particular, for $v \in \mathcal A_u(g_1, g_2)$, we have
  $$
  \tau(v,t)= u(\phi_t^{g_1}(v))- u(v) + t S(g_1,g_2).
  $$
 This concludes, for all $v \in \mathcal A_u(g_1, g_2)$ and for all $t\in \mathbb{R}$,  
 $$
\phi^{g_2}_{ S(g_1,g_2)t}G(v) = G(( \phi_t^{g_1}(v)). 
$$
\end{proof}

\subsection{More properties of maximally stretched measures, maximal currents and the Mather set} \label{subsection, propMCMather}
\hfill\\
We discuss in this subsection more properties of maximally stretched measures, maximal currents and the Mather set from different aspects.

\subsubsection{Ergodic maximally stretched measures}

This subsection discuss ergodic theory in our setting from the perspective of invariant measures. Parallel statements extend to geodesic currents.

A first observation is that  the space of maximally stretched measures or  the space of maximal currents forms a convex set as follows.
\begin{lemma}
   Let $g_1, g_2$ be two Riemannian metrics in $R^-(M)$. The set of maximally stretched measures is a closed convex subset of $\mathcal{M}^1(\phi^{g_1})$. Similarly, the set of maximal currents $MC(g_1,g_2)$ is a closed convex subset in $\mathcal{C}(\Gamma)$.
\end{lemma}
\begin{proof}
    Suppose $m_1,m_2\in MS(g_1,g_2)$ and $m_0=tm_1+(1-t)m_2$ for some $t\in[0,1]$. Then from the properties of maximal currents for $m_1,m_2$, it follows,
    $$\int a_{g_1,g_2}(v) dm_0=S(g_1,g_2).$$
    Therefore $m_0\in MS(g_1,g_2)$ and $MS(g_1,g_2)$ is a convex subset of $\mathcal{M}^1(\phi^{g_1})$. The closedness condition of $MS(g_1,g_2)$ follows from continuity of the geodesic stretch (Lemma \ref{lemma, GeodesicStretchContinuous}).
    
\end{proof}

We then discuss the existence of ergodic maximally stretched measures, or ergodic maximal currents. A current $\mu$ is said to be \emph{ergodic} if the action of $\Gamma$ on $\partial^{(2)}M$ is ergodic with respect to $\mu$. Given a Riemannian metric $g\in R^{-}(M)$, this is equivalent to the fact that the geodesic flow $\phi^{g_1}$on $S^{g_1}M$ is ergodic with respect to its associated invariant measure $m^{g}_{\mu}$. It is a property independent of chosen metrics (see \cite{Su79}, \cite[Proposition 2.3]{ST21}). Therefore, we can consider the subset of ergodic currents in $\mathcal{C} (\Gamma)$, denoted as $\mathcal{C}_{erg}(\Gamma)$. Similarly, the subset of ergodic invariant probability measures is denoted as $\mathcal{M}^1_{erg}(\phi^{g_1})$.

The following is proved in \cite{GKL22} using the Choquet representation theorem.

\begin{lemma} [\cite{GKL22}, Lemma 5.5] \label{lemma, ergodicMaximalCurrent}
Let $g_1, g_2$ be two Riemannian metrics in $R^-(M)$. Then
$$S(g_1,g_2)=\sup_{\mathcal{M}^{1}_{erg}(\phi^{g_1})} I_{m}(g_1,g_2)=\sup_{\mathcal{C}_{erg}(\Gamma)} I_{\mu}(g_1,g_2).$$
Moreover, there exists a maximally stretched measure $m_0\in MS(g_1,g_2)$ that is ergodic. Similarly, there also exists a maximal current $\mu_0$ in $MC(g_1,g_2)$ that is ergodic.
\end{lemma} 
\begin{proof}
    Suppose for $m_0\in\mathcal{M}^{1}(\phi^{g_1})$, we have
    
      $$S(g_1,g_2)=\int a_{g_1,g_2}(v) dm_0.$$
     
    By the Choquet representation theorem (see, for example, \cite[Section 3]{Ph66}), for the compact convex metric space $\mathcal{M}^{1}(\phi^{g_1})$, there exists a probability measure $\tau_0$ in $\mathcal{M}^{1}(\phi^{g_1})$ which represents $m_0$ and is supported on the set $\mathcal{M}^{1}_{erg}(\phi^{g_1})$ of extreme points of $\mathcal{M}^{1}(\phi^{g_1})$. Therefore 
    \begin{align*}
         S(g_1,g_2)&=  \int_{\mathcal{M}^{1}_{erg}(\phi^{g_1})} I_m(g_1,g_2) d\tau_0(m)\\
           &\leq  \sup_{m \in \mathcal{M}^{1}_{erg}(\phi^{g_1})}I_m(g_1,g_2)\le S(g_1,g_2).
    \end{align*}
    This implies that the equality holds. Moreover, for any ergodic $m$ that is in the support of $\tau_0$, we have $S(g_1,g_2) =I_{m}(g_1,g_2).$
    
    The proof works similarly for ergodic maximal currents.
\end{proof}

\subsubsection{The Mather set is nowhere dense}

We prove in this subsection that for generic metrics in $R^{-}(M)$, the Mather set is nowhere dense. Precisely, the Mather set is not nowhere dense only when the marked length spectra of $g_1$ and $g_2$ are \emph{proportional},  that is, there exists a constant $C>0$ such that $\ell_{g_2}([\gamma])=C \ell_{g_1}([\gamma])$ for all $[\gamma]\in[\Gamma]$.  

Denote by $h_{top} (\phi^{g})$ the topological entropy of the geodesic flow $\phi^g$ for a Riemannian metric $g \in R^{-}(M)$. We first recall a related theorem from previous works of the second author.

\begin{theorem}[\cite{Kn95} Theorem 1.2, \cite{GKL22} Proposition 5.4] \label{thm, proportMLS}
   Given two Riemannian metrics $g_1, g_2  \in R^{-}(M)$, we have
    \begin{equation*}\label{eqtn Kn95}
        \frac{h_{top} (\phi^{g_2})}{h_{top} (\phi^{g_1}) }S(g_1,g_2) \ge 1.
    \end{equation*}
    The equality holds precisely when the marked length spectra of $g_1,g_2$ are proportional.
\end{theorem}

Our next key theorem shows that the proportionality of marked length spectra also characterize the topological structure of the Mather set.

\begin{theorem} \label{thm, EmtpyInterior}
   Given $g_1,g_2\in R^-(M)$, the following dichotomy holds:
   \begin{itemize}
       \item If the marked length spectra are not proportional, then the interior of the Mather set $\mathcal{M}(g_1,g_2)$ is empty.
       \item If the marked length spectra are proportional, then the Mather set $\mathcal{M}(g_1,g_2)=S^{g_1}M$.
   \end{itemize}
\end{theorem}

\begin{proof}
    Note that for each $v\in \mathcal{M}(g_1,g_2)$, we have for all $t\in \mathbb{R}$,
    $$\phi^{g_2}_{S(g_1,g_2)t}\circ G(v)=G \circ \phi^{g_1}_{t}(v),$$
    where $G: S^{g_1}M \to S^{g_2}M$ is the orbit equivalence defined in Theorem \ref{corollary, LinearConjugacy}.
        Suppose for the sake of contradiction,
      the  Mather set $\mathcal{M}(g_1,g_2)$  contains an open set $U\subset S^{g_1}M$.  
Since the geodesic flow $\phi^{g_1}$ is topologically transitive, there exists a dense orbit $\phi^{g_1}_t(v_0)$ with $v_0 \in U$. Therefore, the continuity of $G$ implies
  
  $$\phi^{g_2}_{S(g_1,g_2)t}\circ G(v)=G \circ \phi^{g_1}_{t}(v),$$
for any $v\in S^{g_1}M$.

    This implies that for any $[\gamma]\in[\Gamma]$, we have $\ell_{g_2}([\gamma])= S(g_1,g_2) \ell_{g_1}([\gamma])$. Therefore, the marked length spectra are proportional, which leads to a contradiction.

On the other hand, when the marked length spectra are proportional, that is, there exists $C>0$ so that $\ell_{g_2}([\gamma])=C\ell_{g_1}([\gamma])$ for all $[\gamma]\in[\Gamma]$. Then
$ S(g_1,g_2)=\sup\limits_{[\gamma]\in[\Gamma]} \frac{\ell_{g_2}([\gamma])}{\ell_{g_1}([\gamma])}=C$. Therefore, every Dirac measure obtained from a periodic orbit is a maximally stretched measure. Because periodic orbits of $\phi^{g_1}$ are dense in $S^{g_1}M$, it follows
 $$\mathcal{M}(g_1,g_2) =\overline{ \bigcup_{m\in MS(g_1,g_2)} \supp m}= S^{g_1}M.$$
Moreover, in this case, every $\phi^{g_1}$-invariant measure is a maximally stretched measure by Sigmund's Theorem (\cite[Theorem 1]{Sig72}).
\end{proof}

\subsubsection{Minimal subsets of the Mather set} \label{subsubsection,minimalComponents}

Recall in topological dynamical system, a \emph{minimal subset} is a nonempty, closed and invariant subset of a dynamical system such that no proper subset has these three properties. A minimal subset of a dynamical system always exists due to Zorn Lemma. 

Motivated by minimal laminations in Teichmüller theory (see, for example, \cite[Chaper I.4]{CEG06}), we study  minimal subsets of the Mather set $\mathcal{M}(g_1,g_2)$. We show 
in Proposition \ref{Prop, OneOrUncountable} that a minimal subset in the Mather set $\mathcal{M}(g_1,g_2)$ either contains a single closed orbit or it must contain uncountably many orbits (leaves).

To prove this, we introduce the definition of an \emph{isolated orbit}. We say an orbit $\mathcal{O}=\{\phi_t^{g_1}(v)\}_{t\in\R}$ of a minimal subset $\Lambda$ is \emph{isolated}, if for any $w\in\mathcal{O}$, there exists $\varepsilon>0$ and a metric ball $B_w(\varepsilon)\subset S^{g_1} M$ centered at $w$ of the radius $\varepsilon$  with respect to the Sasaki metric $g_1^s$ such that $B_w(\varepsilon)\cap \Lambda$ is a connected geodesic segment homeomorphic to the unit interval.

\begin{proposition} \label{Prop, OneOrUncountable}
   Let $g_1, g_2$ be two Riemannian metrics in $R^-(M)$. Suppose that $\Lambda$ is a minimal subset of $\mathcal{M}(g_1,g_2)$. Then 
    \begin{itemize}
        \item either $\Lambda$ is the orbit of a closed $g_1$-geodesic on $ S^{g_1} M$;
        
        \item or $\Lambda$ contains uncountably many orbits. Moreover, for every point $v\in\Lambda$, we can find a transversal section $\mathcal{P}_v$ to the flow $\phi^{g_1}$ at $v$ so that $\mathcal{P}_v\cap \Lambda$ is a perfect set.
    \end{itemize}
\end{proposition}
  \begin{figure}[!htb]
   \centering    \includegraphics[width=60mm,scale=0.5]{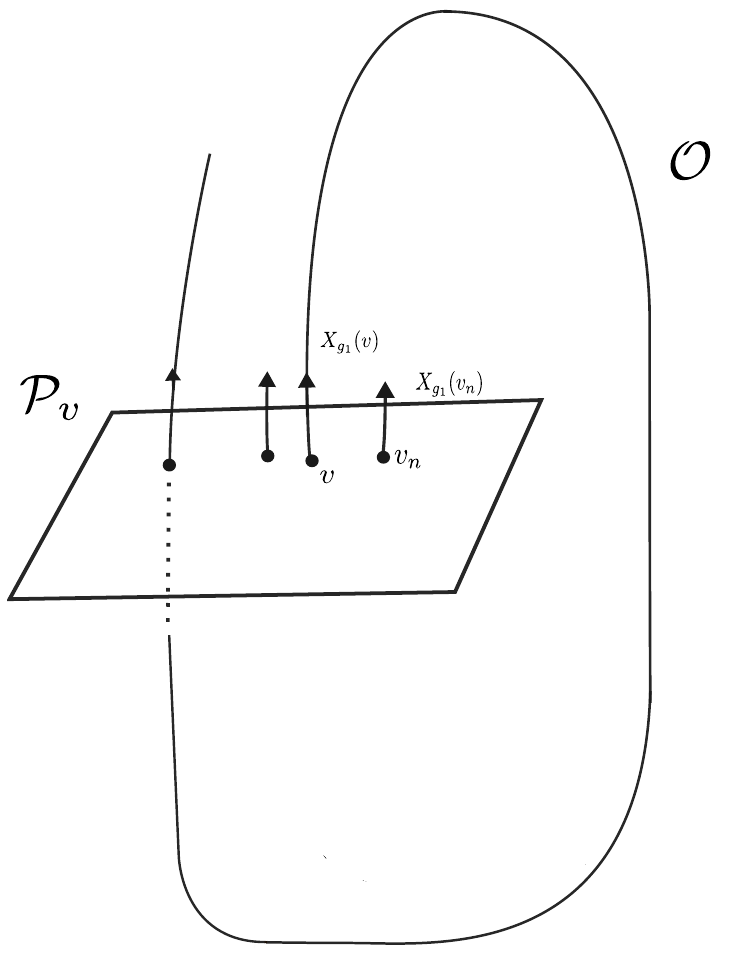}
    \caption{A figure for Property (A). The orbit $\mathcal{O}$ is a non-isolated leaf and $\mathcal{P}_v$ is a traversal to $\mathcal{O}$. } 
    \label{fig: figure1}
\end{figure}
\begin{proof}
    Suppose $\mathcal{O} \subset \Lambda$ is a nonempty isolated orbit. Take a (countable) cover of $\mathcal{O}$ by metric balls $B_{w_i}(\varepsilon_i)$ so that each $B_{w_i}(\varepsilon_i)\cap \Lambda$ is a connected geodesic segment on $\mathcal{O}$ as defined before the proposition. Then 
    $$\Lambda \backslash \mathcal{O} =\Lambda \backslash (\cup_{i} B_{w_i}(\varepsilon_i) ) $$ is a closed invariant set. If $\Lambda \backslash \mathcal{O} $ is empty, i.e. $\Lambda=\mathcal{O} $,  then the isolated orbit $\mathcal{O}$ can not be bi-infinite. This is because $S^{g_1}M$ is compact. A  bi-infinite orbit of the geodesic flow $\phi^{g_1}$ has an accumulation point on $S^{g_1}M$ which contradicts the isolated orbit condition.  Therefore $\mathcal{O}$ must be the orbit of a closed $g_1$-geodesic. On the other hand, if $\Lambda \neq \mathcal{O} $, then it leads to a contradiction with the fact that $\Lambda$ is minimal.
Next we assume all orbits of $\Lambda$ are not isolated. We take one orbit $\mathcal{O}$. Since $\mathcal{O}$ is not isolated, there exists a point $v$ in $\mathcal{O}$ that satisfies the following property: it is an accumulation point of $\Lambda \cap \mathcal{P}_v$ in $\Lambda$, if $\mathcal{P}_v$ is any transversal section to $\mathcal{O}$ at $v$ in $S^{g_1}M$ (i.e. $T_v S^{g_1}M = T_v \mathcal{P}_v \bigoplus X_{g_1}(v)$.  See Figure \ref{fig: figure1}.

    We call the above property the \emph{Property (A)}. The point $v\in \mathcal{O}$  satisfies the Property (A).  Therefore, there exists $ \{v_n\}\subset \Lambda \cap \mathcal{P}_v $ such that $v_n \to v$.  We then want to show  that Property (A) holds for any point in $\Lambda$ (not only at $v$). We start from showing that it holds for any point on $\mathcal{O}$. Because every point $w$ on $\mathcal{O}$ is a translation of $v$ as $w=\phi_T^{g_1}(v)$. For the fixed real number $T$, the map $\phi_T^{g_1}$ is uniformly continuous on the compact space $S^{g_1}M$.  Therefore $d_{g_1^s}(w_n, w)= d_{g_1^s}(\phi_T^{g_1}(v_n),\phi_T^{g_1}(v)) \to 0$. This implies that every point $w$ on $\mathcal{O}$ satsifies Property (A): for any transversal $\mathcal{P}_w$  to $w$, we find $w$ is an accumulation point of $\mathcal{P}_w\cap \Lambda$ in $\Lambda$. Since every orbit of $\Lambda$ is not isolated, we obtain any point of $\Lambda$ satisfies the Property (A) in the same manner.

    Now we go back to the point $v$ which is the limit of $v_n$. Taking $\mathcal{P}_v $ small enough, then we can make sure that it is transversal to the orbit of $v_n$ at $v_n$ for $n$ big enough (i.e. $T_{v_n} S^{g_1}M = T_{v_n} \mathcal{P}_v \bigoplus X_{g_1}(v_n)$ ). Therefore since $v_n$ satisfies the Property (A), it is also an accumulation point of $\Lambda \cap \mathcal{P}_v$ in $\Lambda$. As every point of $\Lambda \cap \mathcal{P}_v$ in $\Lambda$ is an accumulation point, the set $\mathcal{P}_v\cap \Lambda$ is a perfect set. Lastly, since the obit $\mathcal{O}$ can only intersects $\mathcal{P}_v$ countably many times, $\Lambda$ must have uncountably many orbits.
    
\end{proof}

\subsection{The Aubry set and the Peierls barrier} \label{subsection AubrySetPeierlBarrier}
We discuss in this subsection the \emph{Peierls barrier} and its properties. This concept is also motivated from Aubry-Mather theory and Lagrange dynamics (see, for example, \cite{Fat08}). In particular, the zero level set of the Peierls barrier plays an important role in this subsection.

\subsubsection{The Peierls barrier}

\begin{definition} \label{def, PeierlsBarrier}

Given $g_{1}, g_{2} \in R^{-}(M)$, we define the \emph{Peierls barrier} $\mathcal{H}_{g_{1}, g_{2}}: S^{g_{1}} M \times S^{g_{1}} M \rightarrow \mathbb{R} \cup \{-\infty\}$ by
\begin{align*}
\mathcal{H}_{g_{1}, g_{2}}(v_{1}, v_{2}):= 
 \lim _{\varepsilon \rightarrow 0} 
 \lim_{T\to \infty} \sup\bigg\{ &\int_{0}^{t} a_{g_{1}, g_{2}}(\phi_{s}^{g_{1}} v) d s-t \cdot S(g_{1}, g_{2}) \mid   \\
 & d_{g_{1}^{s}}\left(v, v_{1}\right)<\varepsilon, d_{g_{1}^{s}}\left(\phi_{t}^{g_{1}}v, v_{2}\right)<\varepsilon,t\geq T \bigg\}.
\end{align*}

\end{definition}

\begin{remark}
We make the following remarks about the definition of the Peierls barrier.
    \begin{enumerate}
        \item The well-definedness of the Peierls barrier follows from the topological transitivity of the geodesic flow $\phi_{t}^{g_{1}}$. 
        \item The value $-\infty$ can possibly be 
achieved. For example, consider $[\gamma_0]\in[\Gamma]$ so that $\frac{\ell_{g_2}([\gamma_0])}{\ell_{g_1}([\gamma_0])}<S(g_1,g_2)$, and take $v_0\in S^{g_1}M$ tangential to $\gamma_0^{g_1}$, then  
$\mathcal{H}_{g_{1}, g_{2}}(v_{0},v_0)=-\infty$. 
\item In contrast, as we will soon see in Proposition \ref{prop,ManeAubrySet}, when $v_0\in\mathcal{M}(g_1,g_2)$, the Peierls barrier is finite and verifies $\mathcal{H}_{g_{1}, g_{2}}(v_{0},v_0)=0$.

    \end{enumerate}
\end{remark}

 In its original form, the Peierls barrier describes the energy cost, in solid-state physics, associated with the transition from a metallic phase to an insulating phase (see \cite{Au78}).
For our Riemannian geometry setting, the Peierls barrier characterizes the ``infinite-time optimal cost'' of the potential $a_{g_1,g_2}-S(g_1,g_2)$ from $v_1$ to $v_2$.
It provides us a way to detect those optimal orbits of the potential  $a_{g_1,g_2}-S(g_1,g_2)$ which can not support maximally stretched measures and which are independent of the choice of a weak supersolution defining an Aubry set. 

Because the Aubry set, being the intersection of the Aubry sets $\mathcal{A}_u(g_1,g_2)$ associated with all weak supersolutions $u$, is also independent of a certain weak supersolution. We would like to seek for the relation between the Peierls barrier and the Aubry set. More precisely, our goal in this subsection is to relate the zero level set of the Peierls barrier,
\[
\mathcal{H}_0(g_1,g_2)
:=\{v\in S^{g_1}M \mid \mathcal{H}_{g_1,g_2}(v,v)=0\},
\]
to the Aubry set $\mathcal{A}(g_1,g_2)$. To this end, it is convenient to work with a natural \emph{continuous additive cocycle} associated with weak supersolutions.

A weak supersolution $u$ induces a continuous function $c_u:S^{g_1}M \times \mathbb{R}_{\geq 0}  \to \mathbb{R}_{\leq0}$ by letting
$$c_u(v,t):=\int_{0}^{t} a_{g_1, g_2} (\phi_s^{g_1}v )ds - t\cdot S(g_1, g_2)- \big(u(\phi^{g_1}_tv)-u(v)\big).$$
It is easy to check that $c_u$ is nonpositive and verifies the \emph{additive cocycle} condition, that is, for any $v\in S^{g_1}M$ and $t_1,t_2\in \mathbb{R}_{\geq0}$, 
$$c_u(v,t_1+t_2)=c_u(v,t_1)+c_u(\phi^{g_1}_{t_1}v,t_2).$$

\begin{remark}
    We make the following remarks for the cocycle $c_u(v,t)$.
    \begin{itemize}
        \item As a consequence of nonpositiveness, the cocycle $c_u(v,t)$ is monotonically decreasing in $t\geq 0$ with $c_u(v,0)=0$. 
        \item  When $u$ is a strong supersolution, we have $c_u(v,t)= \int_{0}^{t} f\left(\phi_{s}^{g_{1}} v\right) d s$, where $f(v)=a_{g_1,g_2}(v) -S(g_1,g_2)- X_{g_1}u(v)$.
    \end{itemize}
\end{remark}

\begin{lemma}\label{lemma, ManePotentialLessZero}  
 Let $g_1, g_2$ be two Riemannian metrics in $R^-(M)$. For any $v_{0} \in S^{g_{1}} M$, the Peierls barrier satisfies
$$
\mathcal{H}_{g_{1}, g_{2}}\left(v_{0}, v_{0}\right) \leq 0.
$$
\end{lemma}
\begin{proof}
    Let $u: S^{g_{1}} M \rightarrow \mathbb{R}$ be a weak supersolution for $a_{g_1,g_2}$ and $c_u$ be its associated cocycle so that for any $v \in S^{g_{1}} M$,
    \begin{align*}
\int_{0}^{t} a_{g_{1}, g_{2}}\left(\phi_{s}^{g_{1}} v\right) d s-t\cdot S\left(g_{1}, g_{2}\right) & =c_u(v,t) +u\left(\phi_{t}^{g_{1}} v\right)-u(v). 
\end{align*}

This implies
\begin{align*}
\mathcal{H}_{g_{1}, g_{2}}(v_{0}, v_{0}) &= 
 \lim _{\varepsilon \rightarrow 0} 
 \lim_{T\to \infty} \sup\bigg\{  \int_{0}^{t} a_{g_{1}, g_{2}}(\phi_{s}^{g_{1}} v) d s- t \cdot S(g_{1}, g_{2}) \mid  \\
 & \qquad \qquad  \qquad d_{g_{1}^{s}}\left(v, v_{0}\right)<\varepsilon, d_{g_{1}^{s}}\left(\phi_{t}^{g_{1}}v, v_{0}\right)<\varepsilon , t\geq T\bigg\}\\
&=\lim _{\varepsilon \rightarrow 0}\lim_{T\to \infty} \sup  \bigg\{c_u(v,t)+u\left(\phi_{t}^{g_{1}}v \right)-u(v) \mid \\
&\qquad \qquad  \qquad
 d_{g_{1}^{s}}\left(v, v_{0}\right)<\varepsilon, d_{g_{1}^{s}}\left(\phi_{t}^{g_{1}}v, v_{0}\right)<\varepsilon, t\geq T \bigg \} .
\end{align*}

Since $u$ is a uniform continuous function and $c_u$ is nonpositive, it follows

\begin{align*}
 \mathcal{H}_{g_{1}, g_{2}}\left(v_{0}, v_{0}\right) &=
 \lim _{\varepsilon \rightarrow 0} \lim_{T\to \infty} \sup\bigg\{c_u(v,t) \mid  d_{g_{1}^{s}}\left(v, v_{0}\right)<\varepsilon, d_{g_{1}^{s}}\left(\phi_{t}^{g_{1}}v, v_{0}\right)<\varepsilon, t\geq T \bigg\}\\ 
 &\leq 0.
\end{align*}

\end{proof}

\begin{lemma}
     Let $g_1, g_2$ be two Riemannian metrics in $R^-(M)$. The  Peierls barrier zero set

$$
\mathcal{H}_{0}(g_1, g_2):=\left\{v \in S^{g_{1}} M \mid \mathcal{H}_{g_{1}, g_{2}}(v, v)=0\right\}
$$
is a closed subset of $S^{g_{1}} M$.
\end{lemma}
\begin{proof}
    Suppose $v_n$ converges to $v_0$ and $\mathcal{H}_{g_{1}, g_{2}}(v_n, v_n)=0$ for all $n>0$. We want to show $\mathcal{H}_{g_{1}, g_{2}}(v_0, v_0)=0$. It suffices to show for all $\varepsilon_k=\frac{1}{k}$, we have 
    \begin{align*}
    &\mathcal{H}^{\varepsilon_k}_{g_{1}, g_{2}}\left(v_{0}, v_{0}\right) \\
 :=&
 \lim_{T\to \infty} \sup\bigg\{c_u(v,t) \mid  d_{g_{1}^{s}}\left(v, v_{0}\right)<\varepsilon_k, d_{g_{1}^{s}}\left(\phi_{t}^{g_{1}}v, v_{0}\right)<\varepsilon_k, t\geq T \bigg\} \geq 0 .
\end{align*}
Given a fixed $k\in \mathbb{N}$, take $n$ big enough so that  $d_{g_{1}^{s}}\left(v_n, v_0 \right)<\varepsilon'_k\leq \frac{\varepsilon_k}{2}$. Then
 \begin{align*}
0 =\mathcal{H}_{g_{1}, g_{2}}(v_n, v_n)& \le \mathcal{H}^{\varepsilon'_k}_{g_{1}, g_{2}}\left(v_{n}, v_{n}\right) \\
 =&
 \lim_{T\to \infty} \sup\bigg\{c_u(v,t) \mid  d_{g_{1}^{s}}\left(v, v_{n}\right)<\varepsilon'_k, d_{g_{1}^{s}}\left(\phi_{t}^{g_{1}}v, v_{n}\right)<\varepsilon'_k, t\geq T  \bigg\}.
\end{align*}
The choices of $\varepsilon_k$  and $\varepsilon'_k$ together with triangle inequalities for $d_{g_1^s}(\cdot,\cdot)$ imply $$\mathcal{H}^{\varepsilon_k}_{g_{1}, g_{2}}\left(v_{0}, v_{0}\right) \geq \mathcal{H}^{\varepsilon'_k}_{g_{1}, g_{2}}\left(v_{n}, v_{n}\right) \geq 0$$ 
and therefore we obtain $\mathcal{H}_{g_{1}, g_{2}}(v_0, v_0)=\lim\limits_{k\to\infty}\mathcal{H}^{\varepsilon_k}_{g_{1}, g_{2}}\left(v_{0}, v_{0}\right)\geq0$. By Lemma \ref{lemma, ManePotentialLessZero}, we conclude $\mathcal{H}_{g_{1}, g_{2}}(v_0, v_0)=0$.
\end{proof}

\begin{lemma}
     Let $g_1, g_2$ be two Riemannian metrics in $R^-(M)$. The Peierls barrier zero set $\mathcal{H}_{0}(g_1, g_2)$
is a $\phi^{g_{1}}$-invariant subset of $S^{g_{1}} M$.
\end{lemma}
\begin{proof}
     We claim if $v_{0}$ satisfies $\mathcal{H}_{g_{1}, g_{2}}\left(v_{0}, v_{0}\right)=0$, then for any real number $T$, the vector $\phi_{T}v_{0}$ also satisfies $\mathcal{H}_{g_{1}, g_{2}}\left(\phi_{T}v_{0}, \phi_{T}v_{0}\right)=0$. Equivalently, we want to show, for any $\varepsilon>0$, any $\kappa>0$ and any big positive $\tau_0$, we can find a vector $v$ and $t_{0}\geq \tau_0$ so that

\begin{equation} \label{eqtn, EquationVV_0}
d_{g_{1}^{s}}\left(v, \phi_{T}v_{0}\right)<\varepsilon, \; d_{g_{1}^{s}}\left(\phi_{t_{0}}v, \phi_{T}v_{0}\right)<\varepsilon 
\end{equation}

and

\begin{equation}\label{eqtn, EqtnKappa}
\int_{0}^{t_{0}} a_{g_{1}, g_{2}}\left(\phi_{s}^{g_{1}} v\right) d s-t_{0} \cdot S\left(g_{1}, g_{2}\right)>-\kappa. 
\end{equation}

For a fixed given $T \in \R$, define $ C_0 =  C_0(T) := \max_{v \in S^{g_{1}} M} \| D\phi^{g_{1}}_T(v) \|$ where the norm is taken
with respect to the Sasaki metric $g^s_1$. Then for all $v, w \in S^{g_{1}} M $ we have

$$
d_{g_{1}^{s}}\left(\phi_{T}^{g_{1}}v, \phi_{T}^{g_{1}}w \right) \leq C_{0} d_{g_{1}^{s}}(v, w).
$$

Since $\mathcal{H}_{g_{1}, g_{2}}(v_{0}, v_{0}) =0$, for every $\varepsilon >0$ and $\kappa >0$ and $\tau_0>0$, there exist  $t_0 \geq \tau_0$ and $w \in  S^{g_{1}} M$ such that 
$$
d_{g_{1}^{s}}\left(w, v_0\right)<\frac{\varepsilon}{2 C_0} \;\; \text{and} \; \; d_{g_{1}^{s}}\left(\phi_{t_0}^{g_{1}}w, v_0\right)<\frac{\varepsilon}{2 C_0}
$$
as well as

$$
\int_{0}^{t_{0}} a_{g_{1}, g_{2}}\left(\phi_{s}^{g_{1}} w\right) d s-t_{0} \cdot S\left(g_{1}, g_{2}\right)>-\frac{\kappa}{2} .
$$

Hence, for  $v=\phi_{T}^{g_{1}}(w)$ we obtain
$$
d_{g_{1}^{s}}\left(v, \phi_{T}^{g_1}v_{0}\right) = d_{g_{1}^{s}}\left( \phi_{T}^{g_1} w, \phi_{T}^{g_1}v_{0}\right) \le C_0 d_{g_{1}^{s}}\left(w, v_{0}\right) < \varepsilon
$$
and
$$
d_{g_{1}^{s}}\left(\phi_{t_{0}}^{g_1} v, \phi_{T}^{g_1}v_{0}\right) = d_{g_{1}^{s}}\left( \phi_{T}^{g_1} \phi_{t_{0}}^{g_1}w, \phi_{T}^{g_1}v_{0}\right) \le C_0 d_{g_{1}^{s}}\left(\phi_{t_{0}}^{g_1}w, v_{0}\right) < \varepsilon
$$
which yields inequalities \eqref{eqtn, EquationVV_0}. Furthermore
$$
\begin{aligned}
& \int_{0}^{t_{0}} a_{g_{1}, g_{2}}\left(\phi_{s}^{g_{1}} v\right) d s-t_{0} \cdot S\left(g_{1}, g_{2}\right) \\
= & \int_{-T}^{t_{0}-T} a_{g_{1}, g_{2}}\left(\phi_{s}^{g_{1}} v\right) d s-t_{0} \cdot S\left(g_{1}, g_{2}\right)+\left(\int_{t_{0}-T}^{t_{0}} a_{g_{1}, g_{2}}\left(\phi_{s}^{g_{1}} v\right) d s-\int_{-T}^{0} a_{g_{1}, g_{2}}\left(\phi_{s}^{g_{1}} v\right) d s\right) \\
= &\int_{0}^{t_{0}} a_{g_{1}, g_{2}}\left(\phi_{s}^{g_{1}} w\right) d s-t_{0} \cdot S\left(g_{1}, g_{2}\right)+\left(\int_{t_{0}-T}^{t_{0}} a_{g_{1}, g_{2}}\left(\phi_{s}^{g_{1}} v\right) d s-\int_{-T}^{0} a_{g_{1}, g_{2}}\left(\phi_{s}^{g_{1}} v\right) d s\right) \\
> & -\frac{\kappa}{2}+\int_{-T}^{0}\left(a_{g_{1}, g_{2}}\left(\phi_{t_{0}+s}^{g_{1}} v\right)-a_{g_{1}, g_{2}}\left(\phi_{s}^{g_{1}} v\right)\right) d s .
\end{aligned}
$$
Since 
\begin{align*}
d_{g_{1}^{s}}\left( v, \phi_{t_{0}}^{g_{1}} v\right) &= d_{g_{1}^{s}}\left(\phi_T^{g_{1}} w, \phi_T^{g_{1}}\phi_{t_{0}}^{g_{1}} w\right)\leq C_{0} d_{g_{1}^{s}}\left( w, \phi_{t_{0}}^{g_{1}} w\right) \\
&\leq C_{0}\big(d_{g_{1}^{s}}\left(w, v_{0}\right)+d_{g_{1}^{s}}\left(v_{0}, \phi_{t_{0}}^{g_{1}} w\right)\big) < \varepsilon.
\end{align*}
If $T>0$, choosing $\varepsilon$ sufficiently small,
the continuity of $a_{g_{1}, g_{2}}$ and the Lipschitz  continuity of $\phi_{s}^{g_{1}}$ for $s \in[-T, 0]$
$$
\int_{-T}^{0}\left|a_{g_{1}, g_{2}}\left( \phi_{s}^{g_{1}} \phi_{t_{0}}^{g_{1}} v\right)-a_{g_{1}, g_{2}}\left( \phi_{s}^{g_{1}}v\right)\right| d s<\frac{\kappa}{2} .
$$
 If $T<0$, consider $s\in [0,-T]$. The proof is the same. This yields Inequality \eqref{eqtn, EqtnKappa} and therefore $\mathcal{H}_{g_{1}, g_{2}}\left(\phi_{T}^{g_{1}}\left(v_{0}\right), \phi_{T}^{g_{1}}\left(v_{0}\right)\right)=0$.
\end{proof}

We next show that the Mather set is contained in the Peierls barrier zero set $\mathcal{H}_{0}(g_1, g_2)$.

 \begin{proposition} \label{prop,ManeAubrySet}
     
Let $g_{1}, g_{2} \in R^{-}(M)$. The Mather set satisfies

$$
\mathcal{M}(g_1,g_2) \subset \mathcal{H}_{0}(g_1, g_2)
$$
 \end{proposition}
 \begin{proof}
     This proposition can be viewed as a corollary of Proposition \ref{prop, AubryMather}. Take a strong supersolution $u$. In Proposition \ref{prop, AubryMather}, we have shown that if $m \in MS(g_1,g_2)$ and if $f \colon S^{g_1}M\to \mathbb{R}_{\geq 0}$ is the function given by
     $$f(v):=S(g_1,g_2)-a_{g_1,g_2}+X_{g_1}u(v),$$
     then $f$ vanishes on $\supp m$. From Poincare's Recurrence Theorem, we know that $m$-almost every point is recurrent. Suppose $v_0\in \supp m $ is recurrent. Then there exists $t_n\to \infty$ so that $d_{g_{1}^{s}}\left( v_0, \phi_{t_n}^{g_{1}} v_0\right) \to 0$ when $n\to \infty$ and
\begin{align*}
\int_{0}^{t_n} a_{g_{1}, g_{2}}\left(\phi_{s}^{g_{1}} v_0\right) d s-t_n\cdot S\left(g_{1}, g_{2}\right) & =u\left(\phi_{t_n}^{g_{1}} v_0\right)-u(v_0). 
\end{align*}
 This implies $\mathcal{H}_{g_{1}, g_{2}}\left(v_{0}, v_{0}\right)=0$. 

Suppose $v_0\in\supp m$ is not recurrent. Then since non-recurrent points form a null set, every open neighborhood of $v_0$ must contain recurrent points. Therefore  $\mathcal{H}_{g_{1}, g_{2}}\left(v_{0}, v_{0}\right)=0$ by closeness of the subset  $\mathcal{H}_{0}(g_1, g_2)$. This shows that  $\supp  m \subset \mathcal{H}_{0}(g_1, g_2)$ for any $ m \in MS(g_1,g_2)$.  

Finally, since $\mathcal{H}_{0}(g_1, g_2)$ is closed, we conclude
$$
\mathcal M(g_1, g_2) = \overline{\bigcup_{m \in MS(g_1,g_2)}\supp m} \subset \mathcal{H}_{0}(g_1, g_2).
$$
 \end{proof}

In the end, we show that the set $\mathcal{H}_{0}(g_1, g_2)$ is contained in the Aubry set $\mathcal{A}(g_1,g_2)$.
 
 \begin{proposition} \label{prop,ManeAubrySet}
     
Let $g_{1}, g_{2} \in R^{-}(M)$ and let $u$ be any weak supersolution for $a_{g_1,g_2}$. Then the Aubry set $\mathcal{A}_u\left(g_{1}, g_{2}\right)$ of $u$ satisfies

$$
\mathcal{H}_{0}(g_1, g_2) \subset \mathcal{A}_u\left(g_{1}, g_{2}\right).
$$
In particular, we have 
$$
\mathcal{H}_{0}(g_1, g_2) \subset \mathcal{A}\left(g_{1}, g_{2}\right).
$$
 \end{proposition}

\begin{proof} 
    We first show that if $v_{0}$ satisfies $\mathcal{H}_{g_{1}, g_{2}}\left(v_{0}, v_{0}\right)=0$, then it is in $\mathcal{A}_u\left(g_{1}, g_{2}\right)$. Given a weak supersolution $u$, we find a nonpositive cocycle $c_u$. Let us fix a finite positive constant $\tau_0$ and consider $t\geq \tau_0$. We have
$$
\begin{aligned}
\int_{0}^{t} a_{g_{1}, g_{2}}\left(\phi_{s}^{g_{1}} v\right) d s-t\cdot S\left(g_{1}, g_{2}\right) & =c_u(v,t) +u\left(\phi_{t}^{g_{1}} v\right)-u(v) \\
& \leq c_u(v,\tau_0)+u\left(\phi_{t}^{g_{1}} v\right)-u(v).
\end{aligned}
$$

This implies, by finding positive $t$ and vectors $v$ so that $v \rightarrow v_{1}$ and $\phi_{t}^{g_{1}}(v) \rightarrow$ $v_{2}$, together with the continuity of $u$ and $c_u$,

$$
\mathcal{H}_{g_{1}, g_{2}}\left(v_{1}, v_{2}\right) \leq c_u(v_1,\tau_0)+u\left(v_{2}\right)-u\left(v_{1}\right) .
$$

When $v_{0}=v_{1}=v_{2}$ and $\mathcal{H}_{g_{1}, g_{2}}\left(v_{0}, v_{0}\right)=0$, we obtain $c_u(v_0,\tau_0)=0$. Since $c_u$ in monotonically nonincreasing in $t$, we obtain $c_u(v_0,t)=0$ for $t\leq \tau_0$. Because $\tau_0$ is arbitrarily chosen, we conclude for any $t\geq 0$,
$$c_u(v_0,t)=0,$$ 
and hence
$v_{0} \in \mathcal{A}_{u}\left(g_{1}, g_{2}\right)$.

\end{proof}

 \begin{remark}
     We do not know if $\mathcal{H}_{0}(g_1, g_2)$ coincides with our definition of the Aubry set (see Question \ref{question,AubryPeierl} in Appendix \ref{question: open question}). We note that, in an analogous setting in \cite{FS04} for Lagrangian dynamics, the Aubry set was originally defined using Peierls barrier.
 \end{remark}

\subsubsection{Reverse triangle inequality of the Peierls barrier}
We show that the Peierls barrier satisfies the reverse triangle inequality in this subsection. We denote for a fixed $\varepsilon >0$,

\begin{align*}
\mathcal{H}^{\varepsilon}_{g_{1}, g_{2}}(v_{1}, v_{2}):= 
\lim_{T\to \infty} \sup\bigg\{&\int_{0}^{t} a_{g_{1}, g_{2}}(\phi_{s}^{g_{1}} v) d s-t \cdot S(g_{1}, g_{2}) \mid  \\ 
 & d_{g_{1}^{s}}\left(v, v_{1}\right)<\varepsilon, d_{g_{1}^{s}}\left(\phi_{t}^{g_{1}}v, v_{2}\right)<\varepsilon, t\geq T \bigg\},
\end{align*}
and further for a fixed $T>0$, denote
\begin{align*}
\mathcal{H}^{\varepsilon,T}_{g_{1}, g_{2}}(v_{1}, v_{2}):= 
\ \sup\ \bigg\{ &\int_{0}^{t} a_{g_{1}, g_{2}}(\phi_{s}^{g_{1}} v) d s-t \cdot S(g_{1}, g_{2}) \mid t\geq T \\ 
  & \quad d_{g_{1}^{s}}\left(v, v_{1}\right)<\varepsilon, d_{g_{1}^{s}}\left(\phi_{t}^{g_{1}}v, v_{2}\right)<\varepsilon \bigg\}.
\end{align*} Note that $\mathcal{H}^{\varepsilon,T}_{g_{1}, g_{2}}(v_{1}, v_{2})$ monotonically decreases to $\mathcal{H}^{\varepsilon}_{g_{1}, g_{2}}(v_{1}, v_{2})$ when  $T\to \infty$ and $\mathcal{H}^{\varepsilon}_{g_{1}, g_{2}}(v_{1}, v_{2})$ monotonically decreases in $ \varepsilon$  as $ \varepsilon \to 0$ and converges to the Peierls barrier $\mathcal{H}_{g_{1}, g_{2}}\left(\cdot, \cdot\right) $.
Moreover, $\mathcal{H}^{\varepsilon,T}_{g_{1}, g_{2}}(v_{1}, v_{2})$ is monontone decreasing in $\varepsilon$ with limit $ \lim\limits_{\varepsilon \to 0}\mathcal{H}^{\varepsilon,T}_{g_{1}, g_{2}}(v_{1}, v_{2}) $. 

In the sequel, we need the following simple version of the shadowing Lemma. See \cite[Theorem 18.3.14]{KH95} for a general version.

\begin{lemma}[Shadowing Lemma]\label{lem, shadowing}
Let $(M, g)$ be a complete manifold of strictly negative curvature. Then for every 
$\varepsilon>0$, there exists $\delta>0$ such that for all $w_1, w_2\in S^{g}M$ and $t_1,t_2 >0$ with
        $d_{g^s}(\phi^g_{t_1}w_1,w_2) < \delta$,
    there exists some $w\in S^{g}M$ with
    $$
   d_{g^s} (\phi_t^{g}w, \phi_t^{g}w_1)< \varepsilon  \; \text{if} \; \; 0\leq t\leq t_1,
   $$
 and    
    $$
   d_{g^s} (\phi_{t}^{g}w, \phi_{t-t_1}^{g}w_2)< \varepsilon \; \text{if} \; \; t_1\leq t\leq t_1+t_2. 
   $$
\end{lemma}     
\begin{proposition}[Reverse triangle inequality] \label{prop,reverseTriangle}
The Peierls barrier satisfies for all $T>0$ and all
 $v_1,v_2,v_3\in S^{g_1}M$,

$$\mathcal{H}_{g_{1}, g_{2}}\left(v_1, v_3\right)\geq \mathcal{H}_{g_{1}, g_{2}}\left(v_1, v_2\right)+ 
\lim\limits_{\varepsilon \to 0}\mathcal{H}^{\varepsilon,T}_{g_{1}, g_{2}}(v_{2}, v_{3}).$$
In particular, it satisfies the reverse triangle inequality, that is,
$$\mathcal{H}_{g_{1}, g_{2}}\left(v_1, v_3\right)\geq \mathcal{H}_{g_{1}, g_{2}}\left(v_1, v_2\right)+ 
 \mathcal{H}_{g_{1}, g_{2}}\left(v_2, v_3\right).$$

\end{proposition} 
\begin{proof}

  For $(M, g_1)$ and  $\varepsilon>0$, choose $\delta >0$ as in the shadowing Lemma \ref{lem, shadowing}.
  Consider the (not necessarily distinct) vectors $v_1,v_2,v_3 \in S^{g_1}M$ and let $T$ be any positive number. By the definition of $\mathcal{H}^{\frac{\delta}{2}}_{g_{1}, g_{2}}$, for any $\kappa>0$ and any $T_1>0$, there exist $w_1$ and $t_1>T_1$ such that 
    \begin{equation} \label{eqtn, v1v2}
        d_{g_1^s}(v_1,w_1) < \frac{\delta}{2}, \quad d_{g_1^s}(v_2,\phi^{g_1}_{t_1}w_1) < \frac{\delta}{2},
     \end{equation}
and recall from Definition \ref{def:timechange}  
    $$\mathcal{H}^{\frac{\delta}{2}}_{g_{1}, g_{2}}(v_{1}, v_{2})-\kappa \leq b^{g_2}_{(w_1)^{g_1}_{+}} (\pi w_1, \pi \phi^{g_1}_{t_1}w_1)-t_1 S(g_1,g_2).$$
By the definition of  $\mathcal{H}^{\frac{\delta}{2},T}_{g_{1}, g_{2}}$, for any $\kappa >0$,  
there exist $w_2$ and $t_2\geq T$ such that
  \begin{equation}\label{eqtn, v2v3}
  d_{g_1^s}(v_2,w_2) < \frac{\delta}{2}, \quad d_{g_1^s}(v_3,\phi^{g_1}_{t_2}w_2) < \frac{\delta}{2},
  \end{equation}
and 
  $$\mathcal{H}^{\frac{\delta}{2},T}_{g_{1}, g_{2}}(v_{2}, v_{3})-\kappa \leq b^{g_2}_{(w_2)^{g_1}_{+}} (\pi w_2, \pi \phi^{g_1}_{t_2}w_2)-t_2 S(g_1,g_2)$$
  holds.
This in particular implies
$$
d_{g_1^s}(\phi^{g_1}_{t_1}w_1,w_2) \le d_{g_1^s}(\phi^{g_1}_{t_1}w_1,v_2) + d_{g_1^s}(v_2, w_2) < \delta.
$$
The shadowing Lemma yields the existence of
 $w\in S^{g_1}M$ with
   \begin{equation}\label{eqtn, ww_1}
  d_{g_1^s} (\phi_t^{g_1}w, \phi_t^{g_1}w_1)< \varepsilon, \qquad  0\leq t\leq t_1,
     \end{equation}
and    
    \begin{equation}\label{eqtn, ww_2}
  d_{g_1^s} (\phi_{t}^{g_1}w, \phi_{t-t_1}^{g_1}w_2)< \varepsilon, \qquad  t_1\leq t\leq t_1+t_2. 
  \end{equation}
 
This  together with the inequalities \eqref{eqtn, v1v2} and \eqref{eqtn, ww_1}  implies
     $$ d_{g_1^s}(v_1,w) \le d_{g_1^s}(v_1,w_1 ) + d_{g_1^s}(w_1,w)< \varepsilon+ \frac{\delta}{2} $$
and  
$$
d_{g_1^s} (v_3,\phi_{t_1+t_2}^{g_1}w)< d_{g_1^s} (v_3,\phi^{g_1}_{t_2}w_2)+  d_{g_1^s} (\phi^{g_1}_{t_2}w_2, \phi
_{t_1+t_2}^{g_1}w)< \varepsilon+\frac{\delta}{2}.
$$
Furthermore, we can choose $T_1 > 0$ such that
      \begin{align*}
         \mathcal{H}^{\varepsilon+\frac{\delta}{2}}_{g_{1}, g_{2}}(v_{1}, v_{3}) 
         \geq \mathcal{H}^{\varepsilon+\frac{\delta}{2},T_1}_{g_{1}, g_{2}}(v_{1}, v_{3}) -\kappa.
     \end{align*}
     Since  $t_1+t_2\geq T_1+T> T_1$, it follows
     \begin{align*}
\mathcal{H}^{\varepsilon+\frac{\delta}{2}}_{g_{1}, g_{2}}(v_{1}, v_{3})&  
         \geq b^{g_2}_{w^{g_1}_+}(\pi w, \pi \phi^{g_1}_{t_1+t_2}w) -(t_1+t_2) S(g_1,g_2)-\kappa\\
         &\geq  b^{g_2}_{w^{g_1}_+}(\pi w, \pi \phi^{g_1}_{t_1}w) -t_1 S(g_1,g_2) \\
        &  \quad + b^{g_2}_{w^{g_1}_+}(\pi \phi^{g_1}_{t_1}w, \pi \phi^{g_1}_{t_1+t_2}w) -t_2 S(g_1,g_2)-\kappa.
\end{align*}
         By the continuity of the Busemann cocycle (see Subsection \ref{busemann cocycle}), for any $\eta>0$, their exists $\varepsilon_\eta >0$ such that 
         $$b^{g_2}_{{w}^{g_1}_+}(\pi w, \pi \phi^{g_1}_{t_1}w) \geq b^{g_2}_{({w_1})^{g_1}_+}(\pi w_1, \pi \phi^{g_1}_{t_1}w_1)-\eta,$$
         and 
          $$b^{g_2}_{w^{g_1}_+}(\pi \phi^{g_1}_{t_1}w, \pi \phi^{g_1}_{t_1+t_2}w) \geq b^{g_2}_{(w_2)^{g_1}_+}(\pi w_2, \pi \phi^{g_1}_{t_2}w_2)-\eta,$$
          provided $ \varepsilon \le \varepsilon_\eta $. 
          Hence for all  $ \varepsilon \le \varepsilon_\eta $ we have
         \begin{align*}
\mathcal{H}^{\varepsilon+\delta}_{g_{1}, g_{2}}(v_{1}, v_{3})
         &\geq b^{g_2}_{{w_1}^{g_1}_+}(\pi w_1, \pi \phi^{g_1}_{t_1}w_1) -t_1 S(g_1,g_2) \\
         & \quad +b^{g_2}_{({w_2})^{g_1}_+}(\pi w_2, \pi \phi^{g_1}_{t_2}w_2) -t_2 S(g_1,g_2)-\kappa-2\eta \\
         &\geq \mathcal{H}^{\frac{\delta}{2}}_{g_{1}, g_{2}}(v_{1}, v_{2})+\mathcal{H}^{\frac{\delta}{2},T}_{g_{1}, g_{2}}(v_{2}, v_{3}) -3\kappa-2\eta
     \end{align*}
     
     Taking the limits $\varepsilon \to 0$, $\delta=\delta(\varepsilon)\to 0$, we obtain
$$
\mathcal{H}_{g_{1}, g_{2}}(v_{1}, v_{3}) \ge \mathcal{H}_{g_{1}, g_{2}}(v_{1}, v_{2})+\lim\limits_{\varepsilon \to 0}\mathcal{H}^{\varepsilon,T}_{g_{1}, g_{2}}(v_{2}, v_{3}) -3\kappa-2\eta
$$
for all $\kappa, \eta >0$.  This yields the first inequality. The second inequality is a consequence of
$   \mathcal{H}_{g_{1}, g_{2}}(v_{2}, v_{3})  \le \lim\limits_{\varepsilon \to 0}\mathcal{H}^{\varepsilon,T}_{g_{1}, g_{2}}(v_{2}, v_{3}).$
     \end{proof}

\section{Maximal stretch and thermodynamic formalism} \label{section,thermodynamics}
\subsection{The Mather set and measures of maximal entropy } \label{subsection,MMEMatherSet}
\hfill\\
In this section, we use tools from the thermodynamic formalism to study the maximal stretch $S(g_1,g_2)$ and maximally stretched measures. We refer the reader to \cite{Bo08} for a good survey on the subject of thermodynamic formalism which has a close relation with statistic physics.

We will establish a connection between the theory of equilibrium states, measures of maximal entropy, and the study of the Mather set. A starting observation is that the maximal stretch can be expressed via the \emph{pressure function} in the thermodynamic formalism. We recall that the \emph{pressure (function)} of the potential $r a_{g_1,g_2}$  is given by the following \emph{variational principle},
 $$\mathbf{P}(r a_{g_1, g_2})=\sup\limits_{m\in \mathcal{M}^1(\phi^{g_1})} \big(h_{m}(\phi^{g_1})  +\int r a_{g_1,g_2} dm ).$$

 The above supremum is realized by a unique $\phi^{g_1}$-invariant probability measure $m_{ra_{g_1,g_2}}$, called the \emph{equilibrium state} of the potential $r a_{g_1,g_2}$.
In statistical mechanics, the scaling constant $r$ is interpreted as the inverse of the temperature $r=\frac{1}{T}$.  When the temperature $T=\frac{1}{r} \to 0$, the potential energy $E(m)=\int a_{g_1,g_2} dm =I_m(g_1,g_2)$
becomes the dominating term in the pressure function. This motivates the following.

\begin{lemma}
    Given $g_1,g_2\in R^-(M)$, the maximal stretch satisfies
    $$S(g_1,g_2)= \lim\limits_{r\to \infty} \frac{\mathbf{P}(r a_{g_1, g_2})}{r}.$$
\end{lemma}
\begin{proof}
    By the variational principle, 
    \begin{align*}
        \mathbf{P}(r a_{g_1, g_2})&=\sup\limits_{m\in \mathcal{M}^1(\phi^{g_1})} \big(h_{m}(\phi^{g_1})  +\int r a_{g_1,g_2} dm \big)\\
        &\leq \sup\limits_{m\in \mathcal{M}^1(\phi^{g_1})} h_{m}(\phi^{g_1})  + \sup\limits_{m\in \mathcal{M}^1(\phi^{g_1})} \int r a_{g_1,g_2} dm   \\
         &= h_{top}(\phi^{g_1})+ r S(g_1,g_2).
          \end{align*}
    On the other hand, since $h_m(\phi^{g_1})\geq 0$, it also holds  
    $$\mathbf{P}(r a_{g_1, g_2}) \geq \sup\limits_{m\in \mathcal{M}^1(\phi^{g_1})} \int r a_{g_1,g_2} dm = rS(g_1,g_2).$$
    The result follows from a combination of the two inequalities.     
\end{proof}

To simplify our notation, we denote the equilibrium states $m_{ra_{g_1,g_2}}$ as $m_r$. Up to subsequences, the measures $m_{ra_{g_1,g_2}}$ converge in weak-* topology to some weak limits when $r\to \infty$. These weak limits are called the \emph{zero-temperature limits} of equilibrium states. 

In the following, we show that \emph{zero-temperature limits} are maximally stretched measures.

\begin{lemma}\label{lemma,weaklimitEquilibriumStates}
    Each weak limit $m_+=\lim\limits_{n\to\infty} m_{r_n}$ of the equilibrium states for some subsequence $r_n\to \infty$ is a maximally stretched measure.
\end{lemma}

\begin{proof}
   For all $m \in \mathcal{M}^1(\phi^{g_1})$, we have
\begin{eqnarray*}
h_m(\phi^{g_1}) + r \int a_{g_1, g_2} dm &\le & \mathbf{P}(r a_{g_1, g_2})  \\
&=&  h_{m_{r}} (\phi^{g_1})+r \int a_{g_1, g_2}dm_{r}.
\end{eqnarray*}
This implies
$$
\int a_{g_1, g_2}dm_{r} - \int a_{g_1, g_2} dm  \ge \frac{1}{r}\big(  h_{m} (\phi^{g_1})- h_{m_r}(\phi^{g_1}) \big).
$$
Let $r \to  \infty$. The right hand side goes to zero. We obtain, for each weak limit $m_+ := \lim\limits_{n \to \infty} m_{r_n}$ of some subsequence $\{r_n\}$,
 $$ \int a_{g_1, g_2}dm_+ \ge  \int a_{g_1, g_2} dm.$$
 Therefore
$$ \int a_{g_1, g_2}dm_+ =\sup_{m\in\mathcal{M}^{1}(\phi^{g_1})} \int a_{g_1, g_2}dm=S(g_1,g_2),$$
and $m_+$ is a maximally stretched measure.
\end{proof}

We next show that a zero-temperature limit $m_+$ is a measure of maximal entropy of the geodesic flow $\phi^{g_1}$ when restricted to the Mather set $\M(g_1,g_2)$.

\begin{theorem}[Measures of maximal entropy for $\mathcal{M}(g_1,g_2)$]
\label{prop,limitIsTopEntropy}
   Any zero-temperature limit $m_+$ is a measure of maximal entropy of the geodesic flow $\phi^{g_1}$ restricted to the Mather set $\M(g_1,g_2)$, i.e.
$$h_{m_{+}} (\phi^{g_1}) = h_{top}(\phi^{g_1}, \mathcal{M}(g_1,g_2)),$$
    where $h_{top}(\phi^{g_1}, \mathcal{M}(g_1,g_2))$ is the topological entropy of $\phi^{g_1}$ on the closed invariant subset $\mathcal{M}(g_1,g_2)$ and $h_{m_{+}} (\phi^{g_1})$ denotes the metric entropy of $m_+$. 
\end{theorem}
\begin{proof}
    We first give an argument that any $\phi^{g_1}$-invariant probability measure $m$ with support in $\mathcal{M}(g_1,g_2)$ is maximally stretched. By Birkhoff's ergodic theorem for invariant measures, we have
   $$\int_{\mathcal{M}(g_1,g_2)} a_{g_1,g_2}dm =\int_{\mathcal{M}(g_1,g_2)} \lim_{T\to\infty} \frac{1}{T}\int_0^T a_{g_1,g_2}(\phi^{g_1}_tx)dtdm(x).$$
   
   If $x\in \supp(m)\subset \mathcal{M}(g_1,g_2)$ and the above time average exists, then since $\mathcal{M}(g_1,g_2)$ is contained in the Aubry set $ \mathcal{A}(g_1,g_2)$, it follows from the definition of $ \mathcal{A}(g_1,g_2)$ that 
   $$\lim_{T\to\infty} \frac{1}{T}\int_0^T a_{g_1,g_2}(\phi^{g_1}_tx)dt = S(g_1,g_2).$$ 
    
    Combining the above estimates together yields that any $\phi^{g_1}$-invariant probability measure $m$ with support in $\mathcal{M}(g_1,g_2)$ is maximally stretched. Hence for any such measure and any  $r>0$, by the variational principle,
\begin{align*}
    h_{m}(\phi^{g_1}) +r\int a_{g_1,g_2} dm_r &\leq h_{m}(\phi^{g_1})+r\int a_{g_1,g_2} dm\\
    &\leq h_{m_r}(\phi^{g_1}) +r\int a_{g_1,g_2} dm_r.
    \end{align*} 
    Therefore, $ h_{m}(\phi^{g_1}) \leq \inf\limits_{r> 0} h_{m_r}(\phi^{g_1}) $. 
    
    Again, applying the variational principle to the geodesic flow $\phi^{g_1}$ restricted to $\mathcal{M}(g_1,g_2)$ and the potential function $0$ yields
    $$ \max_{m \in MS(g_1,g_2)} h_{m}(\phi^{g_1})
   = h_{top}(\phi^{g_1}, \mathcal{M}(g_1,g_2)) \leq \inf\limits_{r> 0}h_{m_r}(\phi^{g_1}). $$
    
    Since the metric entropy $m \to h_m (\phi^{g_1})$ is upper semicontinuous (see, for example, \cite[Theorem 8.2]{Wa82}), we obtain the following.

    $$\limsup_{r \to \infty} h_{m_r}(\phi^{g_1}) \leq h_{top}(\phi^{g_1}, \mathcal{M}(g_1,g_2)).$$
    This implies for any maximally stretched measure $m_+$ arising as a limit of equilibrium states, we have
    $$h_{m_{+}} ( \phi^{g_1})= h_{top}(\phi^{g_1}, \mathcal{M}(g_1,g_2))=\inf\limits_{r>0} h_{m_r}(\phi^{g_1}),$$
    and therefore, each such measure has maximal
 entropy on the Mather set.

\end{proof}

 Lemma \ref{lemma,weaklimitEquilibriumStates} guarantees the existence of maximally stretched measures arising as limits of equilibrium states. Nevertheless, these  maximally stretched measures are in general not equilibrium states because of the following.
\begin{remark}
We make the following remarks regarding whether maximally stretched measures can be equilibrium states of some potential functions.
\begin{itemize}
    \item If  $\frac{h_{top} (\phi^{g_2})}{h_{top} (\phi^{g_1}) }S(g_1,g_2) > 1$ for $g_1, g_2 \in R^-(M)$,  maximally stretched measures for $g_1,g_2$ cannot be equilibrium states. The reason is that according to Theorem \ref{thm, EmtpyInterior}, the support of a maximally stretched measure does not contain any open set. On the other hand, by Theorem 3.3 of \cite{BR75}, the equilibrium states of the Anosov flows have full support. 
    \item If $\frac{h_{top} (\phi^{g_2})}{h_{top} (\phi^{g_1}) }S(g_1,g_2)=1$, Theorem \ref{thm, EmtpyInterior} implies
$\mathcal{M}(g_1,g_2)=S^{g_1}M$. Since
by Theorem \ref{prop,limitIsTopEntropy}, a maximally stretched measure arising as a limit of equilibrium states is a measure of maximal entropy (i.e. the Bowen-Margulis measure for $\phi^{g_1}$) and is the equilibrium state for the potential function $0$.
\end{itemize}
\end{remark}

\subsection{Further discussion of metric entropy}
We continue to characterize the behavior of metric entropy of equilibrium states $m_r$ of potentials $r a_{g_1,g_2}$ and their zero-temperature limits in this subsection.
We will see in Proposition \ref{prop,EquilibriumStateLimit} that if a maximally stretched measure $m_+$ is a zero-temperature limit, i.e. it arises as a weak limit of equilibrium states $m_{r_n}$, then the metric entropy of $m_{r_n}$ converges to the metric entropy of $m_{+}$ in a monotone manner.

We start from the following proposition.

\begin{proposition} \label{prop, MonotoneMetricEntropy}
    The function $r\mapsto h_{m_r}(\phi^{g_1})$ is analytic. Moreover, 
    $$ \frac{ \partial h_{m_r}(\phi^{g_1})} {\partial r} =-r  \cdot \textnormal{Var} (P_{m_r}(a_{g_1,g_2}), m_r),$$
    where the \emph{variance} of a mean zero Hölder function $f \colon  S^{g_1}M \to \R$ with respect to $m_r$ is given by
    \begin{equation}
    \mathrm{Var}(f,m_r)= \lim_{T\to\infty} \frac{1}{T} \int_{ S^{g_1}M}  \left( \int_{0}^{T} f(\phi_{s}(v))\mathrm{d}s \right) ^2 \mathrm{d}m_r(v),
    \label{eq:Var}
\end{equation}
and the projection operator $P_{m_r}$ is defined as $P_{m_r}(f)(v):= f(v)- \int f \mathrm{d}m_r$. 
\end{proposition}
\begin{proof}
We recall from Remark \ref{remark,regularityTimeChange} that the infinitesimal time change $a_{g_1,g_2}$ is Hölder. By the variational principle,
\begin{equation} \label{eqtn,PressureInVariableT}
 \mathbf{P}(r a_{g_1, g_2}) = h_{m_r}(\phi^{g_1}) + r\int a_{g_1,g_2}\mathrm{d}m_r. 
\end{equation}
We also know from thermodynamic formalism (see \cite[Proposition 4.10]{PP90}),
$$\int a_{g_1,g_2}\mathrm{d}m_r = \frac{d}{dr}  \mathbf{P}(r a_{g_1, g_2}).$$
This yields
\begin{equation} \label{eqtn,entropy}
h_{m_r}(\phi^{g_1}) = \mathbf{P}(r a_{g_1, g_2}) -r \frac{d}{dr}  \mathbf{P}(r a_{g_1, g_2}).
\end{equation}
Since the pressure $\mathbf{P}(ra_{g_1,g_2})$ is analytic in $r$, it follows that  $r\mapsto h_{m_r}(\phi^{g_1})$ is analytic as well. Furthermore, $h_{m_0}(\phi^{g_1}) = h_{top} (\phi^{g_1})$ and therefore $m_0 =  m^{g_1}_{BM}$.

Now taking a derivative on both sides of the Equation  \eqref{eqtn,entropy} yields
 \begin{equation}  \label{eqtn,derivativeOfT}  
\frac{d}{dr} h_{m_r}(\phi^{g_1})=-r  \frac{d^2}{dr^2} \mathbf{P}(r a_{g_1, g_2}) 
 =  -r\cdot \mathrm{Var}(P_{m_r}(a_{g_1,g_2}),m_r),
\end{equation}
where the deduction of second derivative formulas of pressure functions can be found in \cite[Proposition 4.11]{PP90} and \cite[Corollary 2.2, Remark 2.25]{Da23}.

\end{proof}

Let us denote by $m^{g_1}_{BM}$ the Bowen--Magulis measure of the geodesic flow $\phi^{g_1}$ on $S^{g_1}M$.
\begin{corollary}\label{cor,dichotomyEntropy}
The following dichotomy holds.
\begin{itemize}
    \item  If  $\frac{h_{top} (\phi^{g_2})}{h_{top} (\phi^{g_1}) }S(g_1,g_2) =1$, then the metric entropy satisifies, for $r >0$,
    $$  h_{m_r}(\phi^{g_1}) \equiv h_{top}(\phi^{g_1}),$$ 
    and so $$m_r\equiv m^{g_1}_{BM}.$$
   \item 
   If  $\frac{h_{top} (\phi^{g_2})}{h_{top} (\phi^{g_1}) }S(g_1,g_2) >1,$ then $h_{m_r}(\phi^{g_1})$ strictly decreases when $r$ increases. 
   
\end{itemize}
   \begin{proof}
       From the equation,
       $$ \frac{d}{dr} h_{m_r}(\phi^{g_1}) 
=-r  \cdot \textnormal{Var} (P_{m_r}(a_{g_1,g_2}), m_r),$$ 
 it follows that  $h_{m_r}(\phi^{g_1}) $ is a constant independent of $r$ if and only if the variance is zero. From the properties of the variance (\cite[Proposition 4.12]{PP90}), $\textnormal{Var} (P_{m_r}(a_{g_1,g_2}), m_r)=0$ if and only if $a_{g_1,g_2}$ is cohomologous to the constant $\int a_{g_1,g_2} \mathrm{d}m_r$, or equivalently, the marked length spectra of $g_1$ and $g_2$ are proportional (Corollary \ref{cor:equiv}). This yields the first item by Theorem \ref{thm, EmtpyInterior}. If  $a_{g_1,g_2}$ is not cohomologous to a constant, then $\textnormal{Var} (P_{m_r}(a_{g_1,g_2}), m_r)>0$ and therefore the second item follows.
   \end{proof} 
    
\end{corollary}

\begin{proposition} \label{prop,EquilibriumStateLimit}
    Suppose a maximally stretched measure $m_+$ is a zero-temperature limit, that is, it arises as a weak limit of equilibrium states $m_{r_n}$ for some sequence $r_n \to \infty$. Then 
    \begin{equation} \label{eqtn,MetricEntropyLimit}
        h_{m_+}(\phi^{g_1})=\lim\limits_{n\to \infty} h_{m_{r_n}}(\phi^{g_1}).
    \end{equation}
    Moreover, $h_{m_{r_n}}(\phi^{g_1})$ converges monotonically to $h_{m_+}(\phi^{g_1})$ as $n\to \infty$.
\end{proposition}
\begin{proof}
    If   $\frac{h_{top} (\phi^{g_2})}{h_{top} (\phi^{g_1}) }S(g_1,g_2) =1$,  then by Theorem \ref{thm, EmtpyInterior}, the Mather set $\mathcal{M}(g_1,g_2)=S^{g_1}M$ and by Theorem \ref{prop,limitIsTopEntropy},
   $$h_{m_{+}}(\phi^{g_1})= h_{top}(\phi^{g_1}, \mathcal{M}(g_1,g_2))
   = h_{top}(\phi^{g_1}).$$
  On the other hand, by Corollary \ref{cor,dichotomyEntropy}, we have $ h_{m_r}(\phi^{g_1}) \equiv h_{top}(\phi^{g_1})$. Therefore we obtain Equation \eqref{eqtn,MetricEntropyLimit} for this case.
  
   If $\frac{h_{top} (\phi^{g_2})}{h_{top} (\phi^{g_1}) }S(g_1,g_2) >1$, by Corollary \ref{cor,dichotomyEntropy},  $r \mapsto h_{m_r}(\phi^{g_1}) $ strictly decreases. So by the proof of Theorem \ref{prop,limitIsTopEntropy},
$$h_{m_+}(\phi^{g_1})= \inf\limits_{r>0}h_{m_r}(\phi^{g_1})\leq \lim\limits_{n\to \infty} h_{m_{r_n}}(\phi^{g_1}) \leq h_{top}(\phi^{g_1}, \mathcal{M}(g_1,g_2))),$$
where the last inequality follows from the upper semicontinuity of $m\to h_m (\phi^{g_1})$. Since $h_{m_{+}} (\phi^{g_1})= h_{top}(\phi^{g_1}, \mathcal{M}(g_1,g_2))$, Equality \eqref{eqtn,MetricEntropyLimit} now follows.
\end{proof}

Now we can relate the topological entropy of $\phi^{g_1}$ on $S^{g_1}M$ and the topological entropy of $\phi^{g_1}$ on $\mathcal{M}(g_1,g_2)$ by a precise formula.

\begin{corollary} \label{cor formulaTwoEntropy}
   We have the following formula relating $ h_{top}(\phi^{g_1}, \mathcal{M}(g_1,g_2))$ and $h_{top}(\phi^{g_1})$.
      \begin{equation*}
        h_{top}(\phi^{g_1}, \mathcal{M}(g_1,g_2))=h_{top}(\phi^{g_1}) - \int_0^\infty r  \textnormal{Var} (P_{m_r}(a_{g_1,g_2}), m_r) dr.
    \end{equation*}
\end{corollary}
\begin{proof}
    This is a combination of Proposition \ref{prop,EquilibriumStateLimit} and  Proposition \ref{prop, MonotoneMetricEntropy} together with Theorem \ref{prop,limitIsTopEntropy}.
\end{proof}

Combining Corollary \ref{cor formulaTwoEntropy} with the proof of Corollary \ref{cor,dichotomyEntropy}, we obtain the following main result of this subsection.

\begin{corollary}\label{cor,RigidityTopEntropy}
 Given $g_1,g_2\in R^-(M)$, the topological entropy on the Mather set $\M(g_1,g_2)$ satisfies
 $$h_{top}(\phi^{g_1}, \mathcal{M}(g_1,g_2))=h_{top}(\phi^{g_1})$$
if and only if the marked length spectra of $g_1$ and $g_2$ proportional, i.e. there exists a constant $C>0$ such that $\ell_{g_2}([\gamma])= C \ell_{g_1}([\gamma])$ for all $[\gamma]\in[\Gamma]$. 
\end{corollary}

    At the end of this subsection, we come to a question that originally motivated the study in this section. When $\frac{h_{top} (\phi^{g_2})}{h_{top} (\phi^{g_1}) }S(g_1,g_2) >1$, we have that
    $$ h_{top}(\phi^{g_1}, \mathcal{M}(g_1,g_2))<  h_{top}(\phi^{g_1}).$$  A natural question is:  does the topological entropy $h_{top}(\phi^{g_1}, \mathcal{M}(g_1,g_2))$ vanish in this case? This is indeed true for Teichmüller spaces: when $g_1,g_2$ are hyperbolic metrics on a closed surface representing distinct points in the Teichmüller space, a maximally stretched measure after symmetrization always gives a measured geodesic lamination and measured geodesic laminations have zero metric entropy (see Lemma 2.3.5 of \cite{Na03}). Therefore, by Theorem \ref{prop,limitIsTopEntropy}, when $g_1,g_2$ are hyperbolic metrics, a zero-temperature limit satisifies
    $$h_{m_+}(\phi^{g_1})=h_{top}(\phi^{g_1}, \mathcal{M}(g_1,g_2))=0.$$
    However, for general negatively curved metrics $g_1,g_2$, we will see in Example \ref{MatherNotLamination} that the topological entropy $h_{top}(\phi^{g_1}, \mathcal{M}(g_1,g_2))$ can be positive.

\section{Geodesic stretch and Lipschitz maps} \label{section, StretchLip}

Starting from this section, we turn our attention to Lipschitz maps.
Our final goal of this section is to relate the geodesic stretch of an invariant measure to the average ``stretch" of a Lipschitz map with respect to that measure  (see Section \ref{section, FunctionalGfirst}, Lemma \ref{lem, GeodesicStretchLipStretch} and Theorem \ref{prop,GeodesicToGeodesic}). This connection offers a new perspective of understanding geodesic stretches and maximal stretches. We begin by reviewing some basics of Lipschitz maps in Section \ref{section, PreliminaryLipMaps}.

\subsection{Some preliminaries on Lipschitz maps} \label{section, PreliminaryLipMaps}
\hfill\\
The material in this section is classical. We discuss some properties of Lipschitz maps between Riemannian manifolds.
We denote by $C^{0}(M,M)$ the space of all
continuous maps from $M$  to $M$. The space $C^{0}(M,M)$ is a connected infinite-dimensional Banach manifold (\cite[Chapter II. 11]{Ab63}).  
 We are interested in the subspace $C_{\id}^{0}(M,M)$ of $C^{0}(M,M)$ consisting of maps homotopic to the identity.

 \begin{lemma}\label{lem, uniquelift}
     $f \in  C^{0}(M,M)$ is homotopic to the identity map if and only if there exists a unique lift $ \widetilde{f} : \widetilde{M} \to \widetilde{M}$ of $f$ such that $\widetilde{f}  \gamma = \gamma   \widetilde{f}$
for all $\gamma \in \Gamma$.
 \end{lemma}  
 \begin{proof}
     Assume that $f \in  C^{0}(M,M)$ is homotopic to the identity map, i.e. there exists a homotopy  $H(t,x)=H_t(x):[0,1] \times M \to M$ with $H_0 = \id$ and  $H_1= f$.
Let  $\widetilde H:[0,1] \times \widetilde{M} \to \widetilde{M}$ be the unique lift of $H$ such that $\widetilde H_0 = \id$.
This implies  $\widetilde H_t \gamma = \gamma \widetilde H_t $ for all $\gamma \in \Gamma$. Hence, $\widetilde {H_1} = \widetilde{f}$ is a lift of $f$ such that $\widetilde{f} \gamma = \gamma   \widetilde{f}$
for all $\gamma \in \Gamma$. \\
If $\widetilde{f'}$ is another lift of $f$ with $\widetilde{f'} \gamma = \gamma \widetilde{f'}$ for all $\gamma \in \Gamma
$. Then there exists $\gamma_0, \eta_0 \in \Gamma$ such that
 $\widetilde{f'} = \gamma_0 \widetilde{f} \eta_0$. Then 
 $$
 \gamma  \gamma_0 \eta_0\widetilde{f}   =\gamma  \gamma_0 \widetilde{f} \eta_0 = \gamma  \widetilde{f'} = \widetilde{f'}\gamma = \gamma_0 \widetilde{f} \eta_0 \gamma =\gamma_0 \eta_0 \gamma \widetilde{f},  
 $$
 which implies $\gamma  \gamma_0 \eta_0 = \gamma_0 \eta_0\gamma $ for all $\gamma \in \Gamma$, i.e.
 $\gamma_0 \eta_0 $ commutes with all element in $\Gamma$. Since $M = \widetilde{M}/\Gamma$ is a smooth closed manifold admitting Riemannian metrics of negative sectional curvature, any abelian subgroup of $\Gamma$ is infinite cyclic by Preissman's theorem (see \cite[Chapter 12.3]{DoC92}). This implies $\gamma_0\eta_0 = e$
 and therefore  $\widetilde{f'} =\gamma_0  \eta_0 \widetilde{f} =\widetilde{f}$. \\
Conversely, assume that $f \in  C^{0}(M,M)$ has a lift $\widetilde{f}$ such that
$\widetilde{f}\gamma = \gamma  \widetilde{f}$
for all $\gamma \in \Gamma$. We want to show that $f$ is homotopic to the identity. To see this, let us equip $M$ with a negatively curved metric $g$ and consider for each $x \in \widetilde{M} $ the unique constant speed geodesic $ \alpha_x: t \to \widetilde{f_t}(x) $ with $\widetilde{f_0}(x) =x$  and $\widetilde{f_1}(x) = \widetilde{f}(x)$ with respect to $g$. Since $\gamma \widetilde{f}(x) =\widetilde{f}\gamma(x)$ and $\gamma$ is an isometry with respect to $g$, we obtain that $\gamma \alpha_x: t\to \gamma \widetilde{f_t}(x)$ is the unique geodesic connecting $\gamma (x)$ and $\widetilde{f}\gamma(x)$ which is $\alpha_{\gamma x}$. Therefore $\gamma\widetilde{f_t}(x) = \widetilde{f_t}\gamma (x)$. The map $\widetilde{f_t}(x) : [0,1] \times \widetilde{M} \to \widetilde{M}$ is continuous and $\Gamma$-equivariant. It descends to a homotopy $H: [0,1] \times M \to M$ with $H(t,p) = pr \circ \widetilde{f_t}(x)$ where $pr (x) =p$.   
 \end{proof}

 Using the above lemma, we can therefore identify the space  $C_{\id}^{0}(M,M)$ with
$$
\mathcal{C}=\big\{\varphi \in C^{0}(\widetilde{M},\widetilde{M})  \;|  \; \varphi(\gamma x)=\gamma \varphi(x) \;, \forall x\in  \widetilde{M}, \gamma \in\Gamma \big\}.
$$
Moreover, a negatively curved Riemannian metric $g$ on $\widetilde{M}$ induces a metric $d_{\mathcal{C},g} $ on $\mathcal{C}$ defined by
$$
d_{\mathcal{C},g}(\varphi, \psi) := \max_{x\in\widetilde{M}}
 d_g(\varphi(x), \psi(x)).
$$
 Since $d_g(\varphi(x), \psi(x))= d_g(\varphi(\gamma x), \psi(\gamma x))$ for all $\gamma \in \Gamma$ and $\varphi, \psi \in \mathcal{C}$, the maximum can be attained. The metric $ d_{\mathcal{C},g}$ is  $\Gamma$-invariant in the sense that,
 $$d_ {\mathcal{C},g}(\varphi \circ \gamma, \psi \circ \gamma)=d_ {\mathcal{C},g}(\varphi, \psi)$$ for all $\gamma \in \Gamma$ and $\varphi, \psi \in \mathcal{C}$. 
It induces a metric on $C_{\id}^{0}(M,M)$ and we will still denote it as $d_{\mathcal{C},g}$.

\begin{proposition}
    The metric space 
    $(\mathcal{C},d_{\mathcal{C},g})$ is a geodesic length space.  
\end{proposition}

\begin{proof}
    Suppose $\varphi_0,\varphi_1 \in \mathcal{C}$. Given $x\in \widetilde{M}$, let $\alpha_{x}: [0,1] \longrightarrow \widetilde{M}$ be the unique geodesic with respect to $g$ connecting $\varphi_0(x)$ and $\varphi_1(x)$, i.e. 
$\alpha_{x}(0)=\varphi_0(x)$ and $\alpha_{x}(1)=\varphi_1(x)$ and with constant speed $\norm{\alpha'_x}_{g}=d_g(\varphi_0(x),\varphi_1(x))$.

 Define for $s\in[0,1]$ the map
$\psi_s: \widetilde{M}\longrightarrow\widetilde{M}$ as 
$\psi_s(x):= \alpha_{x}(s)$.
Then $\psi :[0,1] \times \widetilde{M} \to \widetilde{M} $ with $\psi(s,x) =\psi_s(x)$ is a homotopy with $\psi_0=\varphi_0$ and $\psi_1=\varphi_1$. 
Moreover, since $\varphi_0$, $\varphi_1$ are $\Gamma$-equivariant,
$\psi_s \in \mathcal{C}$. Therefore,
$$\psi_s(\gamma x)= \alpha_{\gamma x}(s) = \gamma \alpha_x(s) = \gamma  \psi_s(x).$$
The distance between $\psi_t$ and $\psi_s$ with respect to $d_{\mathcal{C},g}(\cdot, \cdot)$ is given by
$$d_{\mathcal{C},g}(\psi_t, \psi_s)= \max_{x\in\widetilde{M}} d_{g}( \alpha_x(t), \alpha_x(s)) = 
|s-t|\max_{x\in \widetilde{M}}  \norm{\alpha'_x}_{g}=|s-t| d_{\mathcal{C},g}(\varphi_0,\varphi_1)$$

Therefore $\{\psi_s\}$ is a geodesic connecting $\varphi_0$ and $\varphi_1$.

\end{proof}

Denote by $\mathcal{C}_{\Lip}(g_1,g_2)$ the space of Lipschitz maps in $\mathcal{C}$ with respect to metrics $g_1,g_2$ on $\widetilde{M}$.

\begin{proposition}  \label{prop,convexityLip}
   Let $g_1,g_2$ be the lifts of a pair of negatively curved metrics to $\widetilde M$
   and for $R >0$, denote by $\mathcal{C}_{R}(g_1,g_2)$ the subset of $\mathcal{C}_{\Lip}(g_1,g_2)$ with Lipschitz constants bounded by $R$. If $\mathcal{C}_{R_0}(g_1,g_2)$ is non-empty, it is a convex subset of $(\mathcal{C}_{R}(g_1,g_2),d_{\mathcal{C},g_2})$ for all $R \ge R_0$. In particular, $\mathcal{C}_{\Lip}(g_1,g_2)$ is a convex subset of  $(\mathcal{C},d_{\mathcal{C},g_2})$.
\end{proposition}
\begin{proof}
    Suppose $\varphi_0, \varphi_1 \in \mathcal{C}_{R_0}(g_1,g_2) $. Consider $\psi:[0,1] \times \widetilde{M} \to \widetilde{M}$ the geodesic homotopy with $\psi_0 = \varphi_0$ and $\psi_1 = \varphi_1$. We claim $\psi_s \in \mathcal{C}_{R_0}(g_1,g_2) $.
From the definition of the geodesic homotopy, it follows $\psi_s(x)=\alpha_x(s)$ for all $x\in \widetilde{M}$, where $\alpha_x:[0,1] \longrightarrow \mathbb{R}_{\geq 0} $ is the geodesic connecting $\varphi_0(x)$ and $\varphi_1(x)$. Since
$g_2$ is a metric of negative curvature, the map $s \to d_{g_2}(\alpha_x(s),\alpha_y(s))$ is a convex function. It is strictly convex when the geodesics $\alpha_x$ and $\alpha_y$ have no segment in common \cite[Theorem 2.2.1]{Jo97}. Therefore for any $x , y \in \widetilde{M}$,
\begin{align*}
d_{g_2}(\psi_s(x), \psi_s(y))&=d_{g_2}(\alpha_x(s),\alpha_y(s)) \\
&\leq s d_{g_2}(\alpha_x(0),\alpha_y(0)) + (1-s)d_{g_2}(\alpha_x(1),\alpha_y(1))\\
&=s d_{g_2}(\psi_0(x),\psi_0(y)) +(1-s) d_{g_2}(\psi_1(x),\psi_1(y)) )\\
&\leq (s\text{Lip}(\psi_0) +(1-s) \text{Lip}(\psi_1)) d_{g_1}(x,y)\\
&\leq \max \{\text{Lip}(\varphi_0),\text{Lip}(\varphi_1)\} d_{g_1}(x,y)
\end{align*}
 which implies $\psi_s \in \mathcal{C}_{R_0}(g_1,g_2) $.  
\end{proof}

Now we prove the convexity of the length function under geodesic homotopies.
\begin{proposition} \label{prop,convexityLength}   
 Let $g_1,g_2$ be the lifts of a pair of negatively curved metrics to $\widetilde M$. Let $\psi:[0,1] \times \widetilde{M} \to \widetilde{M}$   geodesic homotopy
in $\mathcal{C}_{R_0}(g_1,g_2)$ with $\psi_0 = \varphi_0$ and $\psi_1 = \varphi_1$.
Then for any Lipschitz curve $\alpha: [0,1] \to \widetilde{M}$ we have 
$$
L_{g_2}(\psi_s(\alpha)) \le (1-s) L_{g_2}(\varphi_0(\alpha)) +sL_{g_2}(\varphi_1(\alpha)).
$$
 \end{proposition}
\begin{proof}
Let $0 =t_0 <t_1 < \ldots <t_n =1 $ be any partion of the interval $[0,1]$.
We obtain 
\begin{align*}
    \sum_{i =0}^{n-1}d_{g_2}(\psi_s\circ \alpha(t_i), \psi_s\circ \alpha(t_{i+1})) \le  &
(1-s)\sum_{i =0}^{n-1}d_{g_2}(\varphi_0\circ\alpha(t_i), \varphi_0\circ\alpha(t_{i+1}))\\
&+ s\sum_{i =0}^{n-1}d_{g_2}(\varphi_1\circ\alpha(t_i), \varphi_1\circ\alpha(t_{i+1})).
\end{align*}
Therefore, the proposition follows from the definition of the length of a curve.

\end{proof}

 Denote the space of all Lipschitz maps from $(M,g_1)$ and $(M,g_2)$ that are homotopic to the identity by $\Lip_{\id}(M, g_1, g_2)$. This is a subspace of $C_{\id}^{0}(M,M)$. The \emph{Lipschitz constant} of $f\in \Lip_{\id}(M, g_1, g_2)$ is defined by 
$$\Lip(f, g_1,g_2)= \sup_{x,x' \in M, x\neq x'} \frac{d_{g_2}(f(x),f(x'))}{d_{g_1}(x,x')}.$$  

\begin{remark}
 By Lemma \ref{lem, uniquelift}, the space
        $\Lip_{\id}(M, g_1, g_2)$ is in one to one correspondence with,
$$
\mathcal{C}_{\Lip}(g_1,g_2)=\big\{\psi \in \mathcal{C} \; | \;\psi :(\widetilde{M},g_1) \to (\widetilde{M},g_2) \text{ is a Lipschitz map} \}.
$$
\end{remark}

We define for $R\geq L(g_1,g_2)$, a subset in $ \Lip_{\id}(M, g_1, g_2)$,
$$ \mathcal{A}_R(g_1,g_2):=\{f\in  \Lip_{\id}(M, g_1, g_2) \text{ }|\text{ } \Lip(f, g_1,g_2) \leq R\}.$$

 When the background metrics $g_1,g_2$ are clear in the context, we simply write the above space as $ \mathcal{A}_R$.
\begin{lemma} \label{lemma, HolderCompactEmbedding}
    For any $\alpha<1$, the subset $\mathcal{A}_R$ is compact in $C^{\alpha}(M,M)$ with respect to the $C^{\alpha}$ (weak) topology. In particular, it is a compact subset in  $(C_{\id}^{0}(M,M),d_{\mathcal{C},g})$.
\end{lemma}
\begin{proof}
   Any sequence of maps $\{f_n\}$ in $\mathcal{A}_R$ is an equicontinuous 
   family in $C_{\id}^{0}(M,M)$. By Ascoli-Arzelà Theorem, it has convergent subsequence in $C_{\id}^{0}(M,M)$ with respect to $C^0$ topology. This shows that $\mathcal{A}_R$ is a compact subset of $C_{\id}^{0}(M,M)$. For the more general statement, see \cite[Lemma 6.33]{GT01}. 
  
\end{proof}

\subsection{Geodesic stretches and weighted Lipschitz constants} \label{section, FunctionalGfirst}
\hfill\\
In this subsection, we want to discuss the relation between geodesic stretches and  the average ``stretch" of Lipschitz maps with respect to a $\phi_t^{g_1}$-invariant measure $m$.

The space of Lipschitz map from $(M, g_1)$ to $(M, g_2)$ can be identified with the Sobolev space $W^{1,\infty}(M,g_1,g_2)$ 
 (see\cite[Section 4.2.3.]{EG15}). Consider the following \emph{weighted Lipschitz constant} of a Lipschitz map $f$ with respect to a $\phi_t^{g_1}$-invariant measure $m$  given by a functional
 \begin{align*}
 &G: \Lip_{id}(M,g_1,g_2) \times \mathcal{M}(\phi^{g_1})\to \mathbb{R}_{\geq 0}\\
&G(f,m)=\int_{S^{g_1}M} \norm{Df(v)}_{g_2} dm(v),
 \end{align*}
where $Df$ is the weak derivative of $f\in W^{1,\infty}(M,g_1,g_2)$.
However, since the weak derivative $Df$ is only defined almost everywhere with respect to Liouville measure, it is not immediately obvious that the above integral makes sense. The following proposition interprets the integral for a $\phi^{g_1}_t$-invariant measure with the help of geodesic currents.

 Recal from Proposition \ref{prop:m-current}, any geodesic current $\mu$ defines a $\Gamma$-invariant and $\phi^g$-invariant measure $m^g_{\mu}$. 

\begin{proposition}\label{prop,whyIntegralWellDefined}
Let $g_1, g_2$ be two Riemannian metrics in $R^-(M)$. Let $ f \in \Lip_{\id}(M, g_1, g_2)$ be a Lipschitz map and $ m_{\mu}^{g_1}$ be the $\phi^{g_1}_t$ invariant measure induced by a current $\mu \in \mathcal C (\Gamma)$.
Choose a fundamental domain $\mathcal F \subset \widetilde M$ that satisfies Remark \ref{remark:CurrentMeasure} and for each pair $ (\xi_-, \xi_+ ) \in \partial^{(2)} \widetilde M$, denote by 
$c_{\xi_-, \xi_+}^{g_1} $ the oriented $g_1$-geodesic with $c_{\xi_-, \xi_+}^{g_1}(-\infty) = \xi_- $ and $c_{\xi_-, \xi_+}^{g_1}(\infty) = \xi_+$. Then
$$
\int_{S^{g_1}M } \| Df (v) \|_{g_2} d m_{\mu}^{g_1}(v) = \int_{\partial^{(2)} \widetilde M} L_{g_2} ( f(c_{\xi_-, \xi_+}^{g_1} \cap \mathcal F)) d\mu(\xi_-, \xi_+). 
$$
\end{proposition}
\begin{proof}
For $p_0 \in \widetilde M $ and the Hopf map $H^{g_1}_{p_0}: S^{g_1} \widetilde M \to  \partial^{(2)} \widetilde M \times \R$ given by
$$
H^{g_1}_{p_0}(v) = (v^{g_1}_-, v^{g_1}_+,  b_{ v^{g_1}_+}^{\scriptscriptstyle g_1}(p_0, \pi (v))) 
\quad \text{(see Equation (\ref{Hopfmap}) for details)},
$$
define 
$v_{(\xi_-, \xi_+)} :=  (H^{g_1}_{p_0})^{-1}(\xi_-, \xi_+, 0)$ . According to Remark \ref{remark:CurrentMeasure} and Proposition \ref{prop:m-current}  , we have,
\begin{align*}
&\int_{S^{g_1}M } \| Df (v) \|_{g_2} d m_{\mu}^{g_1}(v) =
\int_{S^{g_1} \widetilde {M}}  \| Df(v) \|_{g_2} \chi_{_{S^{g_1}\mathcal{F}}}(v)     \ dm_{\mu}^{g_1}(v)\\
= &\int_{\partial^{(2)} \widetilde M}
\int_{\R}  \| Df(\phi_t^{g_1} (v_{(\xi_-, \xi_+)}) \|_{g_2}  \chi_{_{S^{g_1}\mathcal{F}}}(\phi_t^g (v_{(\xi_-, \xi_+)}) dt \ d\mu (\xi_-, \xi_+) \\
= &\int_{\partial^{(2)} \widetilde M}
\int_{\R}  \| \frac{d}{dt}f(c^{g_1}_{v_{\scriptscriptstyle (\xi_-, \xi_+)}}(t))  \|_{g_2} \chi_{_{\mathcal{F}}}(c^{g_1}_{v_{(\xi_-, \xi_+)}} (t))dt \ d\mu (\xi_-, \xi_+) ,
\end{align*}
where $c^{g_1}_v$ is the $g_1$-unit speed geodesic with initial velocity $v$ and with abuse of notation, the unique lift of $f$ in the class $\mathcal{C}$ is still denoted as $f$.
This integral is well defined since $t \to f(c^{g_1}_{v_{(\xi_-, \xi_+)}}(t)) $
is Lipschitz and therefore the derivative exists a.e. for $t\in \mathbb{R}$ by Rademacher’s Theorem (see \cite[Theorem 3.2]{EG15}).
Now the proposition follows since
$$
\int_{\R}  \| \frac{d}{dt}f(c^{g_1}_{v_{(\xi_-, \xi_+)}}(t))  \|_{g_2} \chi_{_{\mathcal{F}}}(c^{g_1}_{v_{(\xi_-, \xi_+)}} (t))dt = L_{g_2} ( f(c_{\xi_-, \xi_+}^{g_1} \cap \mathcal F)).
$$

\end{proof}

We can also understand the weighted Lipschitz constant of  $f$ as follows.
\begin{lemma} \label{lem: Reinterpret}
Given $v\in S^{g_1}M$, denote $c^{g_1}_v$ as the $g_1$-unit speed geodesic with $c^{g_1}_v(0)= \pi(v)$ and $\dot{c}^{g_1}_v(0)= v$. If $m\in\mathcal{M}^1(\phi^{g_1})$, then
$$
\int_{S^{g_1}M } \| Df (v) \|_{g_2} d m(v) = \int_{S^{g_1}M } L_{g_2} (f(c^{g_1}_v[0,1]))d m(v).
$$
\end{lemma}
\begin{proof}
Since $m$ is $\phi^{g_1}_t$-invariant, we have
$$
\int_{S^{g_1}M } \| Df (v) \|_{g_2} d m(v) =\int_{S^{g_1}M } \| Df (\phi^{g_1}_sv) \|_{g_2} d m(v), 
$$
for all $s \in \R$. This implies that
\begin{align*}
\int_{S^{g_1}M } \| Df (v) \|_{g_2} d m(v) &= \int_0^1\int_{S^{g_1}M } \| Df (\phi^{g_1}_sv) \|_{g_2} d m(v) ds \\
&=\int_{S^{g_1}M }\int_0^1 \| Df (\phi^{g_1}_sv) \|_{g_2}  ds \; d m(v) \\
&=\int_{S^{g_1}M }\int_0^1 \| \frac{d}{ds} f(c^{g_1}_v(s)) \|_{g_2} \ ds \; d m(v) \\
&=\int_{S^{g_1}M }  L_{g_2} (f(c^{g_1}_v[0,1]) \ d m(v).
\end{align*}
\end{proof}

Let $ f\in \Lip_{\id}(M,g_1,g_2)$. According
 to Lemma \ref{lem, uniquelift}, there exists a unique lift $ \widetilde{f}: \widetilde M \to \widetilde M$  such that $\widetilde{f}\gamma = \gamma  \widetilde{f}$ for all $\gamma \in \Gamma$. Then given a fundamental domain $\mathcal F$, there exists a constant $c_f(g_2)$ such that 
 $d_{g_2}(x, \widetilde f (x)) \le c_f(g_2)$ for all $x \in \mathcal F $
  and by the $\Gamma$- equivariance of $\widetilde{f}$, we obtain 
 $$d_{g_2}(x, \widetilde f (x)) \le c_f(g_2)$$
for all $x \in \widetilde M$.
 In particular,  $\widetilde{f}$ has a continuous extension $\widetilde{f} :\partial \widetilde M \to  \partial \widetilde M $ given by the identity.
 Furthermore, the following extension of Lemma \ref{lem,Busem-distance}
 holds for any fixed Lipschitz map $f\in \Lip_{\id}(M,g_1,g_2)$.
 \begin{lemma}\label{lem,Busem-distance-Lip}
 There exists a constant $k_f(g_1, g_2)$ such that for all $v \in S^{g_1} \widetilde M$ and $\xi = v^{g_1}_+$ we have,
$$
| d_{g_2}(\widetilde{f}\pi (v) , \widetilde{f}\pi(\phi_t^{g_1} v)) -b_\xi^{g_2}(\widetilde{f}\pi(v),  \widetilde{f}\pi(\phi_t^{g_1} v))| \le k_f(g_1, g_2).
$$
\end{lemma}
\begin{proof}
By the discussion before this lemma,
$$
|  d_{g_2}(\widetilde{f}\pi (v) , \widetilde{f}\pi(\phi_t^{g_1} v)) - d_{g_2}(\pi (v) , \pi(\phi_t^{g_1} v)) | \le 2 c_f(g_2),
$$
and by cocycle property (\ref{prop1bus}) of the Busemann function,
\begin{align*}
    &|b_\xi^{g_2}(\widetilde{f}\pi(v),  \widetilde{f}\pi(\phi_t^{g_1} v)) 
- b_\xi^{g_2}(\pi(v),  \pi(\phi_t^{g_1} v))| \\ 
=& |b_\xi^{g_2}(\widetilde{f}\pi(v),   \pi (v)) 
- b_\xi^{g_2}(\widetilde{f} \pi(\phi_t^{g_1} v), \pi(\phi_t^{g_1} v)) |\\
\le &  d_{g_2}(\widetilde{f}\pi (v) , \pi (v)) + d_{g_2}(\widetilde{f}\pi(\phi_t^{g_1} v), \pi(\phi_t^{g_1} v)) \le 2 c_f(g_2).
\end{align*}
Furthermore, Lemma \ref{lem,Busem-distance} implies that there exists a constant $R_1=R_1(g_1,g_2)>0$ so that
$$  
| d_{g_2}(\pi (v) , \pi (\phi_t^{g_1} v)) -b_\xi^{g_2}(\pi(v),  \pi(\phi_t^{g_1} v))| \le R_1(g_1, g_2).
$$
Using this inequalities, we obtain
\begin{align*}
&|  d_{g_2}(\widetilde{f}\pi (v) , \widetilde{f}\pi(\phi_t^{g_1} v)) -b_\xi^{g_2}(\widetilde{f}\pi(v),  \widetilde{f}\pi(\phi_t^{g_1} v))| \\
= &| d_{g_2}(\widetilde{f}\pi (v) , \widetilde{f}\pi(\phi_t^{g_1} v)) -
  d_{g_2}(\pi (v) , \pi (\phi_t^{g_1} v)) 
+ d_{g_2}(\pi (v) , \pi (\phi_t^{g_1} v)) \\
-&b_\xi^{g_2}(\pi(v),  \pi(\phi_t^{g_1} v))
  \ + b_\xi^{g_2}(\pi(v),  \pi(\phi_t^{g_1} v))-b_\xi^{g_2}((\widetilde{f}\pi(v)),  \widetilde{f}\pi(\phi_t^{g_1} v))| \\
   \le &4 c_f(g_2) +R_1(g_1, g_2) =: k_f(g_1, g_2).
 \end{align*}
\end{proof}

Define a function $l_f \colon  S^{g_1}M \times \mathbb{R} \to \mathbb{R}_{\geq 0}$ by
$$ l_f (v, t): =  d_{g_2}(\widetilde{f}\pi (\widetilde{v}) , \widetilde{f}\pi(\phi_t^{g_1} \widetilde{v})), $$
where $\widetilde{v}$ is a lift of $v$ on $S^{g_1}\widetilde M $. This is well defined because the right-hand side is $\Gamma$-invariant.
Consider a $\phi^{g_1}$-invaraint probability measure $m \in \mathcal{M}^1(\phi^{g_1})$. By the subadditive ergodic theorem and the subadditivity of the function $l_f (v, t)$, 
the limit 
$\lim\limits_{t \to \infty} \frac{ 1}{t}l_f (v, t)$
exists for $m$-almost every $v \in S^{g_1} M $. On the other hand, from the proof of Lemma \ref{lem,Busem-distance-Lip}, when the limit exists for $v \in S^{g_1} M $, recalling the Definition of geodesic stretch (see Section \ref{subsection, BasicGeodesicStretch}), we obtain
$$
\lim_{t \to \infty} \frac{1}{t}l_f (v, t) =\lim_{t \to \infty} \frac{1}{t}
 d_{g_2}(\pi (\widetilde{v}) , \pi (\phi_t^{g_1} \widetilde{v}))= I_m(g_1,g_2,v). 
$$
In particular, this limit is independent of the choice of $f$ in $\Lip_{\id}(M,g_1,g_2)$.

We next show that the geodesic stretch with respect to a $\phi^{g_1}$-invariant probability measure is always bounded above by the average `` stretch " of a Lipschitz map with respect to that measure.

\begin{lemma} \label{lem, GeodesicStretchLipStretch}
   Let $g_1, g_2$ be two Riemannian metrics in $R^-(M)$ and let $m\in\mathcal{M}^1(\phi^{g_1})$. Given any Lipschitz map $f\in \Lip_{\id}(M,g_1,g_2)$, the geodesic stretch $I_{m}(g_1, g_2)$ satisfies
    \begin{align*}
I_m(g_1, g_2) \leq \int_{S^{g_1} M}  \| Df(v) \|_{g_2} dm.
\end{align*}
\end{lemma}
\begin{proof}
    
By Lemma \ref{lem,Busem-distance-Lip}, we obtain for $m$ almost every $v\in S^{g_1} M$,

\begin{align*}
\lim_{t \to \infty} \frac{1}{t} l_f (v, t) & = \lim_{t \to \infty} \frac{1}{t}
b_\xi^{g_2}(\widetilde{f}\pi(\widetilde{v}),  \widetilde{f}\pi(\phi_t^{g_1} \widetilde{v}))\\
&=  \lim_{t \to \infty} \frac{1}{t} \int_0^t \frac{d}{ds} b_\xi^{g_2}(\widetilde{f}\pi(\widetilde{v}),  \widetilde{f}\pi(\phi_s^{g_1} \widetilde{v}))  ds\\
&=  \lim_{t \to \infty} \frac{1}{t} \int_0^t g_2( B^{g_2}(\widetilde{f}\pi(\phi_s^{g_1} \widetilde{v})), \widetilde{v}^{g_1}_+), D\widetilde{f}
(\phi_s^{g_1} \widetilde{v} )) ds,
\end{align*}
where $\widetilde{v}$ denotes a lift of $v$ on $S^{g_1}\widetilde M $ and $\xi=\widetilde{v}^{g_1}_{+}$. Since the functions  $g_2( B^{g_2}(\widetilde f\pi( \widetilde{v}), \widetilde{v}^{g_1}_+), D\widetilde{f}(\widetilde{v}))$ and $ \| D\widetilde{f}(\widetilde{v}) \|$ are  $\Gamma$-invariant, we obtain from the Birkhoff ergodic theorem for invariant measures (\cite[Theorem 1.14]{Wa82}) and the definition of the geodesic stretch, 
\begin{align*}
I_{m}(g_1, g_2) &= \int_{S^{g_1} M}\lim_{t \to \infty} \frac{1}{t}l_f (v, t) dm(v)\\
 &= \int_{S^{g_1} M}g_2( B^{g_2}(\widetilde f\pi( \widetilde{v}), \widetilde{v}^{g_1}_+), D\widetilde f(\widetilde{v})) dm(v)\\
 & \le \int_{S^{g_1} M}  \| Df(v) \|_{g_2} dm(v),
\end{align*}
where $\|B^{g_2}(\widetilde f\pi( \widetilde{v}),\widetilde{v}^{g_1}_+)\|_{g_2} =1$ follows from the fact that the norm of the gradient of the Busemann function is identically equal to $1$. Furthermore, 
$$
I_{m}(g_1, g_2) = \int_{S^{g_1} M}  \| Df(v) \|_{g_2} dm(v)
$$
if and only if for $m$ almost every $v\in S^{g_1} M$ we have
$$
 D\widetilde{f}(\widetilde{v}) = \| Df(v) \|_{g_2}  B^{g_2}( \widetilde f \pi(\widetilde{v}), \widetilde{v}^{g_1}_+).
$$

\end{proof}

The above Lemma holds for any Lipschitz map homotopic to identity. This motivates us to define
\begin{definition}
Given $g_1,g_2\in R^-(M)$ and given $m\in \mathcal{M}^1(\phi^{g_1})$, we define the \emph{$m$-weighted least Lipschitz constant} $L_{m}(g_1,g_2)$ as   
 $$L_m(g_1,g_2):=\inf\limits_{f\in \Lip_{\id}(M, g_1, g_2)}\int_{S^{g_1}M} \norm{Df(v)}_{g_2} dm.$$
\end{definition}
\noindent Similarly, for a geodesic current $\mu\in \mathcal{C}(\Gamma)$, we define the  \emph{$\mu$-weighted least Lipschitz constant} as $L_{\mu}(g_1,g_2):= L_{\hat{m}^{g_1}_{\mu}}(g_1,g_2)$, which is the weighted least Lipschitz constant of its associated invariant probability measure $\hat{m}^{g_1}_{\mu}$.

 As a simple application of  Lemma \ref{lem, GeodesicStretchLipStretch}, we obtain

\begin{corollary} \label{cor, MinGeqGeodesicStretch}
Let $g_1, g_2$ be two Riemannian metrics in $R^-(M)$. Given $m\in \mathcal{M}^1(\phi^{g_1})$, we have
$$ I_{m}(g_1,g_2) \leq  L_{m}(g_1,g_2).$$
\end{corollary}

The following theorem discusses when the equality holds in Lemma \ref{lem, GeodesicStretchLipStretch}. We state this theorem in the language of geodesic currents.

\begin{theorem} \label{prop,GeodesicToGeodesic}
Given $g_1, g_2 \in R^{-}(M)$ and $\mu \in \mathcal{C}(\Gamma)$ a geodesic current and $f \in \Lip_{\id}(M,g_1,g_2)$ such that

$$
I_{\mu}(g_1, g_2) = \int_{S^{g_1} M}  \| Df(v) \|_{g_2} d\hat{m}^{g_1}_{\mu}(v).
$$ 
Then for  all $(\xi_-, \xi_+) \in \partial^{(2)} \widetilde M$ contained in 
$\supp \mu$, its lift $\widetilde f \in \mathcal{C}$ maps $g_1$-geodesics to corresponding $g_2$-geodesics up to parametrization, i.e.
$$
\widetilde f(c_{(\xi_-, \xi_+)}^{g_1}) = c_{(\xi_-, \xi_+)}^{g_2},
$$
where $c_{(\xi_-, \xi_+)}^{g_1}$ denotes the $g_1$-geodesic with backward endpoint $\xi_-$ and forward endpoint $\xi_+$. In particular, if $\mu$ has full support on $ \partial^{(2)} \widetilde M$, then $\widetilde f$ maps every $g_1$-geodesic to the corresponding $g_2$-geodesic up to parametrization.

\end{theorem}

\begin{proof}
The geodesic stretch satisfies from the deduction of Lemma \ref{lem, GeodesicStretchLipStretch},
\begin{align*}
I_{\mu}(g_1, g_2)& =\int_{S^{g_1} M} \int_{0}^{1} 
g_2( B^{g_2}(\widetilde f\pi( \phi^{g_1}_s \widetilde v), \widetilde v^{g_1}_+), D \widetilde f(\phi^{g_1}_s \widetilde v)) ds  d\hat{m}^{g_1}_{\mu}(v) \\
&= \int_{S^{g_1} M} \int_{0}^{1} \frac{d}{ds} b^{g_2}_{v_+}(x_0, \widetilde f( c^{g_1}_{\widetilde v} (s))) ds \ d\hat{m}^{g_1}_{\mu}(v)\\
&= \int_{S^{g_1} M} b^{g_2}_{v_+}(\widetilde f( c^{g_1}_{\widetilde v}(0) ), \widetilde f( c^{g_1}_{\widetilde v}(1) )) d\hat{m}^{g_1}_{\mu}(v),
\end{align*}
where $x_0$ is a base point on $\widetilde M$ and $\widetilde v$ is a lift of $v$; $c^{g_1}_{\widetilde v}(t)$ is the unit speed $g_1$-geodesic on $\widetilde M$ starting from $\widetilde v$. To simplify notation, we write $v_+$ to mean $ \widetilde v^{g_1}_+$. 

On the other hand, the average ``stretch" of the Lipschitz map $f$ is 
$$
\int_{S^{g_1} M}  \| Df(v) \|_{g_2} d\hat{m}^{g_1}_{\mu}(v) =
 \int_{S^{g_1} M} L_{g_2} ( \widetilde f(  c^{g_1}_{\widetilde v}([0, 1]))) d\hat{m}^{g_1}_{\mu}(v),
$$
Since the Busemann function is 1-Lipschitz, we have

\begin{equation}\label{eqtn,bLd}
b^{g_2}_{v_+} (\widetilde f( c^{g_1}_{\widetilde v}(0) ), \widetilde f( c^{g_1}_{\widetilde v}(1) ) )\le d_{g_2}(\widetilde f( c^{g_1}_{\widetilde v}(0) ),\widetilde f(c^{g_1}_{\widetilde v} (1) ))
\le  L_{g_2} ( \widetilde f(   c^{g_1}_{\widetilde v}([0,  1]))).
\end{equation}

Therefore
$$
I_{\mu}(g_1, g_2) =\int_{S^{g_1} M}  \| Df(v) \|_{g_2} d\hat{m}^{g_1}_{\mu}(v)
$$

implies
\begin{equation}\label{eqtn,b=Lfullmeasure}
b^{g_2}_{v+} (\widetilde f( c^{g_1}_{\widetilde v} (0) ), \widetilde f( c^{g_1}_{\widetilde v}(1)) ) = L_{g_2} ( \widetilde f(  c^{g_1}_{\widetilde v}([0, 1]))),
\end{equation}
for $\hat{m}^{g_1}_{\mu}$ almost every $v \in S^{g_1} M$. Here we use the fact that the above functions on $\widetilde M$ are $\Gamma$-invariant.

We want to show that the Equation \eqref{eqtn,b=Lfullmeasure} in fact holds for every $v\in \supp \hat{m}^{g_1}_{\mu}$. It is clear that the Busemann function $b^{g_2}_{v_+} (\widetilde f( c^{g_1}_{\widetilde v}(0) ), \widetilde f( c^{g_1}_{\widetilde v}(1) )$ is continuous in $\widetilde v$. However we need to be careful that $L_{g_2} ( \widetilde f(  c^{g_1}_{\widetilde v}([0, 1]))$ is only lower-semicontinuous in $\widetilde v$: if $\widetilde v_n \to \widetilde v_0$, then the curves $\widetilde f\circ   c^{g_1}_{\widetilde v_n}$ approximates $\widetilde f\circ \widetilde c^{g_1}_{v_0}$ uniformly in the $C^0$ topology. Then the lengths of these curves satisfy (see, for example, \cite[Chapter I.1, Proposition 1.20 (7)]{BH99}),
\begin{equation}  \label{equation liminfLength}
 L_{g_2} (\widetilde f\circ   c^{g_1}_{\widetilde v_0}([0,1]))\leq \liminf\limits_{n\to\infty} L_{g_2} (\widetilde f\circ   c^{g_1}_{\widetilde v_n}([0,1]) ).
\end{equation}
Now suppose $v_0\in \supp \hat{m}^{g_1}_{\mu}$. For a lift $\widetilde v_0$ of $v_0$, let us consider an open ball $B_{\frac{1}{n}}(\widetilde v_0)$ of radius $\frac{1}{n}$ with respect to the Sasaki metric centered at $\widetilde v_0$. It has positive measure with respect to (the lift of) $\hat{m}^{g_1}_{\mu}$. We can therefore find $\widetilde v_n \in B_{\frac{1}{n}}(\widetilde  v_0)$ that satisfies Equation \eqref{eqtn,b=Lfullmeasure} and that $\widetilde  v_n\to \widetilde  v_0$.
Therefore, by inequality \eqref{eqtn,bLd} and \eqref{equation liminfLength}, we have,
\begin{align*}
\liminf_{n\to\infty}L_{g_2} ( \widetilde f(  c_{\widetilde v_n}^{g_1}([0, 1])))&\geq L_{g_2} ( \widetilde f(   c^{g_1}_{\widetilde v_0}([0, 1])))
\geq  b^{g_2}_{v_{0+}} (\widetilde f(  c^{g_1}_{\widetilde v_0} (0) ), \widetilde f(  c^{g_1}_{\widetilde v_0}(1) ))\\ 
&=\lim\limits_{n\to\infty} b^{g_2}_{v_{n+}} (\widetilde f(  c^{g_1}_{\widetilde v_n} (0) ), \widetilde f(  c^{g_1}_{\widetilde v_n}(1) )).
\end{align*}

By assumption, Equation \eqref{eqtn,b=Lfullmeasure} holds for all $\widetilde  v_n$. The above inequalities imply 

\begin{equation*}
L_{g_2} ( \widetilde f(   c^{g_1}_{\widetilde v_0}([0, 1])))
=  b^{g_2}_{v_{0+}} (\widetilde f(  c^{g_1}_{\widetilde v_0} (0) ), \widetilde f(  c^{g_1}_{\widetilde v_0}(1) )).
\end{equation*}

Hence  Equation \eqref{eqtn,b=Lfullmeasure} in fact holds for (lifts of) every $v\in \supp \hat{m}^{g_1}_{\mu}$.

To conclude the statement, we notice since  $\supp \hat{m}^{g_1}_{\mu}$ is
$\phi^{g_1}_t$-invariant, for all $t \in \R$,
\begin{align} \label{equation Length-distance}
L_{g_2} (\widetilde f( c^{g_1}_{\widetilde v}([t, t+1])))&= L_{g_2} (\widetilde f( c_{\phi^{g_1}_t \widetilde v}^{g_1}([0, 1]))) \\
&=  d_{g_2}(\widetilde f( c^{g_1}_{\phi^{g_1}_t \widetilde v}(0) ),\widetilde f(c^{g_1}_{\phi^{g_1}_t \widetilde v} (1) )) \nonumber \\
&= d_{g_2}(\widetilde f( c^{g_1}_{\widetilde v}(t) ),\widetilde f( c^{g_1}_{\widetilde v}(t+1) )) \nonumber.
\end{align}

This implies that for all $t \in \R$, the curve $ [0, 1] \to \widetilde M$ given by
$$s \mapsto  \widetilde f( c^{g_1}_{\widetilde v}(t+s))$$ agrees up to parametrization with the  $g_2$-geodesic $ c^{g_2}$ connecting 
 $\widetilde f ( c^{g_1}_{\widetilde v}(t))$ to  $\widetilde f( c^{g_1}_{\widetilde v}(t+1))$. Hence we conclude,
for all $v \in  \supp \hat{m}^{g_1}_{\mu}$ and $(v_-, v_+) \in \partial^{(2)} \widetilde M $, we have that
$$
\widetilde f( c_{(v_-, v_+)}^{g_1}) =  c_{(v_-, v_+)}^{g_2}
$$
where $ c^{g_j}_{(v_-, v_+)}$ is the unparametrized geodesic with $c^{g_j}_{(v_-, v_+)}(-\infty) =v_-$
and  $ c^{g_j}_{(v_-, v_+)}(+\infty) =v_+$, for $j=1,2$.

\end{proof}

Combining all we have discussed leads to the key statement in this subsection.

\begin{corollary} \label{cor,MainThm3}
 Suppose $g_1,g_2 \in R^-(M)$. For any maximally stretched measure, we have
\begin{align*} 
S(g_1, g_2) \leq L_m(g_1,g_2).
\end{align*}
Furthermore, if for some maximally stretched measure $m_0$, there exists  $f_0 \in \Lip_{\id}(M, g_1, g_2)$  such that $$S(g_1,g_2) =\int_{S^{g_1}M} \norm{Df_0(v)}_{g_2} dm_0(v)=L_{m_0}(g_1,g_2) ,$$
then  $f_0$ maps any $g_1$-geodesic in the support of $m_0$ to a corresponding $g_2$-geodesic (up to parametrization).
\end{corollary}
\subsection{Some further properties of the functional $G$} 
We discuss in this subsection lower semicontinuity annd convexity of the functional $G$ given by 
 \begin{align*}
 &G: \Lip_{id}(M,g_1,g_2) \times \mathcal{M}(\phi^{g_1})\to \mathbb{R}_{\geq 0}\\
&G(f,m)=\int_{S^{g_1}M} \norm{Df(v)}_{g_2} dm(v).
 \end{align*}

We have the following lemma for lower semicontinuity.
\begin{lemma} 
\label{lem,correctCompactness}
Let $f_0 \in \Lip_{\id}(M,g_1,g_2)$ and  $m_0 \in\mathcal{M}^{1}(\phi^{g_1})$, for any $\varepsilon>0$ there exists a neighborhood $U_1$ of $f_0$  in  $(\Lip_{\id}(M,g_1,g_2), d_{\mathcal{C}, g_2})$ and a neighborhood $U_2$ of $m_0$ in  $\mathcal{M}^{1}(\phi^{g_1})$ such that
    $$
     G(f_0,m_0)  \leq   G(f,m) + \varepsilon
     $$
     for all $f \in U_1$ and $m \in U_2$, where $U_1=U_1(f_0,m_0,\varepsilon)$ and $U_2=U_2(f_0,m_0,\varepsilon)$ both depend on $f_0$, $m_0$ and $\varepsilon$.
   
\end{lemma} 
 \begin{proof}
Consider for each $\ell \in \N$ the partition $0 \le t_0^\ell < \ldots  <t_{2^\ell}^\ell =1 $ of $[0, 1]$ with 
$t_i^\ell = \frac{i}{2^\ell}$.
Using the triangle inequality we obtain that for each $v \in S^{g_1}M$ and all $f \in \Lip_{\id}(M,g_1,g_2)$, the sequence
$$
\ell \mapsto \sum_{i=0}^{2^\ell-1} d_{g_2}(f\circ c^{g_1}_v(t_i^\ell), f\circ c^{g_1}_v(t_{i+1}^\ell))
$$
is monotonically increasing and from the definition of the length functional we obtain
$$
L_{g_2} ( f(c^{g_1}_v[0,1])) = \lim_{\ell \to \infty} \sum_{i=0}^{2^\ell-1} d_{g_2}(f\circ c^{g_1}_v(t_i^\ell), f\circ c^{g_1}_v(t_{i+1}^\ell)).
$$
Then Lebesgue's dominated convergence theorem yields
\begin{align*}
 G(f,m) &=\int_{S^{g_1}M} L_{g_2} ( f(c_{v}^{g_1}[0,1])) dm \\
 &= \lim_{\ell \to \infty} \int_{S^{g_1}M}
  \sum_{i=0}^{2^\ell-1} d_{g_2}(f\circ c^{g_1}_v(t_i^\ell), f\circ c^{g_1}_v(t_{i+1}^\ell)) dm.
 \end{align*}

Let $\varepsilon >0$. For a fixed $f_0\in \Lip_{\id}(M,g_1,g_2)$ and a fixed $m_0\in\mathcal{M}^{1}(\phi^{g_1})$, there exists $\ell=\ell(f_0,m_0,\varepsilon) \in \N$
 such that
 \begin{equation} \label{eqtn, partition}
  G(f_0,m_0)  \leq \int_{S^{g_1}M}
  \sum_{i=0}^{2^\ell-1} d_{g_2}(f_0 \circ c^{g_1}_v(t_i^\ell), f_0 \circ c^{g_1}_v(t_{i+1}^\ell)) dm_0 + \frac{\varepsilon}{4}.
\end{equation}

Choose a neighborhood $U_1=U_1(f_0,m_0,\varepsilon)$ of $f_0$  such that for all $f \in U_1$,
 $$d_{\mathcal{C},g_2}(f ,f_0)\leq \frac{\varepsilon}{ 2^{\ell+2}}.$$
Hence for all $v \in S^{g_1}M$ and $t \in [0,1]$, we have
$$ d_{g_2} (f\circ c^{g_1}_v(t),f_0 \circ c^{g_1}_v(t))\leq d_{\mathcal{C},g_2}(f,f_0)\leq \frac{\varepsilon}{ 2^{\ell+2}}.$$
and 
 by triangle inequality, 
 we obtain
 \begin{align*} 
 &|d_{g_2}(f_0 \circ c^{g_1}_v(t_i^\ell), f_0 \circ c^{g_1}_v (t_{i+1}^\ell))- d_{g_2}(f\circ c^{g_1}_v(t_{i}^{\ell}), f   \circ c^{g_1}_v(t_{i+1}^{\ell}))| \\
\leq &  d_{g_2}(f_0 \circ c^{g_1}_v(t_{i}^{\ell}), f \circ c^{g_1}_v (t_{i}^{\ell})) +  d_{g_2}(f_0 \circ c^{g_1}_v(t_{i+1}^{\ell}), f \circ c^{g_1}_v (t_{i+1}^{\ell})) \\
 \leq & \frac{\varepsilon}{  2^{\ell+1}}.
  \end{align*} 
Therefore 
\begin{equation}\label{eqtn, DistSum}
 \sum_{i=0}^{2^\ell-1} d_{g_2}(f_0\circ c^{g_1}_v(t_i^\ell), f_0 \circ c^{g_1}_v(t_{i+1}^\ell)) \le  \sum_{i=0}^{2^\ell-1} d_{g_2}(f\circ c^{g_1}_v(t_i^\ell), f \circ c^{g_1}_v(t_{i+1}^\ell))  + \frac{\varepsilon}{2}.  
\end{equation}

Since $v\to \sum_{i=0}^{2^\ell-1} d_{g_2}(f_0\circ c^{g_1}_v(t_i^\ell), f_0\circ c^{g_1}_v(t_{i+1}^\ell))$ is a continuous function, by  weak-* topology, we can choose a neighborhood $U_2=U_2(f_0,m_0,\varepsilon)$ of $m_0$, so that for $m \in U_2$,

\begin{align} \label{eqtn, TwoMeasures}
    \bigg| \int_{S^{g_1}M} &
  \sum_{i=0}^{2^\ell-1} d_{g_2}(f_0\circ c^{g_1}_v(t_i^\ell), f_0\circ c^{g_1}_v(t_{i+1}^\ell)) dm \\
  &- \int_{S^{g_1}M}
  \sum_{i=0}^{2^\ell-1} d_{g_2}(f_0\circ c^{g_1}_v(t_i^\ell), f_0\circ c^{g_1}_v(t_{i+1}^\ell)) dm_0\bigg| \leq \frac{\varepsilon}{4}.
  \end{align} 

Combining Equation \ref{eqtn, partition}, \ref{eqtn, DistSum} and \ref{eqtn, TwoMeasures} with an application of Lebesgue's dominated convergence theorem leads to,
  $$ G(f_0,m_0) \leq \int_{S^{g_1}M}
  \sum_{i=0}^{2^\ell-1} d_{g_2}(f \circ c^{g_1}_v(t_i^\ell), f\circ c^{g_1}_v(t_{i+1}^\ell)) dm +\varepsilon \leq G(f, m)+\varepsilon.$$

\end{proof}

\begin{proposition}\label{prop,PropertiesG3}
    The map $G: \Lip_{id}(M,g_1,g_2) \times \mathcal{M}^{1}(\phi^{g_1}) \longrightarrow \mathbb{R}_{\geq 0}$ with $(f,m) \mapsto G(f, m) $ is lower semicontinuous.
\end{proposition}
\begin{proof}
    This follows from Lemma \ref{lem,correctCompactness}.
\end{proof}
For a fixed $m_0\in \mathcal{M}^{1}(\phi^{g_1})$, let us denote  $L_{m_0}: \Lip_{id}(M,g_1,g_2)  \longrightarrow \mathbb{R}_{\geq 0}$ as $L_{m_0}(f):=G(f,m_0)$. Then

 \begin{proposition} \label{prop,PropertiesG}
     For each $m_0 \in\mathcal{M}^{1}(\phi^{g_1})$, the map $ f\to L_{m_0}(f)$ is convex. 
\end{proposition}

 \begin{proof}
    For $R\geq \L(g_1,g_2)$, recall from Section \ref{section, PreliminaryLipMaps} that the sets $ \Lip_{id}(M,g_1,g_2) $ and $\mathcal{A}_R$ are convex.
 Take two maps $\varphi_0$ and $\varphi_1$ in $\mathcal{A}_R$ and connect them by a geodesic $\psi_s$  as in Proposition \ref{prop,convexityLip}. By Proposition \ref{prop,convexityLength},
$$ L_{g_2} ( \psi_s(c_{v}^{g_1}[0,1])) \leq s \cdot L_{g_2} ( \varphi_0(c_{v}^{g_1} [0,1])  +(1-s)L_{g_2} ( \varphi_1(c_{v}^{g_1} [0,1]) $$
for all $v \in S^{g_1}M$. Choose $m  \in\mathcal{M}^{1}(\phi^{g_1})$. Integrating with respect to the probability measure $m$ yields

$$G(\psi_s,m) \leq s G(\varphi_0,m) +(1-s) G(\varphi_1,m).$$
   
 \end{proof}

\begin{proposition} \label{prop,infAchieved}
    Given $m\in \mathcal{M}^1(\phi^{g_1})$, for $R\geq \L(g_1,g_2)$, when restricted to $\mathcal{A}_R$ we have
    $$\inf_{f\in \mathcal{A}_R}L_{m}(f)=\min_{f\in \mathcal{A}_R}L_{m}(f).$$
\end{proposition}
\begin{proof}
 Because
 $\mathcal{A}_R$ is compact and $L_{m}(f)=G(f,m)$ is lower semicontinuous with respect to $f$, the infimum is achieved by some $f_0\in \mathcal{A}_R$. 
\end{proof}

 \section{Best Lipschitz maps} \label{section,BLM}

 In this section, we discuss extremal Lipschitz maps and their relation to maximal stretches in certain settings of negatively curved manifolds. We will emphasize the study of stretch loci, which were originally introduced in \cite{GK17} for hyperbolic spaces.
\subsection{Best Lipschitz maps and the maximal stretch }
\hfill\\
We always assume that the Riemannian metrics $g_1$ and $g_2$ are of negative sectional curvature. The Lipschitz constant of a Lipschitz map $f$ with respect to $g_1$ and $g_2$ is

$$\Lip(f, g_1,g_2)= \sup_{\substack{x,x' \in M \mathstrut \\ x\neq x'}} \frac{d_{g_2}(f(x),f(x'))}{d_{g_1}(x,x')}.$$  

Recall that $ \Lip_{\id}(M, g_1, g_2)$ is the set of Lipschitz maps from $(M,g_1)$ to $(M, g_2)$ that are homotopic to the identity. We would like to study extremal Lipschitz maps in this set.
\begin{definition}
We define the \emph{least Lipschitz constant} from $g_1$ to $g_2$ as
$$
L(g_1,g_2) := \inf \{ \Lip(f, g_1,g_2) \mid f\in \Lip_{\id}(M,g_1,g_2) \}.
$$
A Lipschitz map in $ \Lip_{\id}(M,g_1,g_2)$ that realizes the least Lipschitz constant $L(g_1,g_2)$ is called a \emph{best Lipschitz map}.
\end{definition}

\begin{remark}
    The existence of best Lipschitz maps follows from the compactness of $M$ and Arzelà-Ascoli Theorem.
\end{remark}

We also recall for $g_1,g_2 \in R^-(M)$,
the maximal stretch of $g_1$ to $g_2$ is defined by
$$
S(g_1,g_2) = \max_{\mu \in \mathcal{C}(\Gamma)} I_\mu(g_1,g_2)= \max_{\mu \in \mathcal{C}(\Gamma)} \int a_{g_1,g_2}(v) d\hat{m}^{g_1}_{\mu}.
$$

The following relation between $L(g_1,g_2)$ and $S(g_1,g_2)$ is obvious.
\begin{proposition}[\cite{Th98}] \label{prop,LgeqS}
Given $g_1,g_2\in R^-(M)$, the following inequality always holds:
$$
S(g_1, g_2) \le L(g_1,g_2).
$$
\end{proposition}
\begin{proof}
Let $ f \colon  (M, g_1)  \to (M, g_2)$ be a best Lipschitz map. Let $\gamma=\gamma^{g_1}$ be any closed geodesic in $(M, g_1)$. Since $f$ is homotopic to the identity, we have
$ [f \circ \gamma ] =[\gamma]$ and 
$$
 \int\limits_0^{ \ell_{g_1}([\gamma])} a_{g_1,g_2} (\dot \gamma(t)) dt=  \ell_{g_2}([\gamma]) \le   L_{g_2}(f \circ \gamma) \le  \ell_{g_1}([\gamma]) \cdot L(g_1,g_2),
 $$
 where $L_{g_2}(f \circ \gamma)$ is the $g_2$-length of the Lipschitz curve $f\circ \gamma$.

 By Remark \ref{remark, supremumGeodesicStretch},
 $$
 S(g_1, g_2) = \sup_{[\gamma]\in [\Gamma]} \frac{1}{  \ell_{g_1}([\gamma])} \int\limits_0^{\ell_{g_1}([\gamma])} a_{g_1,g_2} (\dot \gamma(t)) dt  \le L(g_1,g_2)
 $$ 
 which yields the claim.
\end{proof}

It is an interesting but not easy question to address when the equality holds. A well known theorem of \cite{Th98} states that in the Teichm\"uller space, the maximal stretch alway equals to the least Lipschitz constant. We will have further discussion and provide some partial pictures of this question in negatively curved metrics setting in later subsections. In Section \ref{subsection,NotLamination}, we provide a large class of negatively curved metrics with $S(g_1,g_2)=L(g_1,g_2)$. In Appendix \ref{ExampleS<L}, we will also discuss an example of $g_1,g_2\in R^-(M)$ with $L(g_1,g_2)> S(g_1,g_2)$.

  \subsection{Best Lipschitz maps and the stretch locus}
\hfill\\
In \cite{GK17}, Gu\'{e}ritaud and Kassel introduce the stretch locus for best Lipschitz maps in the setting of hyperbolic spaces $\mathbb{H}^n$ for $n\geq 2$. We consider the same object in variable negatively curved manifolds in this subsection. In the sequel, we will discuss relations of stretch loci with other geometric objects.

\begin{definition} \label{defn, stretchLocus}
   Suppose $g_1,g_2 \in R^-(M)$. For $f\in \Lip_{\id}(M,g_1, g_2)$, we define the \emph{stretch locus} of $f$ from $g_1$ to $g_2$ as,
    $$E_f(g_1, g_2):= \{x\in M| \Lip(f, g_1,g_2)=\Lip_x(f, g_1,g_2) \},$$
    where $\Lip_x(f, g_1,g_2)$ is the \emph{local Lipschitz constant} at $x$ defined by 
    $$\Lip_x(f, g_1,g_2)= \inf_{r>0} \Lip(f|_{\overline{B}^{g_1}_x(r)}, g_1,g_2)$$
    and $\Lip(f|_{\overline{B}^{g_1}_x(r)}, g_1,g_2)$ denotes the Lipschitz constant of $f$ restricted to the closed ball $\overline{B}^{g_1}_x(r)$ of radius $r$ centered at $x\in M$ with respect to the metric $g_1$.
\end{definition}

   We list some useful properties of local Lipschitz constants and of the stretch locus of $f\in \Lip_{\id}(M,g_1, g_2)$.  We refer to \cite[Lemma 2.9]{GK17} for details and proofs.

   Recall that  $\mathcal{C}$ is the space of $\Gamma$-equivariant continuous maps from $\widetilde{M}$ to $\widetilde{M}$.
\begin{remark} \label{remark, LocalLip}
 The following statements hold. 
    \begin{enumerate}
        \item The local Lipschitz constant $ x\to \Lip_x(f, g_1,g_2) $ is upper semicontinuous.
        \item  Consider the unique lift $\widetilde{f}\in \mathcal{C}$ of $f\in \Lip_{\id}(M,g_1, g_2)$ on the universal cover $\widetilde M$. Then for any convex subset $X$ of $(\widetilde M,g_1)$,
        $$\Lip(\widetilde{f}|_{X}, g_1,g_2)=\sup_{x\in X} \Lip_x(\widetilde{f}, g_1,g_2),$$
        where $\widetilde{f}|_{X}$ is the restriction of $\widetilde{f}$ to $X$.
        \item As a consequence of item (2), the local Lipschitz constant of $f$ satisfies
        $$\Lip(f, g_1,g_2)=\sup_{x\in M} \Lip_x(f, g_1,g_2),$$
        and so by item (1),
         $$\Lip(f, g_1,g_2)=\max_{x\in M} \Lip_x(f, g_1,g_2),$$
        \item As a consequence of items (1) and (3), the stretch locus $E_f(g_1, g_2)$ of $f\in \Lip_{\id}(M,g_1, g_2)$ is nonempty and closed.
\end{enumerate}
\end{remark}

Moreover, one defines the stretch locus between two metrics.

\begin{definition}
   Suppose $g_1,g_2 \in R^-(M)$. The stretch locus from $g_1$ to $g_2$ is given by
    \begin{align*}
    E(g_1,g_2):= \bigcap_{\substack{f\in \Lip_{\id}(M,g_1,g_2) \\ \Lip(f,g_1,g_2)=L(g_1,g_2)}} E_f(g_1, g_2).
   \end{align*}
   
    Furthermore, a Lipschitz map $f_0\in \Lip_{\id}(M,g_1,g_2)$ is called an \emph{optimal Lipschitz map} from $g_1$ to $g_2$ if its stretch locus $E_{f_0}(g_1, g_2)$ satisfies 
    $$E_{f_0}(g_1, g_2)=E(g_1,g_2).$$
\end{definition}

\begin{remark} \label{remark, StretchClosed}
    It follows from item (4) of Remark \ref{remark, LocalLip} that the stretch locus $E(g_1,g_2)$ is closed.
\end{remark}

In fact, an optimal Lipschitz map from $g_1$ to $g_2$ always exists when $g_1,g_2 \in R^-(M)$. The following proof is based on Lemma 4.13 of \cite{GK17} about existence of optimal Lipschitz maps in hyperbolic $n$ spaces $\mathbb{H}^n$. The crucial method used here is the barycenter method for probability measures. A good introduction of barycenter theory can be found in Section 4 of \cite{KT03}. 

\begin{proposition}
    Suppose $g_1,g_2 \in R^-(M)$. Then there exists an optimal Lipschitz map $f_0$ from $g_1$ to $g_2$. In particular, the stretch locus $E(g_1,g_2)$ is a nonempty subset of $M$.
\end{proposition} 
\begin{proof}
    Since $E(g_1,g_2)$ is closed, let us consider the open subset $M\setminus E(g_1,g_2)$.  For any $x\in M\setminus E(g_1,g_2)$, by the definition of the stretch locus $E(g_1,g_2)$, there exists a best Lipschitz map $f_x$ and some constant $\varepsilon_x>0$ such that 
    $$\Lip_x(f_x, g_1,g_2) < L(g_1,g_2)-\varepsilon_x<L(g_1,g_2). $$
    
  Since the local Lipschitz constant is upper semicontinuous (Remark \ref{remark, LocalLip}), there exists a small neighborhood $U_x$  so that for any $x'\in U_x$,
 $$\Lip_{x'}(f_x, g_1,g_2) < \Lip_{x}(f_x, g_1,g_2)+\frac{\varepsilon_x}{2}< L(g_1,g_2)-\frac{\varepsilon_x}{2}.$$
    
We take a small geodesic ball $B^{g_1}_r(x)$ inside $U_x$ that is convex with respect to $g_1$(see, for example, \cite[Chapter 3, Proposition 4.2]{DoC92}). By abuse of notation, we denote $U_x$ to be $B^{g_1}_r(x)$ . The family $\{U_x\}_{x\in M}$ forms a covering of $M \setminus E(g_1,g_2)$. Since $M\setminus E(g_1,g_2)$ as an open subset of $M$ is $\sigma$-compact, it is a union of countably many compact subsets. Therefore, there exist countably many open convex subsets $U_i=U_{x_i}$ (corresponding to best Lipschitz maps $f_{x_i}$) that covers $M\setminus E(g_1,g_2)$. 

Since the Lipschitz maps $f_{x_i}$ are homotopic to the identity Lemma \ref{lem, uniquelift} implies that they have unique Lifts with $\widetilde f_{x_i}  \gamma = \gamma\widetilde f_{x_i} $ for all $\gamma \in \Gamma$.
Fix $z_0 \in \widetilde M$. For any $x$ in $\mathcal{F}$, we define  the probability measure on $(\widetilde M,g_2)$ as 
$$p_x:= \sum_{k=1}^{\infty} \alpha_k \delta_{\widetilde f_{\scaleto{x_k}{4pt}}(x)},$$
where $\{ \alpha_k \}_{k\in\mathbb{N}}$ is a sequence of positive real numbers satisfying the following conditions:
\begin{itemize}
    \item $\sum\limits_{k\in\mathbb{N}} \alpha_k =1 $.
    \item $\{ \alpha_k \}_{k\in\mathbb{N}}$ decays fast enough so that
$$\int_{\widetilde M} d_{g_2}^2(y,\widetilde f_{x_1}( z_0)) dp_{z_0}(y)= \sum_{k=1}^{\infty} \alpha_k d^2_{g_2}(\widetilde f_{x_k}(z_0) ,\widetilde f_{x_1}(z_0)) < \infty.$$
\end{itemize}
Using the inequality $(a+b+c)^2\leq 3(a^2+b^2+c^2)$ and the fact that the maps $\widetilde f_{x_k}$ are $L:=L(g_1,g_2)$-Lipschitz we obtain
\begin{align*}
\int_{\widetilde M} & d_{g_2}^2(y,\widetilde f_{x_1}(x)) dp_{x}(y) = \sum_{k=1}^{\infty} \alpha_k d^2_{g_2}(\widetilde f_{x_k}(x) ,\widetilde f_{x_1}(x)) \leq 3\sum_{k=1}^{\infty} \alpha_k \big(d^2_{g_2}(\widetilde f_{x_k}(x) ,\widetilde f_{x_k}(z_0))\\
&+3\sum_{k=1}^{\infty} \alpha_k d^2_{g_2}(\widetilde f_{x_k}(z_0) ,\widetilde f_{x_1}(z_0))+3\sum_{k=1}^{\infty} \alpha_k d^2_{g_2}(\widetilde f_{x_1}(z_0) ,\widetilde f_{x_1}(x))\big) \\
&\le 6 L^2 d_{g_2}^2(x,z_0) + 3\sum_{k=1}^{\infty} \alpha_k d^2_{g_2}(\widetilde f_{x_k}(z_0) ,\widetilde f_{x_1}(z_0)) < \infty.
\end{align*}
Therefore, for all $z, x \in \widetilde M$ we have
\begin{align*}
\int_{\widetilde M} d_{g_2}^2(z,y) dp_x(y) &\le 2 \int_{\widetilde M} d_{g_2}^2(z, \widetilde f_{x_1}( z_0))) dp_x(y) + 2 \int_{\widetilde M} d_{g_2}^2( \widetilde f_{x_1}( z_0)),y) dp_x(y)\\
&\le 2d_{g_2}^2(z, \widetilde f_{x_1}( z_0))) +2 \int_{\widetilde M} d_{g_2}^2( \widetilde f_{x_1}( z_0)),y) dp_x(y) < \infty\\
\end{align*}
Hence
$
F_x: \widetilde{M} \rightarrow \mathbb{R}
$
with 
$$F_x(z) =\int_{\widetilde M} d_{g_2}^2(z,y) dp_x(y)$$ defines a strictly convex function which has a unique minimizer
$b(p_x)$, i.e. $F_x(b(p_x))=\min_{z} F_x(z)$. The element $b(p_x)$ is called the $d^2$-barycenter map of the probability measure $p_x$ (see e.g. Proposition 4.3 of \cite{KT03} for further details).

Since the Lipschitz maps $\widetilde f_{\scaleto{x_k}{4pt}}$ are $\Gamma$-equivariant, one obtains $\gamma_* p_x=p_{\gamma x}$ for any $\gamma\in\Gamma$.  Since the barycenter $b$ is $\Gamma$-equivariant  as well (\cite[Lemma 5.1]{KT03}) we obtain for any $\gamma \in \Gamma$
$$\gamma\widetilde f_0(x) =\gamma(b(p_x))=b(\gamma_* p_x)= b(p_{\gamma x})=\widetilde f_0(\gamma x).$$
Therefore, Lemma \ref{lem, uniquelift} implies that $\widetilde f_0 $ descends to a map $f_0 :M \to M$
homotopic to the identity.
Now we show that $\widetilde f_0 :\widetilde M \to \R$
with 
$$\widetilde f_0(x):= b(p_x)$$ 
defines an optimal map. 
For that we note that nonpositive curvature yields the barycenter contracting property (see \cite[Theorem 6.3]{KT03}),  i.e. for any pair $x, x'$ on $\widetilde M$ we have,
$$d_{g_2}(\widetilde f_0(x),\widetilde f_0(x')) \leq \sum_{k=1}^{\infty} \alpha_k d_{g_2}(\widetilde f_{\scaleto{x_k}{4pt}}(x), \widetilde f_{\scaleto{x_k}{4pt}}(x')) \le L d_{g_1}(x, x')$$
Hence $\widetilde f_0$ is the lift of the best Lipschitz map $f_0$. In particular, the above estimate holds for any lift $ \widetilde U_i$ of the convex ball $U_i \subset M\setminus E(g_1,g_2)$  and therefore implies that for any $ z \in \widetilde U_i$ by Remark \ref{remark, LocalLip},
\begin{align*}
\Lip_z(\widetilde f_0,g_1,g_2)& \leq \sup_{x'\neq x \in \widetilde U_i} \frac{d_{g_2}(\widetilde f_0 (x), \widetilde f_0 (x'))}{d_{g_1}(x,x')}   \\
&\leq \sup_{x'\neq x \in \widetilde U_i}  \sum_{k=1}^{\infty}  \alpha_k \frac{d_{g_2}(\widetilde f_{\scaleto{x_k}{4pt}}(x), \widetilde f_{\scaleto{x_k}{4pt}}(x'))}{d_{g_1}(x,x')}\\
&\leq \sum_{k\neq i}  \alpha_k L(g_1,g_2) + \alpha_i \Lip(f_{x_i}|_{\widetilde{U_i}}, g_1,g_2)\\
&< L(g_1,g_2)- \alpha_i \frac{\varepsilon_{x_i}}{2} < L(g_1,g_2).
\end{align*}
 Hence the stretch locus $E_{f_0}(g_1, g_2)$ of $f_0$ does not intersect $M \setminus E(g_1, g_2) $ which implies $E_{f_0}(g_1, g_2) \subset E(g_1, g_2) $ and $E(g_1, g_2)$. By
 the definition of $E(g_1, g_2) $ we have  $E(g_1, g_2) \subset E_{f_0}(g_1, g_2)$
we conclude $E(g_1,g_2) = E_{f_0}(g_1, g_2)$. In particular
item (4) of Remark \ref{remark, LocalLip} yields that $E(g_1,g_2)$
is closed and nonempty.

\end{proof}

\subsection{The Mather set and the stretch locus}
\hfill\\
When the least Lipschitz constant equals to the maximal stretch, we have the following characterization of the relation between the Mather set and the stretch locus. 
\begin{proposition} \label{proposition,MatherStretchLocus}

   Suppose for $g_1,g_2\in R^{-} (M)$, we have  $$S(g_1,g_2)= L(g_1,g_2).$$
    
    Then the projection of the Mather set $\mathcal{M}(g_1,g_2)$ on $M$ is contained in the stretch locus from $g_1$ to $g_2$, that is
     $$\pi(\mathcal{M}(g_1,g_2)) \subset E(g_1,g_2).$$
\end{proposition}
\begin{proof}
Suppose $S(g_1,g_2)= L(g_1,g_2)$. By Lemma \ref{lem, GeodesicStretchLipStretch}, for a maximal current $\mu_0 \in \mathcal{C}^{g_1} (\Gamma)$ and any $f\in \Lip_{\id}(M, g_1, g_2)$,
$$S(g_1,g_2)= I_{\mu_0}(g_1,g_2) \leq  \int_{S^{g_1} M}  \| Df(v) \|_{g_2} d\hat{m}^{g_1}_{\mu_0}.$$
On the other hand, there exists $f_0\in \Lip_{\id}(M, g_1, g_2)$ so that $ L(g_1,g_2) = \Lip(f_0,g_1,g_2)$. Therefore,  
$\norm{Df_0(v)}_{g_2} \le \Lip_{\pi v} (f_0,g_1,g_2) \le L(g_1,g_2)$ for all $v \in S^{g_1} M$ for which $Df_0(v)$ exists. Hence,
$$ S(g_1,g_2)=L(g_1,g_2) = \Lip(f_0,g_1,g_2)\geq \int_{S^{g_1}M} \norm{Df_0(v)}_{g_2} d\hat{m}^{g_1}_{\mu_{0}}.$$
Therefore,
    \begin{equation} \label{equation, maximallyStretchedGeodesics}
        S(g_1,g_2)=\int_{S^{g_1}M} \norm{Df_0(v)}_{g_2} d\hat{m}^{g_1}_{\mu_{0}}= \Lip(f_0,g_1,g_2) .
    \end{equation}
 Furthermore, 
for all  $(\xi,\eta)$ in the support of ${\mu_{0}}$, Theorem \ref{prop,GeodesicToGeodesic} implies that the lift $\widetilde f_0 \in \mathcal{C}$ of $f_0$ takes the $g_1$-geodesic $c_{\xi, \eta}^{g_1}$ to the corresponding $g_2$-geodesic $c_{\xi, \eta}^{g_2}$. Moreover, for $\hat{m}^{g_1}_{\mu_{0}}$ almost every $v$, $\norm{Df_0(v)}_{g_2}= \Lip (f_0,g_1,g_2) = L(g_1,g_2)$. Then Equation \eqref{equation Length-distance} implies, that every geodesic in the support of $\hat{m}^{g_1}_{\mu_{0}} $ of speed one is mapped to a geodesic of speed $L(g_1, g_2) = S(g_1,g_2)$. Therefore, 
$$\pi (\supp \hat{m}^{g_1}_{\mu_{0}}) \subset E_{f_0}.$$
Since this holds for any Lipschitz map that realizes $L(g_1,g_2)$, we obtain,
$$\pi (\supp \hat{m}^{g_1}_{\mu_{0}}) \subset E(g_1,g_2).$$
Since this also holds for any maximal current $\mu_0$ in $\mathcal{C}^{g_1} (\Gamma)$ and since $E(g_1,g_2)$ is a closed subset, 
we obtain $ \pi(\mathcal{M}(g_1,g_2)) \subset E(g_1,g_2)$.
\end{proof}

 \subsection{The stretch locus and geodesic laminations} 
 \label{subsection,NotLamination}
 \hfill\\
This subsection discusses more examples in $R^-(M)$ where the equality $S(g_1,g_2)= L(g_1,g_2)$ holds. It is based on the work of \cite{GK17} which
investigates the relation between stretch loci and \emph{geodesic laminations} in hyperbolic n-space $\mathbb{H}^n$. \emph{(One-dimensional) geodesic laminations} on hyperbolic surfaces were introduced by Thurston and play a central role in Thurston's study of Teichm\"uller spaces and hyperbolic 3-manifolds. We give here its definition for $n$-dimensional closed manifold $M$ with $n\geq 2$.   
\begin{definition}
    A \emph{(one-dimensional)\footnote{In a compact manifold $M$ of dimension $n\geq 3$, it seems that higher dimensional geodesic laminations, namely laminations whose leaves are totally geodesic and of dimension not less than 2, are more rigid and less interesting. For example, in \cite{Zeg91}, it is proved that all leaves of a codimension one lamination of $M$ are compact.} geodesic lamination} $\mathcal{L}$ of $(M,g)$ is a nonempty closed subset of $M$ that consists of a disjoint union of simple complete (closed or bi-infinite) geodesics with respect to $g$.
\end{definition}

The result in \cite{GK17} provides a sufficient condition for the stretch locus $E(g_1,g_2)$ to be a geodesic lamination in the setting of constant negatively--curved metrics, regardless of the dimension of the manifold $M$. The key idea in their proof (\cite[Lemma 5.2]{GK17}) is to show, for an optimal Lipschitz map $f_0: (M,g_1)\to (M,g_2)$, every point of the stretch locus $E_{f_0}(g_1,g_2)=E(g_1,g_2)$ lies on a complete simple geodesic that remains entirely in $E_{f_0}(g_1,g_2)$ and does not intersect other parts of the stretch locus. To establish this, they performed a local analysis around a point $x$ in the strech locus $E_{f_0}(g_1,g_2)$. By considering a small ball $B_x$ centered at $x$, they restrict $f_0$ to the boundary $\partial B_x$, and apply the Kirszbraum--Valentine theorem (\cite[Section 3.1]{GK17}) to obtain a Lipschitz extension $\overline{f_0}: \overline{B_x} \to M$ that minimizes the Lipschitz constant. Optimality then ensures $\overline{f_0}$ still has Lipschitz constant equal to $L(g_1,g_2)$. Next, by a clever argument based on triangle comparison (i.e. Toponogov's theorem), they show that there exist only two diametrically opposite points on $\partial B_x$ so that the geodesic segment connecting them is maximally stretched by $\overline{f_0}$, with stretch factor equal to the Lipschitz constant $L(g_1,g_2)$. By iterating this process and extending this geodesic segment, they construct a complete simple geodesic through $x$ that lies entirely in the stretch locus.

Their argument does not rely on hyperbolic geometry involving high symmetry, in contrast to the original work of Thurston \cite{Th98}. Their proofs are based on classical tools in geometry of negatively curved manifolds, specifically the Kirszbraum--Valentine theorem and the Toponogov's theorem. Consequently, it is not hard to convince oneself that their argument can carry over to variable negative curvature under certain curvature assumptions, as will be described in the following theorem.

For a smooth Riemannian metric $g$ on $M$, we denote
$$K^-_g(p) := \min\limits_{\substack{\Pi \subset T_{p} M}} K_{g}(\Pi), \qquad  K^+_g(p) := \max\limits_{ \Pi \subset T_{p} M} K_{g}(\Pi),$$
and
$$K^-_g := \min\limits_{p\in M} K^-_{g}(p), \qquad  K^+_g := \max\limits_{p\in M} K^+_{g}(p),$$
where $\Pi$ are $2$-dimensional planes in $T_{p}M$ and $K_{g}(\Pi)$ are sectional curvatures for $\Pi$ with respect to $g$.

\begin{theorem}[see \cite{GK17} Theorem 5.1, Lemma 5.2] \label{GK, geodesicLamination}
Suppose $g_1,g_2 \in R^-(M)$ satisfy 
   \begin{equation}\label{eqtn,LaminationConditions}
        0 < \frac{K^-_{g_1}}{K^+_{g_2}}<L(g_1,g_2)^2, 
        \end{equation}  
    Then 
    $$S(g_1,g_2) = L(g_1,g_2).$$
    Moreover, $E(g_1,g_2)$ is a geodesic lamination and each leaf of $E(g_1,g_2)$ is maximally stretched by an optimal Lipschitz map $f_0$, in the sense that the lift $\widetilde{f_0}\in\mathcal{C}$ of $f_0$ multiplies all distances of lifts of leaves of $E(g_1,g_2)$ on $(\widetilde{M},g_1)$ by $L(g_1,g_2)$. 
\end{theorem}

\begin{proof}
    Consider $g'_1=\lambda g_1$ for some positive constant $\lambda$. The following conditions are sufficient to apply Theorem 5.1 and then Lemma 4.6 of \cite{GK17} to obtain the desired results in the above statement for Riemannian metrics $g'_1$ and $g_2$:
    $$L(g'_1, g_2) > 1,$$
    $$0> K^-_{g_1} \geq  K^+_{g_2}.$$

    These are equivalent to 
    $$L(g_1, g_2) > \lambda^{\frac{1}{2}} \geq \bigg(\frac{K^-_{g_1}}{K^+_{g_2}}\bigg)^{\frac{1}{2}} >0. $$
    This yields the conditions in the statement.
\end{proof}
Their theorem can be stated in a stronger form as follows.
\begin{corollary}[\cite{GK17}, Theorem 5.1, Lemma 5.2]  \label{Cor, geodesicLamination}
Suppose $g_1,g_2 \in R^-(M)$ and $f_0$ is an optimal Lipschitz map from $(M,g_1)$ to $(M,g_2) $. Suppose for all points $p\in M$,

    $$0> K^-_{g_1}(p)> K^+_{g_2}(f_0(p)) \cdot L(g_1,g_2)^2. $$
    Then 
    $$S(g_1,g_2) = L(g_1,g_2).$$
    Moreover, $E(g_1,g_2)$ is a geodesic lamination that is maximally stretched by some optimal Lipschitz map as described in Theorem \ref{GK, geodesicLamination}.
\end{corollary}

Theorem \ref{GK, geodesicLamination} permits some deformation in $R^-(M)$ because of the following proposition. For this porposition, we can equip $R^-(M)$ with the $C^{\infty}$ topology (in fact, any $C^k$ topology for $k\geq 2$ suffices). 

\begin{proposition} \label{prop,LaminationConditions}
    The set given by
 \begin{align*}
       \mathcal{R}:=\bigg\{ (g_1,g_2)\in R^-(M) \times R^-(M) \text{ }|\text{ } 0 < \frac{K^-_{g_1}}{K^+_{g_2}}<L(g_1,g_2)^2 \bigg\}
 \end{align*}
        is open in $R^-(M) \times R^-(M)$. Therefore, there exists an open set $U$ in $R^-(M) \times R^-(M)$ such that for $(g_1,g_2)\in U$, the stretch locus $E(g_1,g_2)$ is a geodesic lamination that is maximally stretched by some optimal Lipschitz map as described in Theorem \ref{GK, geodesicLamination} and $S(g_1,g_2)=L(g_1,g_2)$.  
\end{proposition} 
\begin{proof}
    Suppose $(g_1,g_2)\in \mathcal{R}$. We want to show if $(g_1',g_2')$ is sufficiently close to $(g_1,g_2)$ in $R^-(M) \times R^-(M)$, then $(g_1',g_2')$ is also in $\mathcal{R}$. The fact that $(g_1,g_2)\in \mathcal{R}$ implies we can find small $\varepsilon>0$ so that, 
    
   $$L(g_1,g_2)^2- \frac{K^-_{g_1}}{K^+_{g_2}} > \varepsilon>0  . $$

     Sectional curvatures are second variation of metric tensors. For $i=1,2$, when a smooth Riemannian metric  $g_i'$ in $R^-(M)$ is in a sufficiently small open neighborhood of $g_i$ (with respect to the topology from $C^k$-norms for $k\geq 2$), we have
     $$\bigg|\frac{K^-_{g_1}}{K^+_{g_2}}  - \frac{K^-_{g'_1}}{K^+_{g'_2}} \bigg|\leq \frac{\varepsilon}{2}.$$
    By Theorem \ref{GK, geodesicLamination}, we know $S(g_1,g_2)=L(g_1,g_2)$ for $(g_1,g_2)\in \mathcal{R}$. And by Proposition \ref{propositionGK, GeodesicStretchContinuous}, the maximal stretch $S(\cdot,\cdot)$ is continuous. So after adjusting $g_i'$ to be even closer to $g_i$, we can make sure
    $$\bigg|S(g'_1,g'_2)^2-S(g_1,g_2)^2\bigg| \leq \frac{\varepsilon}{2}.$$
    Therefore, we obtain by Proposition \ref{prop,LgeqS}
\begin{align*}
L(g'_1,g'_2)^2&\geq S(g'_1,g'_2)^2 \geq S(g_1,g_2)^2- \frac{\varepsilon}{2}\\
&=L(g_1,g_2)^2-\frac{\varepsilon}{2}\\
&> \frac{K^-_{g_1}}{K^+_{g_2}}+ \frac{\varepsilon}{2} \geq \frac{K^-_{g'_1}}{K^+_{g'_2}}.
\end{align*}

So $(g'_1,g'_2)$ also satisfies  inequality \eqref{eqtn,LaminationConditions} and $(g'_1,g'_2)\in \mathcal{R}$. As a consequence, $E(g'_1,g'_2)$ is a geodesic lamination and $L(g'_1,g'_2)=S(g'_1,g'_2)$. We obtain the desired result.
\end{proof}

 When the dimension of $M$ is $n=2$, the above discussion provides abundant examples of stretch loci being geodesic laminations.

\begin{remark} [Laminations may survive in $R^-(S)$] \label{remark,LaminationSurvive} 
    Let $M=S$ be a closed surface of genus $\mathcal{G}\geq 2$. Take $g_1$ and $g_2$ to be hyperbolic metrics representing different points $[g_1]$ and $[g_2]$ in the Teichm\"uller space $\mathcal{T}(S)$. From \cite{Th98}, we know  
    $$S(g_1,g_2)^2=L(g_1,g_2)^2 >1 = \frac{K^-_{g_1}}{K^+_{g_2}}.$$
    Lift the set  
    \begin{align*}
    \mathcal{T}^{(2)}(S)&:=\mathcal{T}(S) \times \mathcal{T}(S)\setminus \diag 
    \end{align*}
    to a subset of $R^-(S)\times R^-(S)$ and denote it as  $\widetilde{ \mathcal{T}^{(2)}(S)}$. The proposition \ref{prop,LaminationConditions} then implies that there exists a nonempty open neighborhood $U$ of $\widetilde{ \mathcal{T}^{(2)}(S)}$ in $R^-(S)\times R^-(S)$ so that for $(g_1,g_2) \in U$, the stretch locus $E(g_1,g_2)$ is still a geodesic lamination and $L(g_1,g_2)=S(g_1,g_2)$. 
\end{remark}

\subsection{Example of a stretch locus which is not a geodesic lamination}
\hfill\\
Let $M=S$ be a closed surface of genus $\mathcal{G} \geq 2$. In this subsection, we provide an example of metrics $g_1,g_2 \in R^-(S)$ for which the stretch locus $E(g_1,g_2)$ is not a geodesic lamination. 

 In the Teichm\"uller space $\mathcal{T}(S)$, the equality $L(g_1,g_2)=S(g_1,g_2)$ always holds and the geodesic flow $\phi^{g_1}$ on the Mather set $\mathcal{M}(g_1,g_2)$ always has zero topological entropy. In contrast, as a further observation from this example, we show that the topological entropy of the geodesic flow $\phi^{g_1}$ on the Mather set for general $g_1,g_2\in R^-(S)$ can be positive.

\begin{example}[positive entropy example] 
\label{MatherNotLamination} Suppose $M=S$ is a genus $\mathcal{G} =2$ closed surface and $g_1$ is a hyperbolic metric on $M$. Let $M_1\subset (M,g_1)$ be a subsurface bounded by a separating simple closed geodesic $\gamma^{g_1}_0$ (see Figure \ref{fig: figure2}).  Let $a>1$ be a constant. Consider a metric $g_2=\varphi g_1$ which is conformal to $g_1$, where $\varphi: M \to (0,\infty)$ is a smooth function that satisfies the following conditions:
\begin{enumerate}
    \item $\varphi(x)=a \hspace{4cm} \text{  for } x\in M_1$.
     \item 
     $0< \varphi(x)< a \hspace{3.25cm} \text{ for } x\in M \backslash M_1$.
     \item 
     $K_{g_2}(x)= \frac{1}{\varphi(x)}(-1-\Delta_{g_1} \log \varphi )<0 \hspace{0.6cm} \text{ for } x\in M$,
\end{enumerate}
  where $\Delta_{g_1} f = \textnormal{div} \circ \grad f$ is the negative Laplacian Beltrami operator with respect to $g_1$.
Then the stretch locus $E(g_1,g_2)$ (and the projection of the Mather set $\pi (\mathcal{M}(g_1,g_2))$) are not geodesic laminations. Moreover,
\begin{itemize}
    \item the topological entropy of $\phi^{g_1}$ on the Mather set $\mathcal{M}(g_1,g_2)$ satisfies
$$h_{top}(\phi^{g_1},\mathcal{M}(g_1,g_2))>0.$$ 
    \item  The Hausdorff dimension of $\mathcal{M}(g_1,g_2)$ is
    $$
     Hd(\mathcal{M}(g_1,g_2))  =2h_{top}(\phi^{g_1},\mathcal{M}(g_1,g_2))+1>1.$$
\end{itemize} 
 \end{example}

\begin{figure}[!htb]
   \centering    \includegraphics[width=\textwidth]{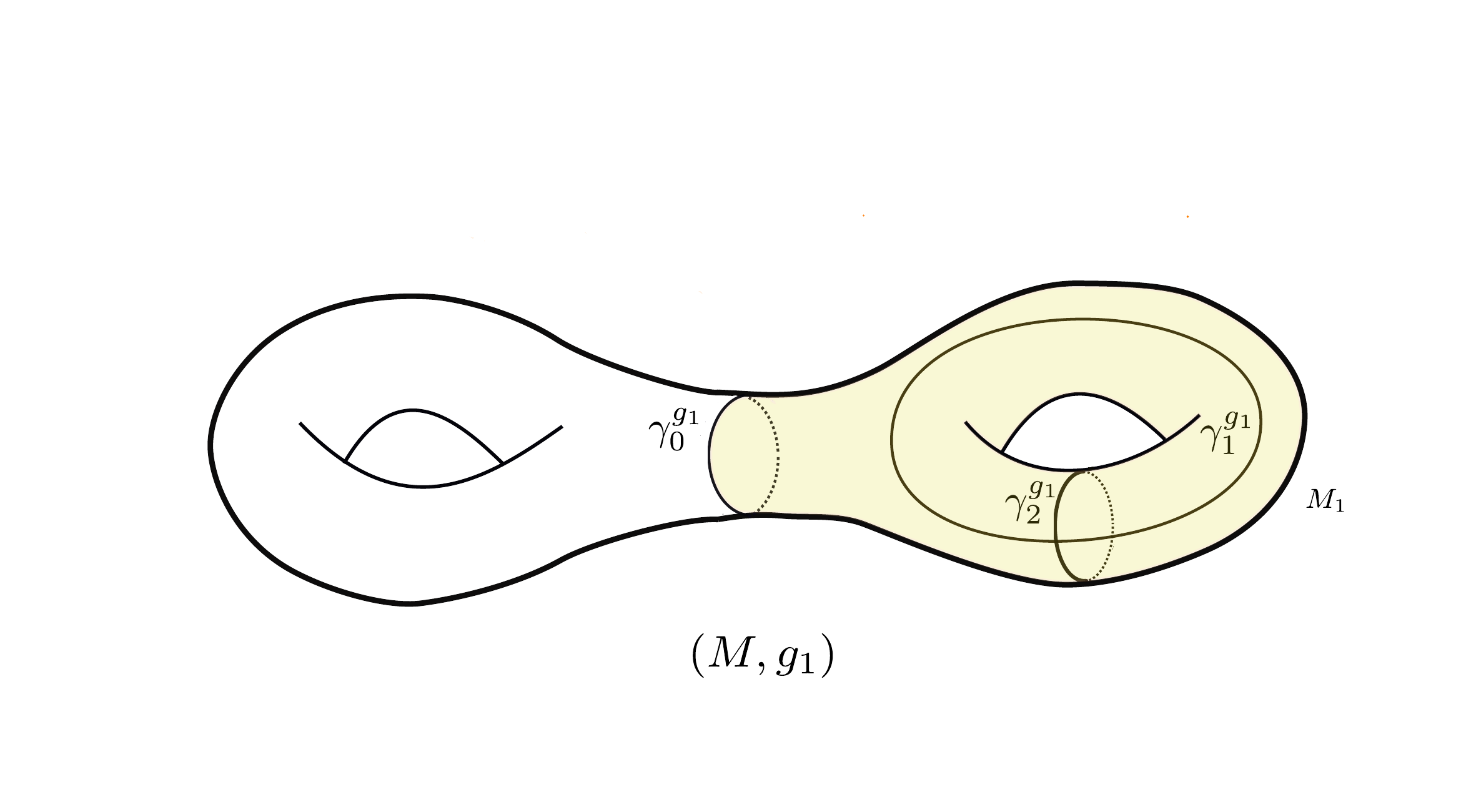}
    \caption{$(M,g_1)$ is a closed hyperbolic surface of genus 2. The yellow part $M_1$ is a submanifold of $(M,g_1)$ bounded by a separating simple closed $g_1$-geodesic $\gamma^{g_1}_0$. The generators of the fundamental group $\pi_1(M_1)$ is represented by two simple closed $g_1$-geodesics $\gamma^{g_1}_1$ and $\gamma^{g_1}_2$ in the figure.} 
    \label{fig: figure2}
\end{figure}

\begin{proof}

Suppose $\gamma:[0,1]\to M_1$ is a closed geodesic with respect to $g_1$. Then it is also a closed geodesic with respect to $g_2$ and 
$$\ell_{g_2}([\gamma])=\int_0^1 g_2(\dot{\gamma}(t),\dot{\gamma}(t))^{\frac{1}{2}}dt= \sqrt{a}\int_0^1 g_1(\dot{\gamma}(t),\dot{\gamma}(t))^{\frac{1}{2}}dt=\sqrt{a}\ell_{g_1}([\gamma]). $$

On the other hand, if $\gamma:[0,1]\to M$ is a closed geodesic with respect to $g_1$ not entirely lying in $M_1$, then we know
\begin{align*} 
\ell_{g_2}([\gamma])&\leq L_{g_2}(\gamma)=\int_0^1 \varphi^{\frac{1}{2}}(\gamma(t))g_1(\dot{\gamma}(t),\dot{\gamma}(t))^{\frac{1}{2}}dt\\
&<\sqrt{a}\int_0^1 g_1(\dot{\gamma}(t),\dot{\gamma}(t))^{\frac{1}{2}}dt=\sqrt{a}L_{g_1}(\gamma)= \sqrt{a}\ell_{g_1}([\gamma]). 
\end{align*}

Together, we obtain 
$$S(g_1,g_2)=\sup_{[\gamma]\in[\Gamma]} \frac{\ell_{g_2}([\gamma])}{\ell_{g_1}([\gamma])}=\sqrt{a}.$$

We also know that the identity map satisfies
$$\text{Lip}(\id, g_1, g_2)=\max_{v\in S^{g_1}M}\frac{||v||_{g_2}}{||v||_{g_1}}=\max_{v\in S^{g_1}M} \varphi(\pi(v) )=\sqrt{a}.$$
As a consequence,  $E_{\id}(g_1,g_2)= M_1$ and also
$$S(g_1,g_2)=\text{Lip}(\id,g_1,g_2)= L(g_1,g_2)$$ and the identity map is a best Lipschitz map.  Moreover, from Proposition \ref{proposition,MatherStretchLocus}, we conclude
$$\pi(\mathcal{M}(g_1,g_2)) \subset E(g_1,g_2)\subset E_{\id}(g_1,g_2)= M_1.$$

From previous arguments, we know that $\pi(\mathcal{M}(g_1,g_2))$ contains all closed $g_1$-geodesics on $M_1$. We conclude that the sets $\pi(\mathcal{M}(g_1,g_2))$ and $E(g_1,g_2)$ are not geodesic laminations. \footnote{This example does not satisfy the condition in Theorem \ref{GK, geodesicLamination} since

$$\frac{K^-_{g_1}}{K^+_{g_2}}\geq \max_{x\in M} \frac{-\varphi(x)}{-1- \Delta_{g_1} \log \varphi } \geq a= L(g_1,g_2)^2. $$}

To obtain a better understanding of the Mather set $\mathcal{M}(g_1,g_2)$, let us further denote $\Gamma_1$ as the subgroup in $\mathrm{PSL}(2,\mathbb{R})$ freely generated by two elements represented by simple closed geodesics $\gamma^{g_1}_1$ and $\gamma^{g_1}_2$ on $M_1$ (See Figure \ref{fig: figure2}).  Consider the complete convex cocompact hyperbolic surface $\mathbb{H}^2/\Gamma_1$ given by a one-holed torus glued with an expanding funnel of infinite volume (see, for example, \cite[page 500]{Eb72II}). This noncompact surface contains $M_1$ as its convex core. Since the Mather set $\mathcal{M}(g_1,g_2) \subset S^{g_1}M_1,$ it follows easily from the definition of the Mather set that $\mathcal{M}(g_1,g_2) \subset \Omega$, where $\Omega$ is the nonwandering set of the geodesic flow on the unit tangent bundle of $\mathbb{H}^2/\Gamma_1$ which is also contained in $S^{g_1}M$. A simple hyperbolic geometry computation shows that if a geodesic leaves $M_1$ in forward time direction in $\mathbb{H}^2/\Gamma_1$, then it will never go back to $M_1$.  On the other hand, because the Mather set $\mathcal{M}(g_1,g_2)$ contains all periodic vectors on $S^{g_1}M$ and because periodic vectors of $S^{g_1}(\mathbb{H}^2/\Gamma_1)$ are dense in $\Omega $ (\cite[Theorem 3.10]{Eb72II}). We conclude from closedness condition of $\mathcal{M}(g_1,g_2)$ that $\Omega \subset \mathcal{M}(g_1,g_2)$. Therefore 
$\mathcal{M}(g_1,g_2)=\Omega$
and the topological entropy $h_{top}(\phi^{g_1},\mathcal{M}(g_1,g_2))$ is equal to the topological entropy of the geodesic flow on $\mathbb{H}^2/\Gamma_1$ which is strictly positive. It also equals the critical exponent $\delta(\Gamma_1)$ of $\Gamma_1$(\cite{Su79}).

Since from Corollary 3.8 of \cite{Eb72II}, the nonwandering set $\Omega$ is identified as 
$$\Omega \simeq \Lambda^{(2)}(\Gamma_1)\times \mathbb{R},$$
where $\Lambda(\Gamma_1)$ denotes the limit set of $\Gamma_1$ and $$\Lambda^{(2)}(\Gamma_1)=\Lambda(\Gamma_1) \times \Lambda(\Gamma_1)\setminus \diag.$$ From Patterson Sullivan theory, the Hausdorff dimension of $\Lambda(\Gamma_1)$ equals to the critical exponent $\delta(\Gamma_1)$ of $\Gamma_1$. As a consequence of the above identification, the Hausdorff dimension of $\mathcal{M}(g_1,g_2)=\Omega$ is 
$$Hd(\mathcal{M}(g_1,g_2))=2\delta(\Gamma_1)+1.$$
\end{proof}

We thank Islam Mitul for discussing properties of convex cocompact manifolds with us. We also thank Sami Douba for pointing out to us the following fact.
\begin{remark}
     Compared  with Corollary \ref{cor,RigidityTopEntropy} and Theorem \ref{thm, EmtpyInterior}, the above example shows that even when the marked length spectra of $g_1$ and $g_2$ are not proportional, the topological entropy of the Mather set $h_{top}(\phi^{g_1},\mathcal{M}(g_1,g_2))$ can be arbitrarily close to $h_{top}(\phi^{g_1})=1$ and the Hausdorff dimension \newline $ Hd(\mathcal{M}(g_1,g_2))$ can be arbitrarily close to the Hausdorff dimension of the full unit tangent bundle by taking hyperbolic metrics $g_1$.This is achieved by choosing hyperbolic metrics $g_1$
 for which the lengths of the separating simple closed $g_1$-geodesics $\gamma_0^{g_1}$
 tend to zero; equivalently, by pinching the hyperbolic surface toward a cusped case.
\end{remark}

We notice that the conditions in Theorem  \ref{GK, geodesicLamination} and arguments in Lemma 5.2 of \cite{GK17} based on triangle comparison theorems are not necessary conditions for $E(g_1,g_2)$ to be a maximally stretched geodesic lamination. We have further discussion of related phenomena in Appendix \ref{appendix, C}. 

 \section{Maximal stretches, Lipschitz maps, and volumes}\label{section,StretchLipVol}

 All results in this section about length spectra, Lipschitz maps and volume are not new. We do not claim any novelty for them. Since these results from previous works of different authors were not presented from the perspective of maximal stretches and least Lipschitz constants,  we include this section to give a nice picture of how these objects, maximal stretches, least Lipschitz constants and volumes, are related to each other. Important related works include but not restricted to \cite{BCG94}, \cite{Kn95}, \cite{CD04} and \cite{GL19}.
\subsection{Lipschitz maps and volumes}
\hfill\\
We discuss relation between least Lipschitz constants and volumes of Riemannian manifolds.
\begin{proposition} 
    Suppose $\dim M=n \geq 2$ and suppose $g_1, g_2\in R^-(M)$. Then 
    $$\frac{\textnormal{vol}(M,g_2)}{\textnormal{vol}(M,g_1)} \leq L(g_1,g_2)^n.$$
   Moreover, equality holds if and only if $g_1$ is isometric to $g_2$ up to a multiplicative constant.
\end{proposition}
\begin{proof}
 
 Suppose $L=L(g_1,g_2)$.  Consider the scaling of the metric $g_2$ given by $g_2'= \frac{1}{L^2}g_2$. Then $\textnormal{vol}(M,g'_2)= \frac{1}{L^n}\textnormal{vol}(M,g_2)$ and $L(g_1,g'_2)=1$. A best Lipschitz map $f \in \Lip_{\id}(M,g_1,g_2)$ is a $1$-Lipschitz map between $(M,g_1)$ and $(M,g_2')$. To prove the statement, it is equivalent to show
 $$\textnormal{vol}(M,g_2')\leq\textnormal{vol}(M,g_1),$$
 with equality if and only if $g_2'$ is isometric to $g_1$. Since $f$ is homotopic to identity, it is of degree one. The result then follows from \cite[Appendix C,  Proposition C.1, Lemme C.2]{BCG94}.

\end{proof}

  \subsection{Maximal stretch and volume}
\hfill\\
We observe the following relation between maximal stretches, minimal stretches, and volumes when $M$ is a closed surface. The proof can be derived from \cite[Theorem 1.1]{CD04} and \cite[Theorem 5.1]{BCLS18}.

 We denote by $s(g_1,g_2):=\min\limits_{\mu\in \mathcal{C}(\Gamma)} I_{\mu}(g_1,g_2)$ the \emph{minimal stretch}. 
\begin{proposition} \label{prop, bonahonIntersection}
    Suppose $\dim M=2$ and suppose $g_1, g_2\in R^-(M)$. Then 
    $$s(g_1,g_2)^2 \leq \frac{\textnormal{vol}(M,g_2)}{\textnormal{vol}(M,g_1)} \leq S(g_1,g_2)^2.$$
   Moreover, either equality holds if and only if $g_1$ is isometric to $g_2$  up to a multiplicative constant.

\end{proposition}
\begin{proof} 
Recall from Subsection \ref{subsection GeodesicStretchIntersection}, we mentioned that Bonahon's intersection number $i:  \mathcal{C}^{sym} (\Gamma) \times  \mathcal{C}^{sym} (\Gamma) \to \mathbb{R}$ is continuous and symmetric, where $\mathcal{C} (\Gamma)$ is the space of $\Gamma$-invariant unoriented geodesic currents. We denote by $\lambda_g$ the associated Liouville current for a metric $g$.  For any $[\gamma] \in [\Gamma]$ and the associated Dirac current $\delta_{[\gamma]}$ and any constant $c>0$, we obtain the corresponding flip invariant current $\frac{1}{2}(\delta_{[\gamma]}+\delta_{[\gamma^{-1}]})$. Since $s(g_1,g_2)$ is the minimal stretch, it follows from Remark \ref{remark,intersection}, 
 \begin{align} \label{equation, geqMinimalStretch}
     i(\frac{c}{2}(\delta_{[\gamma]}+\delta_{[\gamma^{-1}]}), \lambda_{g_2})&=c\ell_{g_2}([\gamma])\\
   &\geq  s(g_1,g_2) c \ell_{g_1}([\gamma])\\
   &= s(g_1,g_2) i(\frac{c}{2}(\delta_{[\gamma]}+\delta_{[\gamma^{-1}]}), \lambda_{g_1}).
    \end{align}
  We can take properly scaled sequences of Dirac currents $\delta_{[\gamma_n]}$ approximating $\lambda_{g_2}$ (resp. $\lambda_{g_1}$). Then $\frac{1}{2}(\delta_{[\gamma_n]}+\delta_{[\gamma_n^{-1}]})$ also approximating $\lambda_{g_2}$ (resp. $\lambda_{g_1}$).Since the intersection number is continuous, from Equation \eqref{equation, geqMinimalStretch} yields,
 \begin{equation} \label{equation, geqMinimalStretchb}
 i( \lambda_{g_2}, \lambda_{g_2})\geq  s(g_1,g_2) i( \lambda_{g_2}, \lambda_{g_1}) \; \;
  \end{equation} 
  and 
   \begin{equation} \label{equation, geqMaxStretchb}
 \; \; i( \lambda_{g_1}, \lambda_{g_2})\geq  s(g_1,g_2) i( \lambda_{g_1}, \lambda_{g_1}).
 \end{equation} 
Since the intersection number is symmetric, we know $i( \lambda_{g_1}, \lambda_{g_2})= i( \lambda_{g_2}, \lambda_{g_1})$. 
Together, we obtain
$$i( \lambda_{g_2}, \lambda_{g_2})\geq  s(g_1,g_2)^2 i( \lambda_{g_1}, \lambda_{g_1}).$$
The intersection number of Liouville currents satisfies from Definition \ref{def:intersection},

$$i( \lambda_g, \lambda_g) = m_{\lambda_g}^g(S^gM)=2\pi \cdot \textnormal{vol}(M,g).
$$ 

We obtain

$$s(g_1,g_2)^2 \leq \frac{\textnormal{vol}(M,g_2)}{\textnormal{vol}(M,g_1)}.
$$
Using the inequality $\ell_{g_2}([\gamma])\leq S(g_1,g_2) \ell_{g_1}([\gamma])$
together with the arguments above yields the other inequality in the proposition.

Now assume that $$s(g_1,g_2)^2 = \frac{\textnormal{vol}(M,g_2)}{\textnormal{vol}(M,g_1)}.$$

This implies $i( \lambda_{g_2}, \lambda_{g_2})= s(g_1,g_2)^2 i( \lambda_{g_1}, \lambda_{g_1})$ and by Equations
\eqref{equation, geqMinimalStretchb} and \eqref{equation, geqMaxStretchb}, we obtain for $j \in \{1,2\}$
  $$
 i( \lambda_{g_j}, \lambda_{g_2})=  s(g_1,g_2) i( \lambda_{g_j}, \lambda_{g_1}).
 $$

From the transformation formula in Proposition \ref{prop: transformation-formula}, we obtain 
$$
i( \lambda_{g_2}, \lambda_{g_1})=  m_{\lambda_{g_1}}^{g_2}(S^{g_2}M) = \int_{S^{g_1}M}   a_{g_1,g_2} dm_{\lambda_{g_1}}^{g_1} =s(g_1,g_2) i( \lambda_{g_1},\lambda_{g_1})
$$
Therefore, 
\begin{equation} \label{equation, IntegralMinimalStretch}
\int_{S^{g_1}M}  ( a_{g_1,g_2}-s(g_1,g_2)) dm_{\lambda_{g_1}}^{g_1} = 0.
\end{equation}
On the other hand, for any $\delta_{\gamma}$, we have
$$\int_{S^{g_1}M}  ( a_{g_1,g_2}-s(g_1,g_2)) dm_{\delta_{[\gamma]}}^{g_1}= \ell_2([\gamma])-s(g_1,g_2)\ell_2([\gamma]) \geq 0.$$
Therefore, \cite[Theorem 1]{LT05} implies  the existence of a H\"older continuous function $u:  S^{g_1}M \to \R$ differentiable along flow lines of the geodesic flow of $g_1$ such that
 $ a_{g_1,g_2}- s(g_1,g_2) \ge X_{g_1} u $.  Combining with Equation \eqref{equation, IntegralMinimalStretch}, we obtain that $ a_{g_1,g_2}- s(g_1,g_2)$ is cohomologous to zero. Therefore $\ell_2([\gamma])= s(g_1,g_2) \ell_2([\gamma])$ for all $[\gamma]$. By the marked length rigidity theorem for surfaces (\cite{Cr90}, \cite{Ot90}), one concludes that $g_2$ is isometric to $g_1$ up to a multiplicative constant.\\
An analogous reasoning yields the other equality in the above proposition. 
\end{proof}
There is a version of the above proposition above for arbitrary dimensions provided the metrics are conformally equivalent. More precisely,
\begin{proposition}
Suppose  $\dim M=n \geq 2$ and  $g_1, g_2\in R^-(M)$ are conformally equivalent. Then
 \begin{equation}\label{eqtn,VolStretch}
     s(g_1,g_2)^n \leq \frac{\textnormal{vol}(M,g_2)}{\textnormal{vol}(M,g_1)} \leq S(g_1,g_2)^n.
     \end{equation}
Moreover, either equality holds if and only if $g_1$ is isometric to $g_2$ up to a multiplicative constant.
\end{proposition}
\begin{proof}
Assume that $g_1, g_2\in R^-(M)$ are conformally equivalent, i.e. there exists a smooth positive function $\varphi: M \to \R$ such that $g_2= \varphi g_1$. Let $\lambda_{g_1}$ be the Liouville current of $g_1$ and let $\hat m_{\lambda_{g_1}}^{g_1} $ be the corresponding probability measure. We follow the arguments in the proof of Theorem 4.1 in \cite{Kn95} using Jensen's inequality.

\begin{eqnarray*}
s(g_1,g_2)^2 \le I_{\lambda_{g_1}} (g_1, g_2)^2  &\le\int_{S^{g_1}M} g_2(v, v) d\hat m_{\lambda_{g_1}}^{g_1}  =\frac{1}{\vol( M,g_1)} \int_{M} \varphi d\vol_{g_1} \\
 & \le\left(\frac{\int_{M} \varphi^{\frac{n}{2}} d\vol_{g_1}}
{\vol( M, g_1)}\right)^{\frac{2}{n}} 
 =  \left(\frac{\vol(M,g_2)}{\vol(M,g_1)}\right)^{\frac{2}{n}}.
\end{eqnarray*}
This yields the first inequality. To obtain the second inequality in the proposition, we take the Liouville current $ \lambda_{g_2}$ of $g_2$.
Since
$$
S(g_1,g_2) \ge I_{\lambda_{g_2}}(g_1,g_2) = \frac{1}{I_{\lambda_{g_2}}(g_2,g_1)}
$$
and
$$
I_{\lambda_{g_2}}(g_2,g_1) \le \left(\frac{\vol(M,g_1)}{\vol(M,g_2)}\right)^{\frac{1}{n}}.
$$
we obtain 
$$
S(g_1,g_2)^n \ge \frac{\vol(M,g_2)}{\vol(M,g_1)}.
$$
If an equality holds in one of the two estimates in Equation \eqref{eqtn,VolStretch} of the proposition,  Jensen's inequality 
must be an equality as well and therefore the conformal factor $\varphi$ is constant.
\end{proof}
Based on the work of Guillarmou and Lefeuvre \cite{GL19},
the estimates in Equation \eqref{eqtn,VolStretch} hold for an arbitrary pair of negatively curved metrics which are sufficiently close in a suitable $C^N$ norm in $R^-(M)$.
More precisely:

\begin{proposition}
Let $g_1$ be a Riemannian metric in $R^-(M)$
let and $N > \frac{n}{2} +2$.
Then there exists an $\varepsilon >0$ such that for all $g_2 \in R^-(M)$
with $\| g_2 -g_1 \|_{C^N} <\varepsilon $  we have 
 \begin{equation*}\label{eqtn,VolStretch-nonconformal}
     s(g_1,g_2)^n \leq \frac{\textnormal{vol}(M,g_2)}{\textnormal{vol}(M,g_1)} \leq S(g_1,g_2)^n.
     \end{equation*}
     Furthermore, either equality holds if and only if $g_1$  and $g_2$ are isometric up to a multiplicative constant.
\end{proposition}  
 \begin{proof}
     
      The proof is a consequence of Theorem 2 of \cite{GL19} which says: there exists an $\varepsilon_1 >0$ such that for all metrics $g_2 \in R^-(M)$ with $\| g_2 -g_1 \|_{C^N} <\varepsilon_1$, the inequalities $\ell_{g_1}([\gamma]) \le \ell_{g_2}([\gamma])$ for all $[\gamma] \in [\Gamma]$ 
    implies $\vol(M,g_1) \le \vol(M,g_2) $. Furthermore, $\vol(M,g_1) =\vol(M,g_2) $ holds if and only if $g_1$ and $g_2$ are isometric. \\
Now choose $\varepsilon >0 $ such that for $g_2 \in R^-(M)$
satisfying $\| g_2 -g_1 \|_{C^N} <\varepsilon$, we have $\| r g_2 -g_1 \|_{C^N} <\varepsilon_1$ for both $r= r_S = S(g_1,g_2)^{-2} $ and  $r=r_s=s(g_1,g_2)^{-2} $. 
Then we obtain by Theorem 2 of \cite{GL19} , for $g^S_2 = r_S g_2$,
$$
   \ell_{g^S_2}([\gamma])= \frac{1}{ S(g_1,g_2)} \ell_{g_2}([\gamma])\leq \ell_{g_1}([\gamma]).
   $$ 
This implies $\vol(M,g^S_2)\leq \vol(M,g_1)$ which is equivalent to
$\frac{\textnormal{vol}(M,g_2)}{\textnormal{vol}(M,g_1)} \leq S(g_1,g_2)^n$. 

Moreover, $\frac{\textnormal{vol}(M,g_2)}{\textnormal{vol}(M,g_1)} = S(g_1,g_2)^n$ implies  $\vol(M,g^S_2) = \vol(M,g_1)$ and therefore  $g^S_2$ is isometric to $g_1$. The other part follows from a similar argument.
 \end{proof}

\appendix

\section{An example of $S(g_1,g_2)< L(g_1,g_2)$} \label{ExampleS<L}
\hfill\\
In this appendix, we discuss an example of $g_1,g_2\in R^-(M)$ with $S(g_1,g_2)< L(g_1,g_2)$ based on the work of \cite{GR24} adapted to our setting.

We start with some preparatory lemmas. 
\begin{lemma} \label{lemma,LengthUniformConverge}
Let $g$ be a Riemannian metric in $R^-(M)$ and let $\gamma: \R \to M$  be a geodesic on $M$ that is 1-periodic, i.e. $\gamma(t) = \gamma(t+1)$ for all $t \in \R$. We denote the restriction of $\gamma$ to the interval $[0,1]$ by $\gamma_1:[0,1]\to M$ and denote $L_g(\gamma):=L_g(\gamma_1)$, the $g$-length of the closed geodesic $\gamma_1$. Consider the space of closed Lipschitz curves on $(M,g)$,
$$\mathcal{L}_c= \{\alpha: [0,1] \to (M,g) \mid \alpha \text{ is Lipschitz and } \alpha(0)=\alpha(1)  \}.$$ Define for some $c> L_g(\gamma)$ the subset 
  $$  \Lip_c([\gamma_1]) :=   \{\alpha \in [\gamma_1] \mid  \alpha\in \mathcal{L}_c, \;  \; \Lip(\alpha) \le c \}
  $$  
  and consider the length functional $L_g$ restricting to $\Lip_c([\gamma_1])$.
  Then for all $\varepsilon > 0$, there exists $\delta >0$ such that for all $\alpha \in \Lip_c([\gamma_1])$
  with 
  $$
  L_g(\alpha) \le  L_g(\gamma) + \delta,
  $$
there exists  $t_0 \in [0,1)$ and a monotone continuous and surjective map $\tau: [0,1] \to [t_0, t_0 +1]$ 
  such that
  $$
  d_g( \alpha(t), \gamma (\tau(t)) \le \varepsilon
  $$
 for all $t \in [0,1]$.  
\end{lemma}
\begin{proof}
Suppose the lemma does not hold.  Then
there exists $ \varepsilon >0$ and a sequence of closed curves $\alpha_{n}:[0,1]\to M$  in $\Lip_c([\gamma_1])$ such that
$$
 L_g(\alpha_n) \le  L_g(\gamma) + \frac{1}{n} \; \text{and} \; d_g( \alpha_n(t), \gamma (\tau(t))) \ge \varepsilon
$$
for some $t \in [0,1]$ and  all monotone continuous and surjective maps $\tau: [0,1] \to [t_0, t_0 +1]$, where $t_0 \in [0,1)$.

By Arzelà-Ascoli Theorem, there exists a subsequence $\alpha_{n_j}:[0,1]\to (M,g)$ that uniformly converges to a closed Lipschitz curve $\alpha$.  Hence
    $$
    \lim\limits_{j\to \infty} \max_{t\in[0,1]} d_g(\alpha_{n_j}(t),\alpha(t))=0.
    $$
    Therefore $\alpha\in [\gamma_1]$. Recall that the length $L_g(\cdot)$ is a lower semicontinuous function of continuous curves (\cite[Chapter I.1, Proposition 1.20 (7)]{BH99}),
    $$L_g(\gamma)\geq \liminf_{j\to \infty} L_g(\alpha_{n_j})\geq L_g(\alpha).$$
Since $\gamma_1$ is up to parametrization the unique shortest closed curve in the homotopy class $[\gamma_1]$, it holds that $L_g(\gamma)=L_g(\alpha)$. Therefore,
there exists for some $t_0 \in [0,1) $ a monotone continuous and surjective map $\tau: [0,1] \to [t_0, t_0 +1]$ such that $ \alpha (t) = \gamma (\tau(t))$.
But then
$$
0= \lim\limits_{j\to \infty} \max_{t\in[0,1]} d_g(\alpha_{n_j}(t),\alpha(t))= \lim\limits_{j\to \infty} \max_{t\in[0,1]} d_g(\alpha_{n_j}(t),\gamma (\tau(t)))
$$
yields a contradiction.
\end{proof}
\begin{lemma} \label{lemma,self-intersection-surfaces}
Let $g$ be a Riemannian metric in $R^-(M)$ and let $\gamma: \R \to M$  be a geodesic that is 1-periodic on $M$ with $\gamma(t) = \gamma(t+1)$ for all $t \in \R$ and has one self-intersection at $p = \gamma(0) = \gamma(s_0)$ for some $s_0 \in (0, 1)$. 
Denote by $\theta_p \in (0 , \frac{\pi}{2}]$ the angle of self-intersection of $\gamma $ at $p$ and therefore $a_p =\cos(\theta_p) \in [0, 1) $. For $\varepsilon >0$ smaller than $1/2$ of the injectivity radius $R_M$ of $M$, let 
$ \alpha: \R \to M$ be a closed continuous curve with $ \alpha(t) = \alpha(t+1)$
freely homotopic to $\gamma$ and let $\tau:\R \to \R$ be a surjective momotone and continuous map with  $\tau(t+1) = \tau(t) +1$  such that 
$d_g( \alpha(t) , \gamma(\tau(t))) < \varepsilon$. Then $\alpha$ has a self-intersection at $q = \alpha (t_0) =  \alpha (t_1)$ such that $0 < t_1 -t_0 <1$ and
\begin{equation} \label{inequality,self-intersection-surfaces}
d_g(p, q) < \varepsilon \bigg(\sqrt{\frac{2}{1-a_p}} +1\bigg)
\end{equation}
Furthermore, the loops $\gamma: [0, s_0] \to M$ and  $\alpha: [t_0, t_1] \to M$ are free homotopic.
Moreover, if $q' = \alpha (t_0') =  \alpha (t_1')$ is any  self-intersection such that the loop 
$\alpha:[t_0', t_1'] \to M$ is free homotopic to  $\gamma: [0, s_0] \to M$ then the inequality \ref{inequality,self-intersection-surfaces} holds for $q'$.
\end{lemma}
\begin{proof}
Since   $\varepsilon $ is smaller than $1/2$ of the injectivity radius $R_M$, there are lifts  $\widetilde \gamma , \widetilde \alpha : \R \to \widetilde M$ of curves $\gamma: \R  \to M$ and $\alpha: \R  \to M$ respectively such
that  $d_g( \widetilde\alpha(t), \widetilde \gamma(\tau(t)) < \varepsilon$. Since $\widetilde \gamma $ is a geodesic and the projection of $\widetilde p = \widetilde \gamma (0) $  is given by the self-intersection $p$,  there exists $\eta \in \Gamma$, the group of covering transformations of $\widetilde{M}$, such that
$\eta \widetilde \gamma \cap \widetilde \gamma = \eta\widetilde p = \widetilde \gamma(s_0)$. Let $U_\varepsilon (\widetilde \gamma)$ and  $U_\varepsilon ( \eta \widetilde \gamma)$ be the $\varepsilon$-tubular neighborhoods of $\widetilde \gamma$ and  $\eta \widetilde \gamma$. Then $\widetilde\alpha(t) \in U_\varepsilon (\widetilde \gamma)$
and $\eta \widetilde\alpha(t) \in U_\varepsilon (\eta \widetilde \gamma)$ for all $t \in \R$. Therefore the curves
 $\widetilde\alpha$ and  $\eta \widetilde \alpha$ intersects in the set $U_\varepsilon (\widetilde \gamma) \cap U_\varepsilon ( \eta \widetilde \gamma)$ and there exists $ t_0 < t_1$ with $t_1-t_0 <1$ and $\eta \widetilde \alpha(t_0) =\widetilde \alpha(t_1) \in U_\varepsilon (\widetilde \gamma) \cap U_\varepsilon ( \eta \widetilde \gamma)$. Since $\eta \widetilde \gamma (0) = \widetilde \gamma(s_0)$, we obtain that the loops $\gamma_{[0,s_0]}$
 and $\alpha_{[t_0, t_1]}$ are free homotopic.  For any $ x \in  U_\varepsilon (\widetilde \gamma) \cap U_\varepsilon ( \eta \widetilde \gamma)$, denote by $x_1$ (resp.  $x_2$) the orthogonal projections of $x$ onto the geodesics $\widetilde \gamma$ (resp.  $\eta \widetilde \gamma$).    Then
 $$
  d_g(x_1, x_2) \le d_g(x_1, x) + d_g(x, x_2) \le 2 \varepsilon
 $$ 
 Consider the triangle given by $x_1, x_2$ and  $x_0 = \widetilde \gamma (s_0) $ and define $a_i = d_g(x_i, x_0)$ for 
 $i \in \{1,2 \}$.  Since the intersection angle of $\eta \widetilde \gamma$ and 
$\widetilde \gamma$ is $\theta_p$, and since the curvature of the surface is nonpositive,
we obtain by triangle comparison and the laws of cosine (see, for example, \cite[Lecture I.B]{BGS85})
\begin{align*}
4 \varepsilon^2 &\ge d_g(x_1, x_2)^2 \ge a_1^2 + a_2^2 - 2 \cos{\theta_p}a_1 a_2 \\
&= (a_1-a_2)^2 + 2 a_1 a_2 - 2 a_p a_1a_2\\
&= (a_1-a_2)^2 + 2 a_1 a_2(1-a_p) \\
&\ge  2 a_1 a_2(1-a_p).
\end{align*}
Assume without loss of generality that $a_1 \le a_2$. Then
$$
4 \varepsilon^2 \ge 2 a_1^2 ((1-a_p)
$$
and therefore
$$
a_1 \le \varepsilon \frac{\sqrt{2} }{\sqrt{1-a_p}}.
$$
This implies for all $ x \in  U_\varepsilon (\widetilde \gamma) \cap U_\varepsilon ( \eta \widetilde \gamma)$,
$$
d_g(x, x_0) \le d_g(x,x_1)+d_g(x_0,x_1) \le a_1 + \varepsilon \le \varepsilon \bigg(\sqrt{\frac{2}{1-a_p}} +1\bigg).
$$
Let 
 $x=\widetilde \alpha(t_1) \in  U_\varepsilon (\widetilde \gamma) \cap U_\varepsilon ( \eta \widetilde \gamma)$ and let $q$ be the projection of $x$. Since the projection of $x_0$ is $p$
and since $d_g(p,q)\leq d_g(x,x_0)$, we obtain Inequality  \eqref{inequality,self-intersection-surfaces}. 
If the loop 
$\alpha:[t_0', t_1'] \to M$ is free homotopic to  $\gamma: [0, s_0] \to M$, then  $\eta \widetilde \alpha(t_0') = \widetilde \alpha(t_1')$. 
Therefore $\eta \widetilde\alpha(t_0') = \widetilde\alpha(t_1')  \in U_\varepsilon (\widetilde \gamma) \cap U_\varepsilon ( \eta \widetilde \gamma)$. Hence, the same argument gives that Inequality (\ref{inequality,self-intersection-surfaces}) holds for the projection point $q'$ of  $\widetilde \alpha(t_1')$.
\end{proof}

We are now able to discuss an example of $S(g_1,g_2)< L(g_1,g_2)$.  We first summarize some results from \cite{GR24}. Given sufficiently small $0<\varepsilon_1< \varepsilon_2$, Gogolev and Reber in \cite{GR24} constructed by perturbation method a pair of Riemannian metrics $g_1,g_2\in R^-(M)$ on a closed surface $M=S$ of genus $\mathcal{G}\geq 2$ with the following properties:
\begin{enumerate}
    \item There exists a closed $g_2$-geodesic $\gamma=\gamma^{g_2}$ with exactly one self-intersection $p = \gamma (0) =\gamma(s_0)$  such that  $\gamma=\gamma^{g_1}$ is also a $g_1$-geodesic up to parametrization.
    \item The two loops $\gamma_1,\gamma_2$ of $\gamma$ separated by $p$ have the following properties,
    $$L_{g_1}(\gamma_1)= L_{g_2} (\gamma_1)- \varepsilon_1, $$
    and 
    $$L_{g_1}(\gamma_2)= L_{g_2} (\gamma_2)+ \varepsilon_2. $$
    \item There exists some $\varepsilon>0$ so that for any $\alpha \in [\Gamma]=[\pi_1 S]$, we have 
    $$\frac{\ell_{g_2} (\alpha)}{\ell_{g_1} (\alpha)}\leq \frac{1}{1+\varepsilon}.$$
     In particular, $S(g_1,g_2) \leq \frac{1}{1+\varepsilon}< 1$.
\end{enumerate}

In the paper \cite{GR24}, the authors focus on the non-existence of shrinking diffeomorphism for such a pair of metrics. For us, since we care about Lipschitz maps homotopic to identity, we provide an argument here of non-existence of Lipschitz maps $f\in \Lip_{\id}(M,g_1,g_2)$ with $\Lip(f,g_1,g_2) < 1$.
\begin{proposition}
    Assume that the positive numbers $\varepsilon_1,\varepsilon_2$ and $\varepsilon_2-\varepsilon_1$ are small enough. Then for metrics $g_1,g_2$ introduced above, there does not exist a Lipschitz map $f\in  \Lip_{\id}(M,g_1,g_2)$ so that $\Lip(f,g_1,g_2) < 1$.  As a consequence, $L(g_1,g_2)\geq 1$.
\end{proposition}

\begin{proof}
    We prove the proposition by contradiction. Suppose there exists $f_0\in  \Lip_{\id}(M,g_1,g_2)$ so that $\Lip(f_0,g_1,g_2) < 1$. Since $f_0\circ \gamma$ is a closed curve homotopic to $\gamma$, we have
  \begin{align*}
      L_{g_2} (\gamma)&\leq L_{g_2}(f_0\circ \gamma)\\
       &\leq \Lip(f_0,g_1,g_2) L_{g_1}(\gamma)\\
      &=\Lip(f_0,g_1,g_2) (L_{g_2}(\gamma)+ \varepsilon_2-\varepsilon_1) < L_{g_2}(\gamma)+ \varepsilon_2-\varepsilon_1.
      \end{align*}
      Therefore we obtain $0\leq L_{g_2}(f_0\circ \gamma) - L_{g_2}(\gamma) < \varepsilon_2 -\varepsilon_1$.
For any $\epsilon >0 $,   Lemma \ref{lemma,LengthUniformConverge} implies the existence
of $ t_0 \in [0,1) $ together with a monotone continuous and surjective map $\tau: [0,1] \to [t_0, t_0 +1]$ 
  such that
  $$
  d_g( f_0 \circ \gamma(t), \gamma (\tau(t))) \le \varepsilon
  $$
 for all $t \in [0,1]$ when $\delta= \varepsilon_2 -\varepsilon_1 $ is sufficiently small.
Since $p =\gamma(0) = \gamma(s_0)$ is the unique self-intersection
 of $\gamma$, $f_0(p) = f_0(\gamma(0)) = f_0(\gamma(s_0))$ is a self-intersection of $f_0\circ \gamma$.
 Since $f_0$ is homotopic to the identity, the loops $\gamma_{[0,s_0]}$ and $f_0 \circ \gamma_{[0,s_0]}$  are free homotopic. Extend $\tau:[0,1]\to [t_0,t_0+1]$ to $\tau:\mathbb{R} \to \mathbb{R}$ by letting $\tau(t+1)=\tau(t)+1$.
Using Lemma \ref{lemma,self-intersection-surfaces}, we obtain $d_{g_2}(p,f_0(p))< \frac{\varepsilon_1}{2}$ provided 
      $\varepsilon \bigg(\sqrt{\frac{2}{1-a_p}} +1 \bigg) \le \frac{\varepsilon_1}{2}$, where $a_p >0$ is the cosine of the angle of self-intersection of $\gamma$.
 Now let $\beta$ be the minimal geodesic connecting $p$ to $f_0(p)$ and $\eta_1=f_0\circ \gamma_1$. Then since we assume $\Lip(f_0,g_1,g_2)  <1$, we obtain
      $$L_{g_2}(\beta \eta_1 \beta^{-1}) \leq \varepsilon_1 + L_{g_2}(f_0 \circ \gamma_1) < \varepsilon_1 + L_{g_1}(\gamma_1)= L_{g_2}(\gamma_1).$$
      Therefore 
      $$L_{g_2}(\gamma_2 \beta\eta_1\beta^{-1})< L_{g_2}(\gamma_1)+ L_{g_2}(\gamma_2) = L_{g_2}(\gamma).$$
      Since the closed curve $\gamma_2  \beta\eta_1\beta^{-1}$ is homotopic to $\gamma$, the estimate above leads to a contradiction of the fact that $\gamma=\gamma^{g_2}$ is the shortest closed curve in its homotopy. Therefore, we conclude that there does not exist a Lipschitz map $f\in  \Lip_{\id}(M,g_1,g_2)$ so that $\Lip(f,g_1,g_2) < 1$.
\end{proof}

\begin{corollary}
     Assume that the positive numbers $\varepsilon_1,\varepsilon_2$ and $\varepsilon_2-\varepsilon_1$ are small enough. Then for metrics $g_1,g_2$ introduced as before, we have
     $$S(g_1,g_2) < L(g_1,g_2).$$
\end{corollary}

\section{Examples for which the stretch locus is a simple closed geodesic}
\label{appendix, C}
\hfill\\
Below, we provide examples of compact surfaces which do not satisfy the condition of Theorem  \ref{GK, geodesicLamination} but where the stretch locus $E(g_1,g_2)$ is a lamination given by a simple closed geodesic.
It is an interesting question to understand what a necessary and sufficient condition is for the stretch locus $E(g_1,g_2)$ to be a geodesic lamination (see Question \ref{question: stretchLocusLamination} in Appendix \ref{question: open question}).

Our example is given by perturbing metrics near on a simple closed geodesic. To give the construction, we first state a lemma. 
\begin{lemma} \label{lemma, fermiCoordinates}
    Let $M=S$ be a closed connected oriented surface of genus $\mathcal{G}\geq 2$. Let $g_0\in R^-(S)$ and let $\gamma$ be a simple closed $g_0$-geodesic on $S$. There exists a small tubular neighborhood $U$ of $\gamma$ and a closed set $V\subseteq U$ containing $\gamma$ such that for $0<s<1$, we can find a smooth bump function $\kappa_s: S\to \mathbb{R}$ satisfying,
\begin{enumerate}
    \item $\kappa_s|_{\gamma}\equiv (1-s)^2$.
     \item $\kappa_s|_{V}<1$.
    \item $\kappa_s|_{U^c}\equiv 1$.
    \item $g_s= \kappa_s g_0$ defines a new metric and the closed curve $\gamma$ is a $g_s$-geodesic up to reparametrization.
\end{enumerate}
\end{lemma}
\begin{proof}
    Let $t\mapsto \gamma(t)$ be the $g_0$ arc-length parametrization of $\gamma$. Consider Fermi coordinate for $g_0$ based at $\gamma$. A point $p$ has Fermi coordinate 
 $(\rho,t)$ if $p$ is of distance $\rho$  to  $\gamma$ (i.e. to the orthogonal projection point on $\gamma$, say $p'$), and $p'=\gamma(t)$. The metric tensor then is 
    $$g_0|_p=d\rho^2+ f(\rho,t)^2 dt^2.$$
    Since $\gamma$ is a $g_0$-geodesic, we have $f(0,t)\equiv 1$ and $\partial_{\rho}f(\rho,t)\big|_{\rho=0}\equiv 0$.

    Let $\varepsilon$ be small so that the boundary of $\varepsilon$-tubular neighborhood $U$ of $\gamma$ are also simple closed curves. Let $0<s<1$ and $m_s(\rho)$ be a smooth function in $\overline{U}$ given by 
     \begin{equation*}
     m_s(\rho)=
\begin{cases}
    1-s,  \hspace{1cm}  \text{when } \rho=0, \\
     1-s \cdot\exp({1-\frac{1}{1-(\frac{\rho}{\varepsilon})^2}}) , \hspace{1cm}   \text{when } 0<|\rho| <\varepsilon, \\
   1,  \hspace{1cm}  \text{when } |\rho| = \varepsilon. 
  \end{cases}
\end{equation*}

For a fixed $s$, the function $m_s(\rho)$ is smooth and monotone in the interval $[0, \varepsilon]$ (resp. the interval $[-\varepsilon,0]$).   Let $\kappa_s: S\to \mathbb{R}$ be given by $\kappa_s\equiv 1$ in the complement of $U$ and $\kappa_s(\rho,t)=m_s(\rho)^2 \leq 1$ in $\overline{U}$. Consider $g_s= \kappa_s g_0$ . The metric tensor for $g_s$ in the same Fermi coordinate is $$g_s|_p= m_s(\rho)^2 d\rho^2+ (m_s(\rho)f(\rho,t))^2 dt^2.$$
The requirement that $\gamma$ is still a reparametrized geodesic for $g_s$ is 
$$\nabla^{g_s}_{\partial_t} \dot{\gamma}(t)\big|_{\rho=0}=h_s(t) \partial_t,$$
for some function $h_s(t)$. One checks from Christoffel symbol computation that this requirement is fulfilled when $\partial_{\rho}m_s(\rho)|_{\rho=0}=0$. This is satisfied with the above chosen $m_s(\rho)$. In this case $h_s(t)\equiv0$ and $t'=(1-s)t$ is the arc-length parameter for $g_s$. Moreover, letting the closed set $V\subset U$ be the closure of $\frac{\varepsilon}{2}$-tubular neighborhood concludes all statements needed. 

\end{proof}

We are now able to discuss the main example in this appendix.

\begin{example}
     Let $M=S$ be a closed connected oriented surface of genus $\mathcal{G}\geq 2$. Consider any hyperbolic metric $g_2$ on $S$. Suppose $\gamma_0$ is a simple closed $g_2$-geodesic on $S$. Let $s>0$ be small and let $g_1$ be a small perturbation of $g_2$ described as in Lemma \ref{lemma, fermiCoordinates}.  Then $g_1\in R^-(S)$ and 
    $$S(g_1,g_2)= L(g_1,g_2) =\frac{1}{1-s}>1,$$
    and  $E(g_1,g_2)=E_{\id}(g_1,g_2)=\{\gamma_0\}$ is a geodesic lamination. However, we have for any $p\in\gamma_0$, the sectional curvatures satisfy
        \begin{equation} \label{eqtn, CurvatureInequality}
       K_{g_1}(p) < K_{g_2}(p) L(g_1,g_2)^2.
         \end{equation} 
\end{example}
\begin{proof}
Since $g_s$ in Lemma \ref{lemma, fermiCoordinates} varies smoothly with respect to $s$, so is its sectional curvature $K_{g_s}$. Because $g_2=g_0$ is a hyperbolic metric. For $s$ small enough, we can ensure that $g_1=g_s$ is negatively curved.

By Corollary \ref{corollary, GeodesicStretchBusemannFunction}, we know 
$$
I_m(g_1,g_2) = \int_{S^{g_1}M} g_2(B^{g_2}(\pi (v), v^{g_1}_+), v) dm.
$$    

Since 
$$g_2(B^{g_2}(\pi (v), v^{g_1}_+), v) = \norm{v}_{g_2} \cos \theta_v, $$
\noindent where $v\in S^{g_1}M$ and $\theta_v$ is the angle formed from the vector $B^{g_2}(\pi (v), v^{g_1}_+)$ to $v$ counterclockwisely.  Since $g_1=k_s g_2$,
$$\norm{v}_{g_2}= \frac{1}{\sqrt{k_s(\pi(v))}} \cdot \norm{v}_{g_1} \leq \frac{1}{1-s}, $$
and $$ \cos \theta_v \leq 1.$$

Both equalites are realised if and only if  $v$ is tangent to $\gamma_0$ which is a geodesic for both $g_1$ and $g_2$. Therefore, 
$$S(g_1,g_2)=I_{\delta_{\gamma_0}}(g_1,g_2)=\frac{1}{1-s},$$
and 
$\mathcal{M}(g_1,g_2)=\{v\in S^{g_1}M \text{ }| \text{ } v  \text{ is tangent to } \gamma_0\}$.
On the other hand  (see, for example \cite[Proposition 5.2, (ii), (iv)]{DU22}),
$$L(g_1,g_2) \leq \Lip (\id,g_1,g_2)= \max_{v\in S^{g_1}M} \frac{\norm{v}_{g_2}}{\norm{v}_{g_1}}=\frac{1}{1-s}.$$
Again, the equality is realised when  $v$ is tangent to $\gamma_0$. We obtain $L(g_1,g_2)=S(g_1,g_2)$ and by  Proposition \ref{proposition,MatherStretchLocus},
$$E(g_1,g_2) \subset E_{\id}= \{\gamma_0 \}= \pi(\mathcal{M}(g_1,g_2))\subset E(g_1,g_2),$$
which yields $E(g_1,g_2)=\{\gamma_0 \}$ is a geodesic lamination. Moreover, the identity map is an optimal Lipschitz map and it maximally stretches $\gamma_0$ by $ L(g_1,g_2)$.  

 Next we want to verify Inequality \eqref{eqtn, CurvatureInequality}. We notice that the sectional curvature for metric of the form 
$g=A(u,v)^2du^2+B(u,v)^2 dv^2$
is
$$K_g(p)=\frac{-1}{AB} ({\partial_u}(A^{-1} \partial_{u}{B})+{\partial_v}(B^{-1} \partial_{v}{A})).$$
As discussed in Lemma \ref{lemma, fermiCoordinates}, we can take the Fermi coordinate based at $\gamma_0$ for $g_2$. The curvature for the metric
$g_2|_p=d\rho^2+ f(\rho,t)^2 dt^2$
is,
$$K_{g_2}(p)=\frac{- \partial^2_{\rho}f(\rho,t)}{f(\rho,t)}=-1.$$
In fact, we have $f(\rho,t)=\cosh \rho$.

Consider $p=\gamma_0(t)$ with Fermi coordinate $(0,t)$. Using the conditions  $\partial_{\rho}m_s(\rho)|_{\rho=0}=0$ together with $f(0,t)\equiv 1$ and $\partial_{\rho}f(\rho,t)\big|_{\rho=0}\equiv 0$, we obtain that the sectional curvature for metric $g_1=m_s(\rho)^2 d\rho^2+ m_s(\rho)^2f(\rho,t)^2 dt^2$ at $p=(0,t)$ is 
$$K_{g_1}(p)=\frac{- 1}{ m_s(0)^2} [\partial^2_{\rho}f(0,t) + \partial^2_{\rho} m_s(0)  m_s(0)^{-1} ].$$

Recall in Lemma \ref{lemma, fermiCoordinates}, the function $m_s(\rho) = 1-s \cdot\exp({1-\frac{1}{1-(\frac{\rho}{\varepsilon})^2}})  $ when  $ 0<|\rho| < \varepsilon $ and so $\partial^2_{\rho} m_s(0)>0$ . Also $0<m_s(0)=1-s<1$. Therefore,
$$\frac{K_{g_1}(p)}{K_{g_2}(p)}=\frac{\partial^2_{\rho}f(0,t) + \partial^2_{\rho} m_s(0) (1-s)^{-1} }{(1-s)^2 \partial^2_{\rho}f(0,t)}> \frac{1}{(1-s)^2}=L(g_1,g_2)>1.$$

Since $f_0=\id$ is an optimal Lipschitz map, this is also an example of $E(g_1,g_2)$ being a maximally stretched geodesic lamination while conditions of Corollary \ref{Cor, geodesicLamination} do not hold.
\end{proof}

\section{A simple analytic lemma}\label{Appendix simpleLemma}
Since we can not find a reference for the following lemma, we include a proof of it here.

\begin{lemma}\label{simpleLemma}
    Let $K$ be a compact topological space and $V$ an arbitrary metric space.
    Given a continuous function $F \colon K \times V \to \R$, then the function
    $G:V \to \R$ given by 
    $$
    G(p)= \max_{x \in K}F(x,p)
    $$
    is continuous as well.
\end{lemma}
\begin{proof}
Let $p_n$ be a sequence in $V$ converging to $p$.
We show first that
\begin{equation} \label{eqtn,LimitMax}
\lim_{n \to \infty} \max_{x \in K}|F(x,p_n) -F(x,p)| = 0.
\end{equation}
If not, there exists $\varepsilon >0 $ and a subsequence $p_{n_k}$ and a sequence $x_k \in K$
such that
$$
|F(x_k,p_{n_k}) -F(x_k,p)| \geq \varepsilon.
$$
By choosing a subsequence, we can assume that $x_k$ converges to $x_0 \in K$.
Then $(x_k,p_{n_k})$ converges to $(x_0, p)$. This leads to a contradiction with continuity of $F$.

Now assume
$$
\max_{x \in K}F(x,p) = G(p) = F(x_p,p)
$$
for some $x_p \in K$. Choose a sequence $x_n \in K$ converging to $x_p$.
Then $G(p_n) \ge F(x_n, p_n)$. Continuity of $F$ implies
$$
\varliminf_{n \to \infty}G(p_n) \ge \lim_{n \to \infty} F(x_n, p_n) =F(x_p, p) = G(p).
$$
On the other hand, if we let $y_n \in K$ be a sequence such that $F(y_n,p_n) =G(p_n)$. Using what we have shown in Equation \eqref{eqtn,LimitMax}, we obtain
$$
\lim_{n \to \infty}(F(y_n,p_n) -F(y_n,p)) =0.
$$
This yields
$$
\varlimsup_{n \to \infty}G(p_n)=\varlimsup_{n \to \infty} F(y_n,p_n) =\varlimsup_{n \to \infty} F(y_n,p) \le G(p).
$$
Hence, we obtain the continuity of G.
\end{proof} 

\section{Some open questions}  \label{question: open question}

We collect a few open questions which are already suggested in the note.

\begin{question} \label{question: Finsler}
    Consider the Riemannian metrics $R^-_1(M)$ of negative curvature and topological entropy $1$ on $M$. Does 
    \begin{equation*}\label{eqtn Kn95}
        d_T(g_1,g_2)=\log S(g_1,g_2)
    \end{equation*}
    descend to an asymmetric Finsler metric on the isometry classes of $R^-_1(M)$?
\end{question}
It has been shown in \cite{Kn95} (see also \cite[Proposition 5.4]{GKL22} and Theorem \ref{thm, proportMLS}) that $d_T$ defines a metric on the isometry classes of $R^-_1(M)$ if and only if marked length spectrum rigidity holds.
As proved in Section 5 of \cite{GKL22}, $d_T$ induces an asymmetric Finsler norm on the tangent space of $R^-_1(M)$, regardless of whether marked length spectrum rigidity holds.
Furthermore, the associated Finsler metric $d_F$ dominates $d_T$, that is, $d_F \geq d_T$.
If $d_F$ agrees with $d_T$, then $d_T$ is a metric, and global marked length spectrum rigidity would follow; see also Conjecture 5.7 in \cite{GKL22}.
So far, by the work of Thurston \cite{Th98}, the equality $d_F=d_T$ is known only for Teichm\"uller spaces.
A first step toward addressing the question of whether $d_F=d_T$ would be to study it for closed surfaces of variable negative curvature and normalized entropy, where marked length spectrum rigidity is known.

\begin{question} 
   Given $g_1,g_2\in R^-(M)$, what are necessary and sufficient conditions for $S(g_1,g_2)=L(g_1,g_2)$? 
\end{question}
 It is always true that $S(g_1,g_2)\leq L(g_1,g_2)$. There exist many examples, including the Teichmüller space, for which $S(g_1,g_2)=L(g_1,g_2)$. However, as already mentioned in the previous Appendix \ref{ExampleS<L}, there exist also examples for which $S(g_1,g_2)<L(g_1,g_2)$. It remains an interesting question to find geometric and analytic conditions for negatively curved metrics such that the equality holds.

\begin{question} \label{question: stretchLocusLamination}
   Given $g_1,g_2\in R^-(M)$, what are necessary and sufficient conditions for the stretch locus $E(g_1,g_2)$ and the projection of the Mather set $\pi(\mathcal{M}(g_1,g_2))$ to define geodesic laminations?
\end{question}
The work in \cite{GK17} provides at least an open neighborhood of examples in $R^-(M)\times R^-(M)$ near the Teichmüller space for which the stretch locus and the projection of the Mather set are geodesic laminations. However, as suggested by Appendix~\ref{appendix, C}, one should also expect examples beyond those guaranteed by the techniques of \cite{GK17}.  It would be interesting to understand a general reason for the appearance of geodesic laminations beyond manifolds of constant negative curvature. 

Regarding to Theorem \ref{mainThm3}, we have the following questions.

\begin{question}\label{question: WeightedLipMap}
     Given $g_1,g_2\in R^-(M),$ 
\begin{enumerate}
    \item 
   does there always exist a maximally stretched measure $m_0$ so that 
    $$S(g_1,g_2)=L_{m_0}(g_1,g_2) ?$$
    \item 
    If such a measure $m_0$ exist, does there  exist a $m_0$-weighted best Lipschitz map $f_0$ so that 
    $$L_{m_0}(g_1,g_2)=\int_{S^{g_1}M} ||Df_0(v)||_{g_2}dm_0(v)?$$
    \end{enumerate}
\end{question}

We recall that the $m$-weighted least Lipschitz constants always satisfy $S(g_1,g_2)\leq L_m(g_1,g_2)\leq L(g_1,g_2)$ regardless of the choice of a maximally stretched measure $m$. It is a more flexible alternative to the best Lipschitz constants in this general setting. We would like to know whether $m$-weighted least Lipschitz constants can coincide with the maximal stretch for good choices of maximally stretched measures.

The last two questions are related to Aubry and Mather sets.

\begin{question} \label{question,AubryPeierl}
   Given $g_1,g_2\in R^-(M)$, does the Aubry set $\mathcal{A}(g_1,g_2)$ agree with the Peierls barrier zero set $\mathcal{H}_0(g_1,g_2)$?
\end{question}

 We already know that $\mathcal{H}_0(g_1,g_2)\subset \mathcal{A}(g_1,g_2)$ is true. However, we do not know whether these two sets are equal. In an analogous setting in \cite{Fat08} for Lagrangian dynamics, these two sets are equal.

\begin{question}
     Given $g_1,g_2\in R^-(M),$ what are the relation between the stretch locus $E(g_1,g_2)$ and the Aubry set $\mathcal{A}(g_1,g_2)$? Furthermore, what happens if we impose the additional condition $S(g_1,g_2)=L(g_1,g_2)$?
\end{question}
This question is motivated by the fact that the Aubry set always contains the Mather set and the stretch locus contains the projection of the Mather set when $S(g_1,g_2)=L(g_1,g_2)$ (see Proposition \ref{proposition,MatherStretchLocus}). Both the stretch locus and the Aubry set are in some sense the non-measure-theoretic counterpart of the Mather set. Therefore, it seems natural to ask whether the projection of the Aubry set and the stretch locus are equal, or if the projection of the Aubry set is contained in the stretch locus or vice versa when $S(g_1,g_2)=L(g_1,g_2)$.

\bibliography{stretch}
\bibliographystyle{alpha}

\end{document}